\documentclass{article}
\usepackage[a4paper,width=150mm,top=25mm,bottom=25mm]{geometry}
\usepackage{amsthm}
\usepackage{thmtools}
\usepackage{mathtools,bbm,amsmath,amsfonts,amssymb,color,epic}
\usepackage{multirow,comment}
\usepackage{hyperref}
\usepackage{cleveref}
\usepackage{siunitx}
\usepackage{cite}
\usepackage[export]{adjustbox}
\usepackage{tikz,pgfplots}
\usepackage{algorithmicx,algorithm,pifont}
\usepackage{algpseudocode}
\usepackage{needspace}
\usepackage{enumitem}
\pgfplotsset{compat=newest}
\usetikzlibrary{arrows,decorations,backgrounds,positioning}
\usepgfplotslibrary{statistics,fillbetween}
\theoremstyle{plain}
\newtheorem{theorem}{Theorem}[section]
\newtheorem{lemma}[theorem]{Lemma}

\newtheorem{corollary}[theorem]{Corollary}
\newtheorem{definition}[theorem]{Definition}
\newtheorem{remark}[theorem]{Remark}

\newtheorem{example}[theorem]{Example}
\def\letters{a,b,c,d,e,f,g,h,i,j,k,l,m,n,o,p,q,r,s,t,u,v,w,x,y,z,%
	A,B,C,D,E,F,G,H,I,J,K,L,M,N,O,P,Q,R,S,T,U,V,W,X,Y,Z}
\makeatletter
\@for \@l:=\letters \do{%
  \expandafter\edef\csname\@l bb\endcsname{%
  \noexpand\ensuremath{\noexpand\mathbb{\@l}}}%
  \expandafter\edef\csname\@l bf\endcsname{{\noexpand\bs \@l}}%
  \expandafter\edef\csname\@l cal\endcsname{\noexpand\ensuremath{%
  \noexpand\mathcal{\@l}}}%
  \expandafter\edef\csname\@l eu\endcsname{\noexpand\ensuremath{%
  \noexpand\EuScript{\@l}}}%
  \expandafter\edef\csname\@l frak\endcsname{\noexpand\ensuremath{%
  \noexpand\mathfrak{\@l}}}%
  \expandafter\edef\csname\@l rm\endcsname{{\noexpand\rm \@l}}%
  \expandafter\edef\csname\@l scr\endcsname{\noexpand\ensuremath{%
  \noexpand\mathscr{\@l}}}%
}
\@for \@l:=\letters \do{%
	\expandafter\edef\csname bb\@l\endcsname{\noexpand\ensuremath{%
  \noexpand\mathbb{\@l}}}%
	\expandafter\edef\csname bf\@l\endcsname{{\noexpand\bs \@l}}%
	\expandafter\edef\csname cal\@l\endcsname{\noexpand\ensuremath{%
  \noexpand\mathcal{\@l}}}%
	\expandafter\edef\csname eu\@l\endcsname{\noexpand\ensuremath{%
  \noexpand\EuScript{\@l}}}%
	\expandafter\edef\csname frak\@l\endcsname{\noexpand\ensuremath{%
  \noexpand\mathfrak{\@l}}}%
	\expandafter\edef\csname rm\@l\endcsname{{\noexpand\rm \@l}}%
	\expandafter\edef\csname scr\@l\endcsname{\noexpand\ensuremath{%
  \noexpand\mathscr{\@l}}}%
}
\newcommand{\bs}[1]{{\boldsymbol#1}}
\renewcommand{\d}{\operatorname{d\!}}
\newcommand{\Cor}{\operatorname{Cor}}
\newcommand{\gfrac}[2]{\genfrac{}{}{0pt}{}{#1}{#2}}
\newcommand{\isdef}{\mathrel{\mathrel{\mathop:}=}}

\newcommand{\dist}{\operatorname{dist}}
\newcommand{\diam}{\operatorname{diam}}
\newcommand{\cbin}{{c_{\textnormal{b}}}}
\newcommand{\cdiam}{c_{\textnormal{diam}}}
\newcommand{\cunif}{{c_{\textnormal{uni}}}}

\newcommand{\children}{\operatorname{children}}

\DeclareMathOperator*{\esssup}{\operatorname{ess\,sup}}
\DeclareMathOperator*{\argmax}{\operatorname{argmax}}
\DeclareMathOperator*{\argmin}{\operatorname{argmin}}
\title{Data-intrinsic approximation in metric spaces}
\author{J.~D\"olz\thanks{Corresponding author}~\thanks{Institute for Numerical Simulation, University of Bonn, Germany. \href{mailto:doelz@ins.uni-bonn.de}{doelz@ins.uni-bonn.de}}~ and M.~Multerer\thanks{Dalle Molle Institute for Artificial Intelligence, USI Lugano, Switzerland. \href{mailto:michael.multerer@usi.ch}{michael.multerer@usi.ch}}}
\date{}

\begin{document}
\maketitle
\begin{abstract} %
Analysis and processing of data is a vital part of our modern society and
requires vast amounts of computational resources. To reduce the computational
burden, compressing and approximating data has become a central topic. We
consider the approximation of labeled data samples, mathematically described as
site-to-value maps between finite metric spaces. Within this setting, we
identify the discrete modulus of continuity as an effective data-intrinsic
quantity to measure regularity of site-to-value maps without imposing further
structural assumptions. We investigate the consistency of the discrete modulus
of continuity in the infinite data limit and propose an algorithm for its
efficient computation. Building on these results, we present a sample based
approximation theory for labeled data. For data subject to statistical
uncertainty we consider multilevel approximation spaces and a variant of the
multilevel Monte Carlo method to compute statistical quantities of interest. Our
considerations connect approximation theory for labeled data in metric spaces to
the covering problem for (random) balls on the one hand and the efficient
evaluation of the discrete modulus of continuity to combinatorial optimization
on the other hand. We provide extensive numerical studies to illustrate the
feasibility of the approach and to validate our theoretical results.
\end{abstract}



\section{Introduction}
\subsection{Motivation}
The analysis and processing of social network data, text data, audio files,
photos, and videos, as well as scientific data such as measurements and
simulations, have become vital to our modern society. It is estimated that in
the year 2023 the world created, captured, copied, and consumed around 123
zettabytes, that is, $123\cdot10^{21}$ bytes of data. This number is expected to
rise to almost 400 zettabytes by 2028, compare \cite{statista_bigdata}.
Processing and storing these data requires immense computational resources.
Therefore, the compression and approximation of data is imperative to mitigate
the computational burden. In this article, we consider \emph{labeled data}, such
as time series, images or videos, where a label or value is assigned to each
instance of the data. To mathematically formalize such data, we introduce
\emph{site-to-value maps} defined on a finite set $X_N$ of \(N\)
\emph{data sites} taking \emph{data values} in a set $Y_N$. Hence, approximating
labeled data amounts to the approximation of the site-to-value maps.

The approximation of site-to-value maps is understood as inferring simpler
representations of site-to-value maps without having to necessarily process all
data instances. Computational efficiency and error estimates may then be derived
from additional knowledge. The latter is closely tied to the model selection
problem, which makes the pivotal assumption that the site-to-value map is the
discretization of some unknown function that belongs to a certain class to
derive approximation rates. The applied approximation scheme is then dependent
on this presumed class. Typically, the approximation error then depends on the
particular choice of the function class, introducing potential bias and
uncertainty in the approximation. In a data-centric world, where data themselves
are the most relevant source of information, making additional assumptions about
their behavior for the purpose of approximation may be unwarranted. It is more
natural to develop an approximation theory and algorithms that only operate on
the available data. In the present article, we propose such an approximation
theory.

\subsection{Literature review}
There exists an abundance of model-centric approximation schemes for
site-to-value, especially in the Euclidean setting. The underlying principle is
to represent the site-to-value map by a dictionary of functions, such as
polynomials, splines or wavelets, possibly discarding coefficients with small or
vanishing effect on the overall representation, see \cite{DL1993,deB01,%
DVDD1998,Vet2001,Mal99} and the references therein. For scattered data sites,
approximation in reproducing kernel Hilbert spaces is a standard approach, see,
for example, \cite{Wen05}. Under an additional Gaussian process model
assumption, even uncertainty quantification is available, see \cite{RW06}. These
approaches may be combined with subsampling, such as the Nystr\"om
approximation, see \cite{WS2001}, or greedy strategies, see \cite{SWH2024}, to
achieve computational efficiency. For the non-Euclidean setting, Euclidean
embeddings of preimage and image spaces can be used to approximate the
induced map with techniques for Euclidean spaces. Such embeddings can be
realized in a classical analytical fashion or using kernel-induced feature maps,
auto encoders and other dimensionality reducing techniques, see \cite{BGG2024}
and the references therein. In this regard, neural networks have gained
particular interest, see, for example, \cite{GBC2016,HU2018,LJP+2021}.

For general data sites, data-centric approximation schemes become relevant.
Examples of such approaches are random forests, see \cite{Bre01}, Haar-wavelets
on trees, see \cite{GNC10}, or diffusion wavelets on graphs, see \cite{CM06}.
More recently, samplets, which form a multiresolution analysis of localized,
discrete signed measures, have been introduced, see \cite{HM25} for a survey
and \cite{EGMQ25} for a construction on graphs.

Data-centric, wavelet like bases particularly allow to measure smoothness
of data. In \cite{GNC10}, H\"older spaces of discrete data have been introduced.
This notion has been generalized by means of microlocal spaces, see
\cite{Jaf91}, on discrete data in \cite{AGM25}. The idea of approximating moduli
of smoothness, see, for example, \cite{DT87}, is considered in \cite{Dit87} for
certain univariate point distributions, while dyadic meshes grids are considered
in \cite{DP14}. In these approaches, smoothness is measured in terms of
available derivatives, or more generally, in terms of bounds on H\"older
exponents. Even so, in certain applications, this may be a too crude measure for
the particular data instance. Therefore, we pursue an approach based on the
\emph{modulus of continuity} and its discretization, respectively. The modulus
of continuity is attributed to Lebesgue, see \cite{Leb1909,Ste2006}, and
provides regularity information for any continuous function defined between
metric spaces, without imposing additional assumptions. 

\subsection{The approach taken}
The first straightforward observation we make and exploit is that site-to-value
maps are trivially continuous functions, if we equip the domain $X_N$ and the
range $Y_N$ with the structure of compact metric spaces. The second observation
is that every continuous function on a compact metric space exhibits a (optimal)
modulus of continuity. As a rule of thumb, the faster the decay of the modulus
of continuity towards zero, the more regular the function. As a consequence, the
modulus of continuity has become a valuable tool in approximation theory and
analysis, see \cite{BL2000,DL1993,Hei01,Jac1930}.

For the specific setting of site-to-value maps \(f_N\colon X_N\to Y_N\),
the (discrete) modulus of continuity reads
\[
\omega(Y_N,t)
\isdef
\max_{\gfrac{\bfx,\bfx'\in X_N:}{d_{X_N}(\bfx,\bfx')\leq t}}
d_{Y_N}\big(f_N(\bfx),f_N(\bfx')\big),
\]
and amounts to a measure of regularity of site-to-value maps, see
\Cref{fig:modcont} for two illustrating examples. 
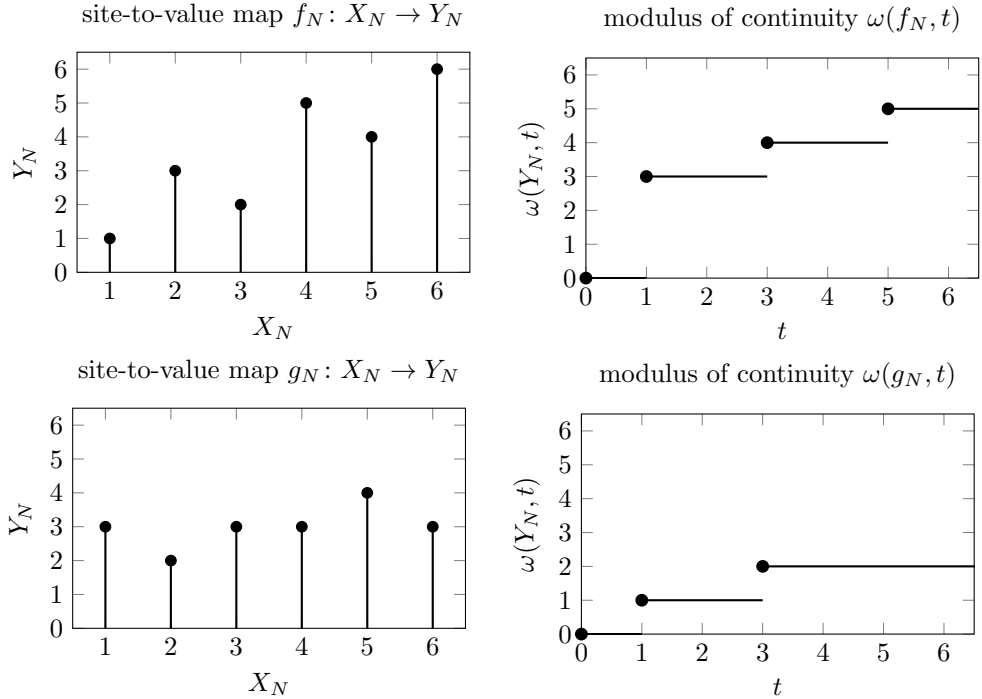
\begin{figure}
\centering
\begin{tikzpicture}
\begin{axis}[
	title={site-to-value map $f_N\colon X_N\to Y_N$},
	xlabel={$X_N$},
	ylabel={$Y_N$},
	xtick={1,2,3,4,5,6},
	ytick={0,1,2,3,4,5,6},
	xmin=0.5,
	xmax=6.5,
	ymin=0,
	ymax=6.5,
	width=0.45\textwidth,
	height=0.3\textwidth
	]
\addplot[only marks,mark=*] 
  coordinates {(1,1) (2,3) (3,2) (4,5) (5,4) (6,6)};
\addplot[thick] coordinates {(1,1) (1,0)};
\addplot[thick] coordinates {(2,3) (2,0)};
\addplot[thick] coordinates {(3,2) (3,0)};
\addplot[thick] coordinates {(4,5) (4,0)};
\addplot[thick] coordinates {(5,4) (5,0)};
\addplot[thick] coordinates {(6,6) (6,0)};
\end{axis}
\end{tikzpicture}
\quad
\begin{tikzpicture}
\begin{axis}[
	title={modulus of continuity $\omega(f_N,t)$},
	xlabel={$t$},
	ylabel={$\omega(Y_N,t)$},
	xtick={0,1,2,3,4,5,6},
	ytick={0,1,2,3,4,5,6},
	xmin=0,
	xmax=6.5,
	ymin=0,
	ymax=6.5,
	width=0.45\textwidth,
	height=0.3\textwidth
	]
	\addplot+[jump mark left,black,mark=*,mark options={fill=black},thick]
    coordinates {(0,0) (1,3) (3,4) (5,5) (7,5)};
\end{axis}
\end{tikzpicture}
\begin{tikzpicture}
\begin{axis}[
	title={site-to-value map $g_N\colon X_N\to Y_N$},
	xlabel={$X_N$},
	ylabel={$Y_N$},
	xtick={1,2,3,4,5,6},
	ytick={0,1,2,3,4,5,6},
	xmin=0.5,
	xmax=6.5,
	ymin=0,
	ymax=6.5,
	width=0.45\textwidth,
	height=0.3\textwidth
	]
\addplot[only marks,mark=*] coordinates 
  {(1,3) (2,2) (3,3) (4,3) (5,4) (6,3)};
\addplot[thick] coordinates {(1,3) (1,0)};
\addplot[thick] coordinates {(2,2) (2,0)};
\addplot[thick] coordinates {(3,3) (3,0)};
\addplot[thick] coordinates {(4,3) (4,0)};
\addplot[thick] coordinates {(5,4) (5,0)};
\addplot[thick] coordinates {(6,3) (6,0)};
\end{axis}
\end{tikzpicture}
\quad
\begin{tikzpicture}
\begin{axis}[
	title={modulus of continuity $\omega(g_N,t)$},
	xlabel={$t$},
	ylabel={$\omega(Y_N,t)$},
	xtick={0,1,2,3,4,5,6},
	ytick={0,1,2,3,4,5,6},
	xmin=0,
	xmax=6.5,
	ymin=0,
	ymax=6.5,
	width=0.45\textwidth,
	height=0.3\textwidth
	]
\addplot+[jump mark left,black,mark=*,mark options={fill=black},thick]
  coordinates {(0,0) (1,1) (3,2) (7,2)};
\end{axis}
\end{tikzpicture}
\caption{\label{fig:modcont}Illustration of two continuous functions
$f_N,g_N\colon\{1,\ldots,6\}\to\bbN$ (left) and their corresponding
discrete, optimal, right-continuous moduli of continuity as a measure of their
smoothness (right).}
\end{figure}
As a maximum over a set of finite cardinality, the discrete modulus of
continuity is, in principle, computationally accessible from the data. This
leads us to the following questions.
\begin{enumerate}[label=(Q\arabic*)]
\item\label{item:Q1} When is the modulus of continuity of site-to-value maps
  consistent in the infinite data limit?
\item\label{item:Q2} How can we efficiently compute or estimate the modulus of
  continuity from data?
\item\label{item:Q3} How can we exploit the data-intrinsic regularity
  provided by the modulus of continuity to construct efficient approximation
  schemes for site-to-value maps?
\end{enumerate}
By providing answers and follow up questions to (Q1) to (Q3), the aim of this
article is to derive data-intrinsic approximation schemes and rates which only
exploit the data-intrinsic regularity provided by the discrete modulus of
continuity $\omega(Y_N,t)$.

\subsection{Contributions}
Guided by the questions \ref{item:Q1} to \ref{item:Q3}, the contributions of
this article to the approximation of site-to-value maps $f_N\colon X_N\to Y_N$
between finite metric spaces are as follows.
\begin{enumerate}[label=(C\arabic*)]
\item For both, deterministically chosen and empirical data sites, we provide
point-wise and $L^p$ approximation rates and consistency analyses of the modulus
of continuity in the infinite data limit. To this end, we connect the question
of approximation to the (random) set cover problem of (random) balls and
to covering numbers. This addresses \ref{item:Q1}.
\item We provide an algorithm for the efficient approximation of the modulus of
continuity. To do this, we establish connections to nearest neighbor searches
and combinatorial integer optimization. This addresses \ref{item:Q2}.
\item We establish an interpolation-based approximation theory of piecewise
constant functions for deterministic and random site-to-value maps. To achieve
the latter, we construct multilevel approximation spaces on labeled data and
provide a natural extension of the multilevel Monte Carlo method to labeled
data in metric spaces. This addresses \ref{item:Q3}.
\end{enumerate}
We complement these contributions with non-trivial numerical examples,
illustrating the feasibility of the approach and the validity of our theoretical
results.

\subsection{Outline}
The remainder of this article is structured as follows. In
Section~\ref{sec:Preliminaries}, we introduce the discrete modulus of continuity
and transfer well known results from literature to the discrete setting.
Section~\ref{sec:detcons} studies the consistency of the discrete
modulus of continuity for deterministic data sites. The consistency of the
discrete modulus of continuity for empirical data sites is addressed in
Section~\ref{sec:Consistency} by assuming that data sites are independent
samples of a given probability distribution. In Section~\ref{sec:Computation},
we address the efficient computation of the discrete modulus of continuity,
avoiding the naive quadratic cost in terms of the number of data sites.
Section~\ref{sec:Approximation} derives approximation results for the piecewise
constant approximation of site-to-value maps. As a relevant application of the
present framework, we consider the multilevel Monte Carlo method for the
computation of first and second order statistics of high-dimensional random
vectors in Section~\ref{sec:MLMC}. Extensive numerical results are presented in
Section~\ref{sec:Numerics} and concluding remarks are stated in
Section~\ref{sec:Conclusion}.

\section{Moduli of continuity}\label{sec:Preliminaries}
Let $(\calX,d_\calX)$ and $(\calY,d_\calY)$ be metric spaces and assume that
$\calX$ be compact with diameter 
\[T_\Xcal\isdef\diam(\Xcal)<\infty.
\]
Given a set of \emph{data points}
\begin{equation}\label{eq:datapoints}
\{(\bfx_1,\bfy_1),\ldots,(\bfx_N,\bfy_N)\}\subset\calX\times\calY
\end{equation}
comprised of \emph{data sites} $X_N=\{\bfx_1,\ldots,\bfx_N\}$ and
\emph{data values} $Y_N=\{\bfy_1,\ldots,\bfy_N\}$, we aim at deriving a notion
of smoothness for the \emph{site-to-value map}
\begin{equation}\label{eq:sampleFunc}
f_N\colon X_N\to Y_N,
\qquad
\bfx_i\mapsto\bfy_i.
\end{equation}
We introduce the following quantities.
\begin{definition}
The \emph{fill distance} is defined as
	\begin{equation}\label{eq:FillDistance}
		h_{X_N}
		\isdef\sup_{\bfx\in\Xcal}\min_{\bfx_i\in X}
		d_{\Xcal}(\bfx,\bfx_i),
	\end{equation}
	whereas the \emph{separation distance} is given by
	\begin{equation}\label{eq:SepDistance}
		q_{X_N}\isdef\min_{i\neq j}d_{i,j}
	\end{equation}
with the \emph{distance matrix}
\(
{\bs D}\isdef [d_{i,j}]_{i,j=1}^N, d_{i,j}\isdef d_\calX(\bfx_i,\bfx_j).
\)

The set \(X_N\) is \emph{quasi-uniform}, if there exists a constant \(c>0\)
such that 
\begin{equation}\label{eq:quasi-unif}
\frac 1 c q_{X_N}\leq h_{X_N}\leq c q_{X_N}.
\end{equation}
\end{definition}

To define a \emph{discrete modulus of continuity}, we first recall the classical
definition of the \emph{modulus of continuity}, see, for example,
\cite{DL1993,BL2000}.

\begin{definition}\label{def:modcont}
For a function
\(f\colon\Xcal\to\Ycal\) the \emph{modulus of continuity}
\(\omega(f,\cdot)\colon[0,\infty]\to[0,\infty]\) is defined as
\[
\omega(f,t)\isdef\sup_{\gfrac{\bfx,\bfx'\in\Xcal:}
{d_{\Xcal}(\bfx,\bfx')\leq t}}d_{\Ycal}\big(f(\bfx),f(\bfx')\big).
\]

We say that the modulus of continuity is \emph{continuous} if it is continuous
with respect to the canonical topology of
$[0,\infty]=\bbR_{\geq 0}\cup\{\infty\}$.

For a monotonically increasing function $\rho\colon[0,\infty]\to[0,\infty]$
we further introduce the semi-norm
\[
|f|_{\rho}\isdef\sup_{t>0}\rho(t)^{-1}\omega(f,t)
\]
and say that $f$ is of class $\rho$ if $|f|_{\rho}<\infty$.
\end{definition}
\begin{remark}
\begin{enumerate}
\item The most natural classes $\rho$ are the ones generated by functions which
are moduli of continuity themselves.
\item The modulus of continuity of a function $f$ of class $\rho$ satisfies
\[
\omega(f,t)\leq\rho(t)|f|_{\rho}.
\]
\item The seminorm in the previous definition allows us to introduce function
classes in terms of the modulus of continuity. For $\rho(t)=t^{\alpha}$ we
obtain the well known $|\cdot|_{\operatorname{Lip}(\alpha)}$-seminorm.
\end{enumerate}
\end{remark}

We collect the basic properties of the modulus of continuity in the following
lemma, see, for example, \cite{DL1993}.

\begin{lemma}[Properties of the modulus of continuity]\label{lem:ModCont}
Let \(f\colon\Xcal\to\Ycal\) be uniformly continuous. Then,
the modulus of continuity
\begin{enumerate}
\item is continuous at \(0\) and \(\omega(f,t)\to \omega(f,0)=0\) for
\(t\to 0\).
\item is non-negative, continuous and non-decreasing for \(t>0\).
\end{enumerate}
\end{lemma}

\Needspace{2\baselineskip}
\begin{remark}{~}\label{rem:MOC}
\begin{enumerate}
\item Since \(\Xcal\) is compact, we note that any continuous function
\(f\colon\Xcal\to\Ycal\) is also \emph{uniformly continuous}, that is,
for each \(\varepsilon>0\) there exists \(\delta>0\) such that for all
\(\bfx,\bfx'\in\Xcal\) with \(d_{\Xcal}(\bfx,\bfx')<\delta\) there holds
\(d_{\Ycal}\big(f(\bfx),f(\bfx')\big)<\varepsilon\).

\item In the literature, the modulus of continuity from \Cref{def:modcont} is
often referred to as the \emph{optimal} modulus of continuity. 

\item Vice versa, any function which satisfies the assertions of
\Cref{lem:ModCont} and is a majorant to the (optimal) modulus of continuity
is considered as a modulus of continuity. Throughout this article we do not
rely on this distinction and always refer to the optimal modulus of continuity
from \Cref{def:modcont}, whenever we refer to the modulus of continuity.
\end{enumerate}
\end{remark}

Every site-to-value map \(f_N\colon X_N\to Y_N\subset\Ycal\),
see \eqref{eq:sampleFunc}, is uniformly continuous, which is immediate by
choosing \(0<\delta < q_X\), compare the first item of \Cref{rem:MOC}.
Moreover, \((X_N,d_{\Xcal})\) is a metric space by restricting the metric of
\(\Xcal\) to \(X_N\). This stipulates the definition of the \emph{discrete
modulus of continuity} for \(f\colon\Xcal\to\Ycal\) according to
\begin{equation}\label{eq:discmodcont}
\omega_N(f,t)=\max_{\gfrac{\bfx_i,\bfx_j\in X_N:}{d_{i,j}\leq t}}
d_\calY\big(f(\bfx_i),f(\bfx_j)\big).
\end{equation}

\begin{remark}{~}
\begin{enumerate}
\item The properties from Lemma~\ref{lem:ModCont} directly carry over to
\(\omega_N(f,t)\).
\item The definition \eqref{eq:discmodcont} does not require the knowledge of
a function $f$ defined on all of \(\Xcal\).
\item Vice versa, one can always find a function
\(f\colon\Xcal\to\Ycal\) with \(f({\bfx}_i)={\bfy}_i\) for \(i=1,\ldots,N\).
Therefore, it is possible to define the discrete modulus of continuity
exclusively using the set of data points. 
To emphasize this fact, we may also write
\begin{equation}\label{eq:discretemodulusdata}
\omega_N(Y_N,t)\isdef\max_{\gfrac{\bfx_i,\bfx_j\in X_N:}{d_{i,j}\leq t}}
d_\calY(\bfy_i,\bfy_j).
\end{equation}
\item We will later adopt the perspective that $\calX$ is a finite set of data
sites and $X_N$ a subset thereof. In this case the continuous modulus of
continuity from \Cref{def:modcont} and the discrete modulus of continuity from
\Cref{eq:discretemodulusdata} coincide on $\calX$.
\end{enumerate}
\end{remark}

We have the following result for the discrete modulus of continuity with respect
to the set of data sites.

\begin{lemma}[Monotonicity of discrete modulus of continuity]%
\label{lem:monotonicity}
Let \(X_N\subset X_{N'}\) for \(N'\geq N\). Then there holds
\[
\omega_N(f,t)\leq \omega_{N'}(f,t)\leq \omega(f,t),\quad t\geq 0,
\]
for any function \(f\colon\Xcal\to\Ycal\).
\end{lemma}

\begin{proof} Without loss of generality, let
\(X_{N'}=X_{N}\cup\{{\bs x}_{N+1},\ldots,{\bs x}_{N'}\}\).
We introduce the index sets
\[
I_{N,t}\isdef\big\{(i,j)\in\{1,\ldots N\}^2: d_{i,j}\leq t\big\}
\]
and observe
\(I_{N,t}\subset I_{N',t}\) for any \(N'\geq N\) and any \(t\geq 0\). 
Analogously, we define the sets
\begin{equation}\label{eq:FuncDist}
F_{N,t}\isdef\big\{d_{\Ycal}\big(f(\bfx_i),
f(\bfx_j)\big):(i,j)\in I_{N,t}\big\},
\end{equation}
which satisfy \(F_{N,t}\subset F_{N',t}\).
Therefore, we infer
\[
\omega_N(f,t)=\max F_{N,t}\leq \max F_{N',t}= \omega_{N'}(f,t)
\quad\text{for any }t\geq 0.
\]
The second bound follows in complete analogy.
\end{proof}

The discrete modulus of continuity gives rise to a discrete version of the
\(|\cdot|_{\rho}\)-seminorm. We define
\begin{equation}\label{eq:discLipNorm}
|f|_{\rho,N}
\isdef\max_{i,j\in\{1,\ldots,N\}: i\neq j}\rho(d_{i,j})^{-1}\omega_N(f,d_{i,j}).
\end{equation}
Again, in complete analogy, if only a site-to-value map \(f_N\colon X_N\to Y_N\)
is known, we set
\begin{equation}\label{eq:dataLipNorm}
|f_N|_{\rho,N}
\isdef\max_{i,j\in\{1,\ldots,N\}: i\neq j}
\rho(d_{i,j})^{-1}\omega_N(Y_N,d_{i,j}).
\end{equation}
\begin{lemma}[Monotonicity of the discrete seminorm]
Let \(X_N\subset X_{N'}\) for \(N'\geq N\). Then there holds
\[
|f|_{\rho,N}\leq |f|_{\rho,N'}\leq|f|_{\rho}
\]
for any function \(f\colon\Xcal\to\Ycal\).
\end{lemma}
\begin{proof}
The assertion is an immediate consequence of \Cref{lem:monotonicity}.
\end{proof}

\section{Consistency for deterministic data sites}\label{sec:detcons}
In this section, we derive consistency estimates for the discrete modulus of
continuity for deterministic data sites that meet specific properties. More
precisely, we assume that the site-to-value map $f_N\colon X_N\to Y_N$ is the
restriction of some function $f\colon\calX\to\calY$, that is, $f_N=f|_{X_N}$.
We derive approximation estimates of the discrete modulus of continuity 
$\omega_N(Y_N,\cdot)=\omega_N(f,\cdot)$ to the modulus of continuity
$\omega(f,\cdot)$ and discuss the consistency behavior for $N\to\infty$. To this
end, we require no further assumptions on $\calX$ despite of it being a compact
metric space. For example, $\calX$ may be a finite but larger set than \(X_N\)
or a metric space of infinite cardinality in which the data sites are embedded.

\subsection{Covering numbers}\label{sec:coveringnumbers}
\begin{definition}\label{def:CoveringNumbers}
An \emph{$\varepsilon$-net} for $(\calX,d_{\calX})$
is a subset \(\Kcal\subset\Xcal\) such that for each
\({\bs x}\in\Xcal\) there exists \({\bs x}'\in\Kcal\)
with 
\({\bs x}\in B_{\varepsilon}({\bs x}')
\isdef\{{\bs x}\in\Xcal:d_\Xcal(\bfx,\bfx')\leq\varepsilon\}\).
The smallest cardinality of an $\varepsilon$-net
for \(\Xcal\) is the 
\emph{covering number} \(\calN(\calX,\varepsilon)\),
that is,
\begin{equation}\label{eq:CoveringNumber}
	\calN(\calX,\varepsilon)\isdef\min\{|\Kcal|:
	\Kcal\text{ is an $\varepsilon$-net for }(\calX,d_{\calX})\}.
\end{equation}
\end{definition}

\begin{remark}Sometimes, it is assumed that an \(\varepsilon\)-net
\(\Kcal\subset\Xcal\)
is also an \emph{\(\varepsilon/2\)-packing} in the sense that
\({\bs x}\not\in B_{\varepsilon/2}({\bs x}')\) for all
\({\bs x},{\bs x}'\in\Kcal\)
with \({\bs x}\neq{\bs x}'\), see, for example, \cite{Cla2006,KL2004,BKL2006}.
In this case, the covering number can be shown to be equivalent to the
\emph{packing number}
\[
\Pcal(\Xcal,\varepsilon)\isdef\max\{|\Kcal|:
	\Kcal\text{ is an $\varepsilon$-packing for }(\calX,d_{\calX})\}
\]
meaning that
\(
\Pcal(\Xcal,\varepsilon)\leq\calN(\calX,\varepsilon)
\leq\Pcal(\Xcal,\varepsilon/2),
\)
see \cite[Theorem 4]{Kol1956}.
\end{remark}
It follows directly from Definition~\ref{def:CoveringNumbers} that
\begin{equation}\label{eq:coveringnumbermonotonic}
\calN(\calX,\varepsilon)
\geq
\calN(\calX,\varepsilon')
\end{equation}
for all $0\leq\varepsilon<\varepsilon'$.

While the covering number is rarely accessible in closed form, 
upper and lower bounds to it have been the subject of several
investigations in the literature. 
In the following, we give three intuitive examples.
\begin{example}\label{ex:Assouad}
\begin{enumerate}
\item If $\calX\subset\bbR^d$ is equipped with the Euclidean metric, 
then it is a standard result that
\[
\frac{\lambda(\calX)}{\lambda\big(B_r(0)\big)}
\leq
\calN(\calX,r)
\leq
3^d\frac{\lambda(\calX)}{\lambda\big(B_r(0)\big)},
\]
where $\lambda$ is the Lebesgue measure on $\bbR^d$.
\item More generally, a metric space $\calX$ is \emph{Ahlfors $n$-regular},
if there exists a Borel regular measure $\mu$ and constants $C\geq 1$,
$n>0$ such that $C^{-1}R^n\leq\mu(B_R(\bfx))\leq CR^n$ for all $\bfx\in\calX$,
$R>0$, see \cite{Hei01}. By the same arguments as for $\bbR^d$, the covering
number of an Ahlfors $n$-regular space satisfies
\[
c^{-1}r^{-n}
\leq
\calN(\calX,r)
\leq
cr^{-n}
\]
for all $r>0$ and some constant $c>0$. Especially, the Euclidean space
$\bbR^d$ is Ahlfors $d$-regular.
\item The metric space \((\Xcal,d_\Xcal)\) is \emph{\((C,s)\)-homogeneous}, if 
\[
\Ncal\big(B_R({\bs x}),r\big)\leq C\bigg(\frac R r\bigg)^s
\]
for all \({\bs x}\in\Xcal\) and all \(0<r<R<\infty\). 
The infimum of the set of all \(s\) such that 
\((\Xcal,d_\Xcal)\) is \((C,s)\)-homogeneous is 
the \emph{Assouad dimension} of \((\Xcal,d_\Xcal)\), see \cite{Assouad}.
For a metric space with Assouad dimension \(s\) there holds
\begin{align}\label{eq:assouadcover}
\Ncal(\Xcal,r)\leq cr^{-s},
\end{align}
with a constant \(c>0\) depending on \(\Xcal\). In particular, an
Ahlfors $n$-regular metric space has Assouad dimension $n$. 
Finally, we remark that the Assouad dimension of a finite metric space
is $0$.
\end{enumerate}
\end{example}

An important property of $\varepsilon$-nets is that their covering property is
preserved in the image of a given metric.
\begin{lemma}[Metric maps $\varepsilon/2$-nets to $\varepsilon$-nets]%
\label{lem:metricmapscovering}
Let $(\calX,d)$ be a compact metric space, $\varepsilon>0$, and
$X_N=\{\bfx_1,\ldots,\bfx_N\}$ an $\varepsilon/2$-net. Let 
$Z_N\subset\calZ\subset[0,\infty]$ be defined as
\begin{align}\label{eq:Z}
Z_N=\{d_\calX(\bfx,\bfx')\colon\bfx,\bfx'\in X_N\},\qquad
\calZ=\{d_\calX(\bfx,\bfx')\colon\bfx,\bfx'\in \calX\}.
\end{align}
Then $Z_N$ is an $\varepsilon$-net for $\calZ$, that is, there holds
\[
\bigcup_{\bfz\in Z_N}B_\varepsilon(\bfz)=\calZ.
\]
\end{lemma}

\begin{proof}
We show that for all $t\in Z$ there is $s\in Z_N$ such that
$|t-s|\leq \varepsilon$. Let $t\in\calZ$. Hence, there exist
$\bfx,\tilde{\bfx}\in\calX$ such that $t=d_\calX(\bfx,\bfx')$. Moreover,
since $X_N$ is a $\varepsilon/2$-net for $\calX$, there further exist
$\bfx_i,\bfx_j\in X_N$ such that
\[
d_\calX(\bfx,\bfx_i)\leq\frac{\varepsilon}{2},
\qquad
d_\calX(\bfx',\bfx_j)\leq\frac{\varepsilon}{2}.
\]
Thus, we have $s=d_\calX(\bfx_i,\bfx_j)\in Z_N$ and it remains to bound $|t-s|$.

First let \(t\geq s\). By the triangle inequality, we derive 
\begin{align*}
t-s
&=
d_\calX(\bfx,\bfx')-d_\calX(\bfx_i,\bfx_j)\\
&\leq
d_\calX(\bfx,\bfx_i)+d_\calX(\bfx_i,\bfx_j)
+d_\calX(\bfx_j,\bfx')-d_\calX(\bfx_i,\bfx_j)\\
&=
d_\calX(\bfx,\bfx_i)+d_\calX(\bfx',\bfx_j)\\
&\leq
\varepsilon.
\end{align*}
Now, let \(t<s\). Analogously, we infer
\begin{align*}
s-t&=d_\calX(\bfx_i,\bfx_j)-d_\calX(\bfx,\bfx')\\
&\leq d_\calX(\bfx_i,\bfx)+d_\calX(\bfx,\bfx')
+d_\calX(\bfx',\bfx_j)-d_\calX(\bfx,\bfx')\\
&=d_\calX(\bfx,\bfx_i)+d_\calX(\bfx',\bfx_j)\\
&\leq
\varepsilon.
\end{align*}
Combining the two estimates proves the assertion.
\end{proof}

An immediate consequence of this lemma is the following consistency result for
continuous moduli of continuity.
\begin{corollary}[Consistency for continuous moduli of continuity]%
\label{cor:continuousconsistency}
Let $f\colon\calX\to\calY$ have a continuous modulus of continuity. Moreover, let
$\{X_N\}_{N\in\bbN}$ be a sequence of $r_N/2$-nets with $r_N\to 0$. Then
\[
\omega(f,t)-\omega_N(f,t)\to 0
\qquad
\text{as}~N\to\infty.
\]
The convergence is uniform in $t$.
\end{corollary}
\begin{proof}
Convergence for fixed $t$ follows directly from \Cref{lem:metricmapscovering}.
Moreover, the function $f$ having a modulus
of continuity implies continuity of $f$. Since $\Xcal$ is compact, 
the function $f$ is even uniformly continuous, which, in turn, implies the
uniform convergence in $t$.
\end{proof}

We note that even a uniformly continuous function on a compact metric space is
not required to have a continuous modulus of continuity if $\calZ$ in
\Cref{lem:metricmapscovering} is a non-convex set. In fact, non-constant
site-to-value maps have always discontinuous moduli of continuity, such that a
more refined analysis is required. In the following we provide such an analysis.

\subsection{Pointwise approximation}\label{sec:pwdetconv}
The following theorem bounds the approximation error of the discrete modulus of
continuity compared to the continuous one at a given point $t>0$. For $t=0$,
there always holds $\omega(f,0)=\omega_N(f,0)=0$, see also \Cref{lem:ModCont}.
\begin{theorem}[Covering diameter $r<t$ implies error bound at $t$]%
\label{thm:coveringtobound}
Let $f\colon\calX\to\calY$ be continuous. Let $t>0$ be fixed and let
$X_N=\{\bfx_1,\ldots,\bfx_N\}$, $N\geq\calN(\calX,r/2)$, be an $r/2$-net
for $0<r\leq t$, that is,
\begin{equation}\label{eq:detcoverXr}
	\bigcup_{n=1}^NB_{r/2}(\bfx_n)=\calX.
\end{equation}
Then, there holds
\[
0
\leq
\omega(f,t)-\omega_N(f,t)
\leq
\omega(f,t)-\omega(f,t-r)+2\omega(f,r).
\]
\end{theorem}
\begin{proof}
From \Cref{lem:monotonicity}, we have
\begin{equation}\label{eq:addsub}
0
\leq
\omega(f,t)-\omega_N(f,t)
=
\omega(f,t)-\omega(f,t-r)
+\underbrace{\omega(f,t-r)-\omega_N(f,t)}_{=(\spadesuit)}.
\end{equation}
To prove the assertion, we need to estimate $(\spadesuit)$. The definition of
the modulus of continuity implies
\begin{align*}
(\spadesuit)
=
\sup_{\substack{\bfx,\bfx'\in\calX\\d_\calX(\bfx,\bfx')\leq t-r}}
d_\calY\big(f(\bfx),f(\bfx')\big)
-
\max_{\substack{\tilde{\bfx}_N,\tilde{\bfx}_N'\in X_N\\
d_\calX(\tilde{\bfx}_N,\tilde{\bfx}_N')\leq t}}d_\calY\big(f(\tilde{\bfx}_N),
f(\tilde{\bfx}_N')\big).
\end{align*}
To allow for a comparison between supremum and maximum, observe that
\Cref{eq:detcoverXr} implies that for every $\bfx,\bfx'\in\calX$ there exist
$\bfx_N,\bfx_N'\in X_N$ such that
\begin{equation}\label{eq:eps3}
d_\calX(\bfx,\bfx_N)\leq\frac{r}{2},
\qquad
d_\calX(\bfx',\bfx_N')\leq\frac{r}{2}.
\end{equation}
As a consequence, for given 
$\bfx,\bfx'\in\calX$ with $d_\calX(\bfx,\bfx')\leq t-r$, the problem
\begin{equation}\label{eq:nearestmodneighbor}
(\bfx_N,\bfx_N')
\in
\argmin_{
	\substack{d_\calX(\bfx,\tilde\bfx_N)\leq\frac{r}{2}\\
    d_\calX(\bfx',\tilde\bfx_N'))\leq\frac{r}{2}\\
		d_\calX(\tilde\bfx_N,\tilde\bfx_N')\leq t
	}
}\big(d_\calX(\bfx,\tilde\bfx_N)+d_\calX(\bfx',\tilde\bfx_N')\big)
\end{equation}
always exhibits a solution. The situation is illustrated in \Cref{fig:balls}.
To keep the notation light we suppress the dependence of $(\bfx_N,\bfx_N')$ on
$(\bfx,\bfx')$.
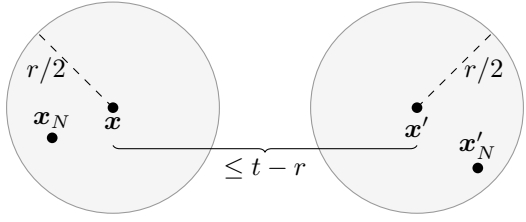
\begin{figure}[htb]
\begin{center}
\begin{tikzpicture}[scale=1]
\def\dx{4}      %
\def\dy{0}        %
\def\t{4}       %
\def\r{1.4}       %
\def\yoffset{0.2} %
\coordinate (x) at (0,0);
\coordinate (x') at (\dx,\dy);
\draw[fill=black!10,opacity=0.4] (x) circle (\r);
\draw[fill=black!10,opacity=0.4] (x') circle (\r);
\fill (x) circle (2pt) node[below,black] {${\bs x}$};
\fill (x') circle (2pt) node[below,black] {${\bs x}'$};
\draw[black,dashed] (x) -- +(-0.707*\r,0.707*\r)node[midway,left,black]{$r/2$};
\draw[black,dashed] (x') -- +(0.707*\r,0.707*\r)node[midway,right,black]{$r/2$};
\draw[decoration={brace,mirror,raise=0.5cm},decorate] (x) -- (x')
node[midway,below=0.55cm,black] {$\leq t-r$};

\coordinate (y) at (\r*0.5-1.5,-2*\yoffset);
\coordinate (y') at (1.5 + \dx - \r*0.5, -4*\yoffset);
\fill (y) circle (2pt) node[above,black] {${\bs x}_N$};
\fill (y') circle (2pt) node[above,black] {${\bs x}_N'$};
\end{tikzpicture}
\end{center}
\caption{\label{fig:balls}Visualization of the situation in
\eqref{eq:nearestmodneighbor}.}
\end{figure}

Now, by observing that the triangle inequality implies the bound
\[
d_\calY(f(\bfx),f(\bfx'))
-
d_\calY(f(\bfx),f(\bfx_N))
-
d_\calY(f(\bfx'),f(\bfx_N'))
\leq
d_\calY(f(\bfx_N),f(\bfx_N'))
\]
and that there holds
\[
d_\calY\big(f(\bfx_N),f(\bfx_N')\big)
\leq
\max_{\substack{\tilde{\bfx}_N,\tilde{\bfx}_N'\in X_N\\
d_\calX(\tilde{\bfx}_N,\tilde{\bfx}_N')\leq t}}d_\calY
\big(f(\tilde{\bfx}_N),f(\tilde{\bfx}_N')\big),
\]
we obtain that
\begin{align*}
(\spadesuit)
&=
\sup_{\substack{\bfx,\bfx'\in\calX\\d_\calX(\bfx,\bfx')\leq t-r}}
d_\calY\big(f(\bfx),f(\bfx')\big)
-
\max_{\substack{\tilde{\bfx}_N,\tilde{\bfx}_N'\in X_N\\
d_\calX(\tilde{\bfx}_N,\tilde{\bfx}_N')\leq t}}
d_\calY\big(f(\tilde{\bfx}_N),f(\tilde{\bfx}_N')\big)\\
&\leq
\sup_{\substack{\bfx,\bfx'\in\calX\\d_\calX(\bfx,\bfx')\leq t-r}}
\bigg(d_\calY\big(f(\bfx),f(\bfx')\big)
-
d_\calY\big(f(\bfx_N),f(\bfx_N')\big)\bigg)\\
&\leq
\sup_{\substack{\bfx,\bfx'\in\calX\\d_\calX(\bfx,\bfx')\leq t-r}}
\bigg(
d_\calY\big(f(\bfx),f(\bfx')\big)\\
&\phantom{\leq\sup_{\substack{\bfx,\bfx'\in\calX\\
  d_\calX(\bfx,\bfx')\leq t-r}}
	\bigg(}
-\Big(d_\calY(f(\bfx),f(\bfx'))-d_\calY(f(\bfx),f(\bfx_N))
-d_\calY(f(\bfx'),f(\bfx_N'))\Big)
\bigg)\\
&=
\sup_{\substack{\bfx,\bfx'\in\calX\\d_\calX(\bfx,\bfx')\leq t-r}}
\Big(d_\calY\big(f(\bfx),f(\bfx_N)\big)
+d_\calY\big(f(\bfx'),f(\bfx_N')\big)\Big)\\
&\leq
\sup_{\substack{\bfx,\bfx'\in\calX\\d_\calX(\bfx,\bfx')\leq t-r}}
\Big(\omega\big(f,d_\calX(\bfx,\bfx_N)\big)
+\omega(f,d_\calX(\bfx',\bfx_N')\big)\Big)\\
&\leq
2\omega(f,r).
\end{align*}
Herein, we used the monotonicity of the modulus of continuity in the last
inequality. The assertion follows with \Cref{eq:addsub}.
\end{proof}
We note that the theorem recovers \Cref{cor:continuousconsistency} if the
modulus of continuity is continuous due to $\omega(f,0)=0$, see
\Cref{lem:ModCont}. We also note that moduli of continuity having a jump in
the left vicinity of $t$ cannot be well approximated. This is to be expected,
as discontinuous functions cannot be well approximated by piecewise constant
functions, in general, and the discrete modulus of continuity is indeed
piecewise constant.

We have the following corollary.
\begin{corollary}[Best $N$-term approximation rate at %
$t$]\label{cor:bestNtermt}
Let $f\colon\calX\to\calY$ be continuous and let $t>0$ and $N\in\bbN$ be fixed.
Further, set
\[
r(N,t)\isdef\min\big\{0\leq r\leq t\colon \calN(\calX,r)\leq N\big\}.
\]
Then, there exists a set $X_N=\{\bfx_1,\ldots,\bfx_N\}$ such that
\[
0
\leq
\omega(f,t)-\omega_N(f,t)
\leq
\omega(f,t)-\omega\big(f,t-r(N,t)\big)+2\omega\big(f,r(N,t)\big).
\]
\end{corollary}
\begin{remark}\label{rem:fullcover}
The case $r(N,t)=0$ in \Cref{cor:bestNtermt} implies $|\calX|\leq N$. Thus, as
$\omega(f,0)=0$, and as stated by the corollary, we can choose $N$ samples
such that the discrete modulus of continuity coincides with the continuous
modulus of continuity.
\end{remark}

\subsection{Approximation and consistency on intervals}\label{sec:lpdetconv}
While the previous results are concerned with the pointwise consistency of the
modulus of continuity with respect to \(t\), the following results inspect the
consistency with respect to bounded intervals
\(t\in[a,b]\) in \(L^p(a,b)\) with $[a,b]\subset[0,T_\calX]$.
\begin{theorem}[Covering diameter $0<r\leq a$ implies error bound in $L^p(a,b)$]
\label{thm:coveringtoboundLp}
Let $f\colon\calX\to\calY$ be continuous and consider
$[a,b]\subset (0,T_\calX]$. Moreover, let $N_r\geq\calN(\calX,r/2)$, 
$0<r\leq a$, and let
$X_{N_r}=\{\bfx_1,\ldots,\bfx_{N_r}\}$ be an $r/2$-net. Then, there holds
\[
\|\omega(f,\cdot)-\omega_{N_r}(f,\cdot)\|_{L^p(a,b)}
\leq
\begin{cases}
2^{\frac{p-1}{p}}\|\omega(f,\cdot)-\omega(f,\cdot-r)\|_{L^p(a,b)}
+2(b-a)^{1/p}\omega(f,r),&p\in[1,\infty),\\[0.2em]
\|\omega(f,\cdot)-\omega(f,\cdot-r)\|_{L^\infty(a,b)}+2\omega(f,r),&p=\infty.
\end{cases}
\]
\end{theorem}
\begin{proof}
We first consider the case $p\in[1,\infty)$. Since $t\mapsto t^p$ is convex for
$t\geq 0$, there
 holds
\[
\bigg(\frac{A+B}{2}\bigg)^p
\leq
\frac{A^p+B^p}{2}
\]
and thus
\[
(A+B)^p\leq 2^{p-1}(A^p+B^p)
\]
for any $A,B\geq 0$. Next, observing that $A=\omega(f,a)-\omega(f,a-r)$ and
$B=2\omega(f,r)$ are non-negative, see \Cref{lem:ModCont}, we obtain with
\Cref{thm:coveringtobound} that
\begin{align*}
\|\omega(f,\cdot)-\omega_{N_r}(f,\cdot)\|_{L^p(a,b)}^p
&\leq
\int_a^b|\omega(f,t)-\omega(f,t-r)+2\omega(f,r)|^p\mathrm{d} t\\
&\leq
2^{p-1}\int_a^b|\omega(f,t)-\omega(f,t-r)|^p\mathrm{d} t+2^p
\int_a^b|\omega(f,r)|^p\mathrm{d} t\\
&=
2^{p-1}\|\omega(f,t)-\omega(f,t-r)\|_{L^p(a,b)}^p+2^p(b-a)\omega(f,r)^p.
\end{align*}
Concavity of $t\mapsto t^{1/p}$ yields the assertion for $p\in[1,\infty)$. The
assertion for $p=\infty$ follows from \Cref{thm:coveringtobound} and the
triangle inequality.
\end{proof}

A consequence similar to \Cref{cor:bestNtermt} can be stated, which yields
the same insights as \Cref{rem:fullcover}. Moreover, close inspection of
\Cref{thm:coveringtoboundLp} shows that the requirements on $N$ can be too
restrictive if the interval approaches zero, that is, when $a\to 0$. This is due to
the potential growth of the covering number $\calN(\calX,r)$ when $0<r\leq a\to 0$.
This issue can be fixed as follows.

\begin{corollary}[Approximation in $L^p(0,b)$]\label{cor:detLpconv}
Let $f\colon\calX\to\calY$ be continuous and let
$\alpha\in(0,b]$, \(b\leq T_\Xcal\), such that $N\geq\calN(\calX,\alpha/2)$. 
Then, there holds
\begin{align*}
&\|\omega(f,\cdot)-\omega_{N_r}(f,\cdot)\|_{L^p(0,b)}\\
&\quad\qquad\leq
\begin{cases}
\alpha^{1/p}\omega(f,\alpha)+2^{(p-1)/p}
\|\omega(f,\cdot)-\omega(f,\cdot-\alpha)\|_{L^p(\alpha,b)}
+2(b-\alpha)^{1/p}\omega(f,\alpha),&p\in[1,\infty),\\
\omega(f,\alpha)+\|\omega(f,\cdot)-\omega(f,\cdot-\alpha)\|_{L^\infty(\alpha,b)}
+2\omega(f,\alpha),&p=\infty.
\end{cases}
\end{align*}
\end{corollary}
\begin{proof}
The assertion then follows from
\[
\|\omega(f,\cdot)-\omega_N(f,\cdot)\|_{L^p(0,b)}
\leq
\|\omega(f,\cdot)-\omega_N(f,\cdot)\|_{L^p(0,\alpha)}
+
\|\omega(f,\cdot)-\omega_N(f,\cdot)\|_{L^p(\alpha,b)}.
\]
The first term can be bounded by exploiting the non-negativity,
see \Cref{lem:ModCont}, and monotonicity, see \Cref{lem:monotonicity}, of the
discrete modulus of continuity with respect to \(N\) according to
\[
\|\omega(f,\cdot)-\omega_N(f,\cdot)\|_{L^p(0,\alpha)}
\leq
\|\omega(f,\cdot)\|_{L^p(0,\alpha)}
\leq
\begin{cases}
\alpha^{1/p}\omega(f,\alpha),&p\in[1,\infty),\\
\omega(f,\alpha),&p=\infty.
\end{cases}
\]
The second term is bounded by \Cref{thm:coveringtoboundLp}.
This yields the assertion.
\end{proof}

The combination of the previous results allow us to prove consistency in
$L^p(a,b)$ with $[a,b]\subset [0,T_\calX]$.

\begin{corollary}[Consistency in $L^p(a,b)$]\label{cor:detconsab}
	Let $f\colon\calX\to\calY$ be continuous and let
  $[a,b]\subset [0,T_\calX]$. Further let  $\{X_{N_k}\}_{k\in\bbN}$ 
  be a sequence of $r_k/2$-nets with $r_k\to 0$ as $k\to\infty$.
  Then, there holds
	\[
	\|\omega(f,\cdot)-\omega_{N_k}(f,\cdot)\|_{L^p(a,b)}\to 0
	\qquad
	\text{as}~k\to\infty
	\]
	for all $p\in[1,\infty)$. If $f$ has further a
  continuous modulus of continuity, then there especially holds
	\[
	\|\omega(f,\cdot)-\omega_{N_k}(f,\cdot)\|_{L^\infty(a,b)}\to 0
	\qquad
	\text{as}~k\to\infty.
	\]
\end{corollary}
\begin{proof}
The case $p=\infty$ is \Cref{cor:continuousconsistency}. For $p\in[1,\infty)$,
we distinguish two cases. For $a>0$ we infer from \Cref{thm:coveringtoboundLp}
that
\[
\|\omega(f,\cdot)-\omega_{N_k}(f,\cdot)\|_{L^p(a,b)}
\leq
2^{(p-1)/p}\|\omega(f,\cdot)-\omega(f,\cdot-r_k)\|_{L^p(a,b)}+2(b-a)^{1/p}
\omega(f,r_k).
\]
Continuity of the shift operator in $L^p$ implies that the first term tends to
zero as $k\to\infty$, whereas the second term tends to zero due to
\Cref{lem:ModCont}.

For $a=0$, we set $\alpha_k=\max\{r_k,b/k\}>0$. Then $\alpha_k\to 0$ as 
$k\to\infty$ and  $\{X_{N_k}\}_{k\in\bbN}$ is especially a sequence of
$\alpha_k/2$-nets.
Employing \Cref{cor:detLpconv}, this implies
\[
\begin{aligned}
&\|\omega(f,\cdot)-\omega_{N_k}(f,\cdot)\|_{L^p(0,b)}\\
&\qquad\leq
\alpha_k^{1/p}\omega(f,\alpha_k)
+
2^{(p-1)/p}\|\omega(f,\cdot)-\omega(f,\cdot-\alpha_k)\|_{L^p(\alpha_k,b)}
+
2(b-\alpha_k)^{1/p}\omega(f,\alpha_k).
\end{aligned}
\]
Therefore, a minor modification of the argument for $a>0$ yields the
assertion.
\end{proof}

This consistency result is in line with what is to be expected. The discrete
modulus of continuity is a piecewise constant function and piecewise constant
functions are good approximants for discontinuous functions in $L^p(a,b)$,
$p\in[1,\infty)$. This is not necessarily the case for $L^\infty(a,b)$, where we
require the limit function to be continuous.

\subsection{Consistency of the $|\cdot|_{\rho,N}$-seminorm}
Observing that the $|\cdot|_{\rho,N}$-seminorm of $f$ can be considered as a
weighted version of the $\|\cdot\|_{L^\infty(0,T_\calX)}$-norm of
$\omega_N(f,\cdot)$, the following consistency result is an intuitive extension
of \Cref{cor:detconsab}. The
only difficulty is a potential singularity of $\rho^{-1}$ close to the origin.

\begin{theorem}[Consistency of the $|\cdot|_{\rho,N}$-seminorm]%
\label{thm:seminormconv}
Let $\omega(f,\cdot),\rho(\cdot)^{-1}\omega(f,\cdot)\colon[0,\infty]
\to[0,\infty)$ be continuous and $|f|_{\rho}<\infty$. Let $\{r_k\}_{k\in\bbN}$
be a non-increasing sequence with $\lim_{k\to\infty}r_k=0$ and
$X_{N_k}=\{\bfx_1,\ldots,\bfx_{N_k}\}\subset\calX$ be $r_k/2$-nets of $\calX$
with $X_{N_k}\subset X_{N_{k+1}}$, $k\in\bbN$. Then, there holds
\[
\lim_{k\to\infty}|f_{N_k}|_{\rho,N_k}
=
\lim_{k\to\infty}|f|_{\rho,N_k}
\to
|f|_\rho,\quad\text{as}~k\to\infty.
\]
\end{theorem}
\begin{proof}
It follows from the definition that $|f|_{\rho,N_k}$ is non-decreasing in $k$
and bounded from above by $|f|_\rho$. This yields
$\lim_{k\to\infty}|f|_{\rho,N_k}\leq |f|_\rho$. We show by contradiction that
the inequality is an equality. Assume that there exists $0<M<|f|_\rho$ such that
$|f|_{\rho,N_k}\leq M$ for all $k$. By the definition of the
$|f|_\rho$-seminorm, there is a sequence
$\{{t}_j\}_{j\in\bbN}\subset\calZ\setminus\{0\}$ with $\calZ$ 
as in \Cref{eq:Z}, such that
\[
  \lim_{j\to\infty}\rho({t}_j)^{-1}\omega(f,{t}_j)=|f|_\rho.
\]
Thus, there exists $J\in\bbN$ such that $\rho({t}_J)^{-1}\omega(f,{t}_J)>M$.

Now, let $0<\varepsilon<\rho({t}_J)^{-1}\omega(f,{t}_J)-M$. Since
$\rho(\cdot)^{-1}\omega(f,\cdot)$ is continuous, \Cref{lem:metricmapscovering}
implies that there is $k\in\bbN$ large enough such that $s\in Z_{N_k}$,
with $Z_{N_k}$ as in \Cref{eq:Z}, satisfies
\[
\big|
\rho({t}_J)^{-1}\omega(f,{t}_J)
-
\rho(s)^{-1}\omega(f,s)
\big|
\leq
\frac{\varepsilon}{2}.
\]
Exploiting $Z_{N_k}\subset Z_{N_{k+\ell}}$ for $\ell\in\bbN$ and the
continuity of $\omega(f,\cdot)$ as well as \Cref{lem:ModCont} and
\Cref{thm:coveringtobound}, we conclude that there further exists $\ell$ large
enough such that
\[
0
\leq
\rho(s)^{-1}\big(\omega(f,s)-\omega_{N_{k+\ell}}(f,s)\big)
\leq
\rho(s)^{-1}\big(\omega(f,s)-\omega(f,s-r_{k+\ell})+\omega(f,r_{k+\ell})
\big)
\leq
\frac{\varepsilon}{2}.
\]
The triangle inequality thus implies
\[
\big|
\rho({t}_J)^{-1}\omega(f,{t}_J)
-
\rho(s)^{-1}\omega_{N_{k+\ell}}(f,s)
\big|
\leq\varepsilon.
\]
Based on the previous bound together with the choice of 
$0<\varepsilon<\rho({t}_J)^{-1}\omega(f,{t}_J)-M$, we obtain
\begin{align*}
\rho(s)^{-1}\omega_{N_{k+\ell}}(f,s)
&=
\rho(s)^{-1}\omega_{N_{k+\ell}}(f,s)-\rho({t}_J)^{-1}
\omega(f,{t}_J)+\rho({t}_J)^{-1}\omega(f,{t}_J)\\
&>
-\varepsilon+\rho({t}_J)^{-1}\omega(f,{t}_J)
>M.
\end{align*}
This implies
\[
|f|_{\rho,N_{k+\ell}}
\geq
\rho(s)^{-1}\omega_{N_{k+\ell}}(f,s)
>
M,
\]
which is a contradiction to the assumption.
\end{proof}

\section{Consistency for empirical data sites}\label{sec:Consistency}
The following section is dedicated to consistency estimates for the discrete
modulus of continuity for empirical data sites. Assuming that the data sites are
drawn according to a probability distribution and that the site-to-value map
$f_N\colon X_N\to Y_N$ is the restriction of a function
$f\colon\calX\to\calY$ we derive probabilistic estimates 
for the approximation of the modulus of continuity by the discrete one.
Further, we discuss the consistency of the discrete modulus of continuity in the
infinite sample limit, that is, $N\to\infty$. The additional requirement on
$\calX$ compared to \Cref{sec:detcons} is as follows.

Let \(\Bcal\subset 2^\Xcal\) denote the Borel \(\sigma\)-field on \(\Xcal\).
Throughout this section, we
assume that the measurable space \((\Xcal,\Bcal)\) is equipped with a
probability measure \(\Pbb\colon\Bcal\to[0,1]\). Without loss of generality we
assume that $\bbP$ is nondegenerate, that is, all open neighborhoods of all
elements in $\calX$ have positive measure. 
Throughout this section, we assume that the data sites in 
$X_N=\{\bfx_1,\ldots,\bfx_N\}$ have been drawn independently according to
$\bbP$.

\subsection{Random coverings}
The following two standard results are important ingredients to prove the
consistency of the discrete modulus of continuity for empirical data sites.
We state here for the reader's convenience.
\begin{lemma}[Positive probability of balls]\label{lem:infPnonzero}
There holds
\[
\inf_{\bfx\in\calX}\bbP\big(B_r(\bfx)\big)>0
\]
for all $r>0$.
\end{lemma}
\begin{proof}
We prove the claim by contradiction. Assume that there exists an \(r>0\)
such that \(\inf_{\bfx\in\calX}\bbP\big(B_r(\bfx)\big)=0\). Then,
there exists a minimizing sequence
$\{\bfx_n\}_{n\in\bbN}\subset\calX$ such that
\[
\lim_{n\to\infty}\bbP\big(B_r(\bfx_n)\big)
=
\inf_{\bfx\in\calX}\bbP\big(B_r(\bfx)\big)
=
0.
\]
Since $\calX$ is compact, the minimizing sequence $\{\bfx_n\}_{n\in\bbN}$ has a
convergent subsequence $\{\bfx_{\varphi(n)}\}_{n\in\bbN}$ with limit
$\bfx^\star\in\calX$. There holds 
\begin{equation}\label{eq:probzero}
\lim_{n\to\infty}\bbP\big(B_r(\bfx_{\varphi(n)})\cap B_r(\bfx^\star)\big)
\leq
\lim_{n\to\infty}\bbP\big(B_r(\bfx_{\varphi(n)})\big)
=
0.
\end{equation}
Furthermore, we have for each ${\bfx}\in B_r(\bfx_{\varphi(n)})$ that
\[
d_\calX({\bfx},\bfx^\star)
\leq
d_\calX({\bfx},\bfx_{\varphi(n)})+d_\calX(\bfx_{\varphi(n)},\bfx^\star)
\leq
r+\varepsilon(n),
\]
where we set \(\varepsilon(n)\isdef d_\calX(\bfx_{\varphi(n)},\bfx^\star)\).
This implies the inclusion
\(
B_r(\bfx_{\varphi(n)})
\subset
B_{r+\varepsilon(n)}(\bfx^\star)
\)
and, therefore, $B_r(\bfx_{\varphi(n)})\cap B_r(\bfx^\star)\to B_r(\bfx^\star)$
for \(n\to\infty\), since \(\varepsilon(n)\to 0\).
By the continuity of \(\Pbb\) from below and its nondegeneracy on $\calX$,
we conclude
\[
\lim_{n\to\infty}\bbP\big(B_r(\bfx_{\varphi(n)})\cap B_r(\bfx^\star)\big)
=
\bbP\big(B_r(\bfx^\star)\big)>0.
\]
This is a contradiction to \Cref{eq:probzero}.
\end{proof}

\begin{lemma}[Covering probability]\label{lem:coveringprob}
For any \(r\geq0\), there holds
\[
\bbP\bigg(\bigcup_{n=1}^NB_{r}(\bfx_n)=\calX\bigg)
\geq P(r,N)
\]
with 
\begin{equation}\label{eq:CoveringProb}
P(r,N)\isdef
1-\calN(\calX,r/2)\eta(r/2)^N,\qquad
\eta(r)
\isdef
\sup_{\bfx\in\calX}
\Big(1-\bbP\big(B_{r}(\bfx)\big)\Big)<1.
\end{equation}
\end{lemma}
\begin{proof}
From \Cref{lem:infPnonzero}, we directly infer that $\eta(r)<1$.
Further, by definition of the covering number, we know that there exist
$\calN(\calX,r/2)$ balls 
\(B_{r/2}({\bs x}_n')\), \(n=1,\ldots,\calN(\calX,r/2)\),
covering $\calX$. This implies by
the Bonferroni inequality that
\begin{align*}
\bbP\bigg(\bigcup_{n=1}^NB_{r}(\bfx_n)=\calX\bigg)
&\geq
\Pbb\big(X_N
\cap B_{r/2}(\bfx_n')\neq\emptyset
\text{ for }n=1,\ldots,\calN(\calX,r/2)\big)\\
&\geq
1-
\sum_{n=1}^{\calN(\calX,r/2)}\bbP\big(X_N\cap B_{r/2}(\bfx_n')=\emptyset\big).
\end{align*}
Now, since the points $\bfx_1\ldots,\bfx_N$ are independent realizations we
obtain
\[
\bbP\big(
X_N
\cap B_{r/2}(\bfx_n')=\emptyset\big)
=\Big(1-\Pbb\big(B_{r/2}(\bfx_n')\big)\Big)^N\leq\eta(r/2)^N
\]
for any $n=1,\ldots,\Ncal(\Xcal,r)$.
Inserting this bound into the previous estimate yields the assertion.
\end{proof}
The following straightforward monotonicity result will be helpful later on.
\begin{lemma}[Covering probabilities are monotonically increasing]%
\label{lem:Pmonotonicity}
    Let $N\in\bbN$. Then there holds
    \[
    P(r,N)\leq P(r',N)\leq 1
    \]
    for all $0\leq r\leq r'$.
\end{lemma}

\begin{remark}
\Cref{lem:coveringprob,lem:Pmonotonicity} formalize what 
is intuitively clear: For fixed $N$, the probability of covering $\calX$ with
$N$ balls of radius $r$ decreases when $r$ decreases. Vice versa, for fixed $r$,
the probability of covering $\calX$ with $N$ balls with radius $r$ increases
for increasing $N$.
\end{remark}

\subsection{Pointwise consistency}
The following theorem connects the deterministic pointwise approximation result
from \Cref{thm:coveringtobound} with the probability that the empirical data
sites satisfy the required $r/2$-net property.
\begin{theorem}[Consistency in probability for fixed $t$]%
\label{thm:pointwise_prob}
Let $f\colon\calX\to\calY$ have continuous modulus of continuity. Then,
$\omega_N(f,t)$ converges to $\omega(f,t)$ in probability
for all \(t\geq 0\), that is, for all $\varepsilon>0$, there holds
\[
\lim_{N\to\infty}\bbP\big(\omega(f,t)-\omega_N(f,t)\leq\varepsilon\big)
=1\quad\text{for all }t\geq 0.
\]
More precisely, for all $t\geq 0$ there exists $r\in(0,t]$, which only depends
on $\omega(f,\cdot)$ and $t$, such that
\begin{equation}\label{eq:quantitativebound}
\bbP\big(\omega(f,t)-\omega_N(f,t)\leq\varepsilon\big)
\geq
\bbP\bigg(\bigcup_{n=1}^NB_{r/2}(\bfx_n)=\calX\bigg)
\geq
P(r/2,N).
\end{equation}
\end{theorem}

\begin{proof}
The result for $t=0$ follows trivially. Hence, let $t>0$.
Since $\omega(f,\cdot)$ is continuous there exists for every $\varepsilon>0$
an $r\in(0,t]$ such that
\[
\omega(f,t)-\omega(f,t-r)\leq\frac{\varepsilon}{3},
\qquad
\omega(f,r)\leq\frac{\varepsilon}{3}.
\]
\Cref{thm:coveringtobound} thus implies
\[
0
\leq
\omega(f,t)-\omega_N(f,t)
\leq
\omega(f,t)-\omega(f,t-r)+2\omega(f,r)
\leq
\varepsilon
\]
for all sets of data sites $X_N$ which form an $r/2$-net.
Estimate \eqref{eq:quantitativebound} follows now
with \Cref{lem:coveringprob} and implies in view of
in \Cref{eq:CoveringProb}
convergence in probability.
\end{proof}

Based on the pointwise consistency result in probability, consistency results
in $L_\bbP^s(\calX)$ and $\bbP$-almost sure consistency for fixed $t$ can be
readily deduced as follows.

\begin{corollary}[Consistency of moments for fixed $t$]%
\label{cor:pointwise_Lr}
Under the assumptions of \Cref{thm:pointwise_prob}, 
there holds for all $t\geq 0$, $s\in[1,\infty)$ that
\[
\lim_{N\to\infty}\bbE\big[|\omega(f,t)-\omega_N(f,t)|^s\big]=0.
\]
\end{corollary}

\begin{proof}
It is sufficient to show that $\omega_N(f,t)$ is uniformly integrable over
$\calX$, see, for example, \cite[Exercise 6.15]{Folland}. This is trivially
fulfilled, since \Cref{lem:monotonicity} implies that
\[
\bbE[|\omega_N(f,t)|^s]
\leq
\bbE[|\omega(f,t)|^s]
=|\omega(f,t)|^s
\]
for all $s\in[1,\infty)$.
\end{proof}

\begin{corollary}[$\bbP$-almost sure consistency for fixed $t$]%
\label{cor:pointwise_as}
Under the assumptions of \Cref{thm:pointwise_prob}, and for all $t\geq 0$,
$\omega_N(f,t)$ converges to $\omega(f,t)$ $\bbP$-almost surely, that is,
\[
\bbP\Big(\lim_{N\to\infty}\omega_N(f,t)=\omega(f,t)\Big)=1.
\]
\end{corollary}

\begin{proof}
The consistency in probability, as shown in \Cref{thm:pointwise_prob}, 
implies that there exists a subsequence $\omega_{N_k}(f,t)$ which converges 
almost surely to $\omega(f,t)$, that is,
\[
\bbP\Big(\lim_{k\to\infty}\omega_{N_k}(f,t)=\omega(f,t)\Big)=1,
\]
see, for example, \cite[Theorem 2.30]{Folland}.
The monotonicity of $\omega_N(f,t)$ in $N$, see \Cref{lem:monotonicity},
now implies the assertion.
\end{proof}

\subsection{Consistency on intervals}
In analogy to the deterministic case, the following
results inspect the consistency in \(L^p(a,b)\).
\begin{theorem}[Consistency in probability in $L^p(a,b)$]\label{thm:Lr_prob}
Let $[a,b]\subset[0,T_\calX]$. Let $f\colon\calX\to\calY$ be continuous for
$p\in[1,\infty)$ and let $f$ have continuous modulus of continuity for
$p=\infty$. Then, for all $p\in[1,\infty]$, we have
$\|\omega_N(f,\cdot)-\omega(f,\cdot)\|_{L^p(a,b)}\to 0$ in probability, 
that is, for all $\varepsilon>0$, there holds
\[
\lim_{N\to\infty}\bbP\big(\|\omega(f,\cdot)-\omega_N(f,\cdot)\|_{L^p(a,b)}
\leq\varepsilon\big)=1.
\]
More precisely, for each interval $[a,b]$, there exists $r\in(0,b]$, which only
depends on $\omega(f,\cdot)$ and $[a,b]$, such that
\[
\bbP\big(\|\omega(f,\cdot)-\omega_N(f,\cdot)\|_{L^p(a,b)}\leq\varepsilon\big)
=
\bbP\bigg(\bigcup_{n=1}^NB_{r/2}(\bfx_n)=\calX\bigg)
\geq
P(r/2,N).
\]
\end{theorem}

\begin{proof}
Without loss of generality we can assume that $a=0$. In this case,
we have for all $\alpha\in[0,b]$ that 
\[
\|\omega(f,\cdot)-\omega_N(f,\cdot)\|_{L^p(0,b)}
\leq
\|\omega(f,\cdot)-\omega_N(f,\cdot)\|_{L^p(0,\alpha)}
+
\|\omega(f,\cdot)-\omega_N(f,\cdot)\|_{L^p(\alpha,b)}.
\]
In analogy to the proof of \Cref{cor:detLpconv} the first term on the right-hand
side can be estimated by
\[
\|\omega(f,\cdot)-\omega_N(f,\cdot)\|_{L^p(0,\alpha)}
\leq
\|\omega(f,\cdot)\|_{L^p(0,\alpha)}
\leq
\begin{cases}
	\alpha^{1/p}\omega(f,\alpha),&p\in[1,\infty),\\
	\omega(f,\alpha),&p=\infty.
\end{cases}
\]
Now, for fixed $\varepsilon>0$, we choose $0<\alpha\leq b$ such that
$\alpha^{1/p}\omega(f,\alpha)<\varepsilon/2$ for $p\in[1,\infty)$
or $\omega(f,\alpha)<\varepsilon/2$ for $p=\infty$. Then, we obtain
the bound
\[
\|\omega(f,\cdot)-\omega_N(f,\cdot)\|_{L^p(0,b)}
\leq
\frac{\varepsilon}{2}+\|\omega(f,\cdot)-\omega_N(f,\cdot)\|_{L^p(\alpha,b)},
\]
which holds \(\Pbb\)-almost surely. It remains to show that
\[
\bbP\bigg(\|\omega(f,\cdot)-\omega_N(f,\cdot)\|_{L^p(\alpha,b)}
\leq\frac{\varepsilon}{2}\bigg)
\to
1,
\qquad
N\to\infty.
\]
For $p\in[1,\infty)$ continuity of the shift operator and \Cref{lem:ModCont}
imply that there exists $r\in(0,\alpha]$ such that
\[
2^{(p-1)/p}\|\omega(f,\cdot)-\omega(f,\cdot-r)\|_{L^p(\alpha,b)}
+2(b-\alpha)^{1/p}\omega(f,r)
\leq
\frac{\varepsilon}{2}.
\]
Further, since we consider the limit $N\to\infty$ and since
$\calX$ is compact, we can assume that $N\geq\calN(\calX,r/2)$. Then,
\Cref{thm:coveringtoboundLp} implies that
\begin{align*}
&\bbP\bigg(\|\omega(f,\cdot)-\omega_N(f,\cdot)\|_{L^p(\alpha,b)}
\leq\frac{\varepsilon}{2}\bigg)\\
&\qquad\geq
\bbP\bigg(2^{(p-1)/p}\|\omega(f,\cdot)-\omega(f,\cdot-r)\|_{L^p(\alpha,b)}
+2(b-\alpha)^{1/p}\omega(f,r)\leq\frac{\varepsilon}{2}\bigg)\\
&\qquad\geq
\bbP\bigg(\bigcup_{n=1}^NB_{r/2}(\bfx_n)=\calX\bigg)\\
&\qquad\geq
P(r/2,N)\\
&\qquad\to
1
\end{align*}
as $N\to\infty$. The assertion for $p=\infty$ follows in complete analogy,
using the corresponding bound from \Cref{thm:coveringtoboundLp} for this case.
\end{proof}

Consistency results in $L_\bbP^s(\calX)$ and $\bbP$-almost sure consistency on
intervals are an immediate consequence of the previous theorem.

\begin{corollary}[Consistency of moments in $L^p(a,b)$]%
\label{cor:Lr_Lr}
Under the assumptions of \Cref{thm:Lr_prob}, there holds 
$\|\omega(f,\cdot)-\omega_N(f,\cdot)\|_{L^p(a,b)}\to 0$
in $L_\bbP^s(\calX)$
for any $p\in[1,\infty]$ 
and any $s\in[1,\infty)$, that is, 
\[
\lim_{N\to\infty}
\Big(\bbE\big[\|\omega(f,\cdot)-\omega_N(f,\cdot)\|_{L^p(a,b)}^s\big]\Big)^{1/s}
=0.
\]
\end{corollary}
\begin{proof}
The proof is in analogy to \Cref{cor:pointwise_Lr}.
\end{proof}

\begin{corollary}[$\bbP$-almost sure consistency in $L^p(a,b)$]%
\label{cor:Lr_as}
Under the assumptions of \Cref{thm:Lr_prob}, and with $p\in[1,\infty]$, 
the sequence $\|\omega(f,\cdot)-\omega_N(f,\cdot)\|_{L^p(a,b)}\to 0$
converges almost surely, that is,
\[
\bbP\Big(\lim_{N\to\infty}
\|\omega(f,\cdot)-\omega_N(f,\cdot)\|_{L^p(a,b)}=0\Big)=1.
\]
\end{corollary}

\begin{proof}
The assertion follows similarly to the proof of \Cref{cor:pointwise_as}.
\end{proof}

\subsection{Consistency of the empirical $|\cdot|_{\rho,N}$-seminorm}
The consistency of the $|\cdot|_{\rho,N}$-seminorm for empirical data sites
follows directly from the deterministic case.
\begin{theorem}[Consistency of the empirical $|\cdot|_{\rho,N}$-seminorm]%
\label{thm:empiricalseminormconsistency}
Let $\rho(\cdot)^{-1}\omega(f,\cdot)\colon[0,\infty]\to[0,\infty)$ be continuous
and $|f|_\rho<\infty$. Then, there holds
\[
\lim_{N\to\infty}|f_N|_{\rho,N}
=
\lim_{N\to\infty}|f|_{\rho,N}
=
|f|_\rho
\]
$\bbP$-almost surely.
\end{theorem}

\begin{proof}
We first show by contradiction that the fill distance $h_{X_N}$, see
\Cref{eq:FillDistance}, of $N$ random, independently drawn points on $\calX$
converges to zero $\bbP$-almost surely for $N\to\infty$. Assume that
$h_{X_N}\geq h_0>0$ for all $N\in\bbN$. Then, there exists $\bfx\in\calX$ such
that $\emptyset=B_{h_0}(\bfx)\cap\{\bfx_i\}_{i=1}^N$ for all $N\in\bbN$. By the
same argumentation as in the proof of \Cref{lem:coveringprob}, the probability
of this being the case is bounded by
\[
\bbP\Big(\{\bfx_i\}_{i=1}^N\cap B_{h_0}(\bfx)=\emptyset\Big)
\leq
\eta(h_0)^N
\]
with $\eta(h_0)<1$ defined as in \Cref{eq:CoveringProb}. For $N\to\infty$,
this probability goes to zero, meaning that with probability 1 there exists
$i\in\bbN$ such that $\bfx_i\in B_{h_0}(\bfx)$. This is a contradiction to
$h_0$ being a lower bound to the fill distance for all $N\in\bbN$.

Hence, we infer that the family of balls $\{B_{h_{X_N}}(\bfx_i)\}_{i=1}^N$
generates a sequence of $h_{X_N}$-nets on $\calX$ and the assumptions of
\Cref{thm:seminormconv} are satisfied $\bbP$-almost surely with $r_k/2=h_{X_N}$.
This implies the assertion.
\end{proof}

\section{Computation of the discrete modulus of continuity}%
\label{sec:Computation}
The computation of the discrete modulus of continuity from
\Cref{eq:discmodcont}, that is,
\[
\omega_N(f,t)=\max_{\gfrac{\bfx_i,\bfx_j\in X_N:}{d_{i,j}\leq t}}
d_\calY\big(f(\bfx_i),f(\bfx_j)\big),
\]
requires the evaluation of the distances $d_\calY\big(f(\bfx_i),f(\bfx_j)\big)$
for all $i,j=1,\ldots,N$. Therefore, the computational cost is $\calO(N^2)$.
The same is true if we aim at computing the discrete seminorm $|f|_{\rho,N}$,
see \Cref{eq:discLipNorm,eq:dataLipNorm}, from data.
Fortunately, for sufficiently small values of $t$, the computational cost
can be reduced if the metric space \(\Xcal\) exhibits sufficient
structure and if we restrict the evaluation of
$d_\calY\big(f(\bfx_i),f(\bfx_j)\big)$ to pairs of data sites $(\bfx_i,\bfx_j)$
with distance smaller than a fixed $t$. If such structure is present, the latter
can efficiently be realized using an $\varepsilon$-nearest neighbor search.

Based on this consideration, we provide a heuristic algorithm for the efficient
evaluation of the modulus of continuity. This heuristic is based on two further
observations. The first observation is that the modulus of continuity
is monotonically increasing, see \Cref{lem:ModCont}. The second observation is
that, for our applications, we require an accurate approximation of the
modulus of continuity only close to the zero, where its magnitude is typically
small. Less precise approximations are acceptable for increasing distance from
zero. In the following, we first recall the efficient implementation of the
$\varepsilon$-nearest neighbor search as well as a coarsening
strategy for set covers. Afterwards, we discuss the details of our algorithm for
the efficient evaluation of the modulus of continuity.

\subsection{$\varepsilon$-nearest neighbor search}\label{sec:epsnearest}
For an efficient implementation of the search for nearest neighbors, further
assumptions on the structure of the metric space \((\Xcal,d_{\Xcal})\) are
required, compare \Cref{ex:Assouad}.
\begin{definition}
Let $(\calX,d_\calX)$ be a metric space.
\begin{enumerate}
\item The space $(\calX,d_\calX)$ is \emph{Ahlfors $n$-regular}, if there exists
a Borel regular measure $\mu$ and constants $C\geq 1$, $n>0$ such that
\[
C^{-1}R^n\leq\mu\big(B_R(\bfx)\big)\leq CR^n
\]
for all $\bfx\in\calX$, $R>0$.
\item The space $(\calX,d_\calX)$ is \emph{\((C,s)\)-homogeneous}, if 
\[
\Ncal\big(B_R({\bs x}),r\big)\leq C\bigg(\frac R r\bigg)^s
\]
for all \({\bs x}\in\Xcal\) and all \(0<r<R<\infty\). 
The infimum of the set of all \(s\) such that 
\((\Xcal,d_\Xcal)\) is \((C,s)\)-homogeneous is 
the \emph{Assouad dimension} of \((\Xcal,d_\Xcal)\).
\item The space $(\calX,d_\calX)$ has bounded \emph{doubling dimension}
if there exists \(0\leq s<\infty\) such that
\[
\Ncal\big(B_r({\bs x}),r/2\big)\leq 2^s
\]
for all \({\bs x}\in\Xcal\) and any \(r>0\). We refer to the infimum over
all \(s\) such that the bound holds as the doubling dimension of
$(\calX,d_\calX)$.
\end{enumerate}
\end{definition}
A bounded Assouad dimension is equivalent to 
a bounded doubling dimension, see, for example, \cite{Hei01}. 
For the particular case of an Ahlfors $n$-regular space, however,
the Assouad dimension and the doubling dimension are both equal to $n$.

For the purpose of deriving cost bounds of the $\varepsilon$-nearest neighbor
search, it is assumed in literature that $(\calX,d_\calX)$ has bounded doubling
dimension \(s\).
The cost for the \(\varepsilon\)-nearest neighbor search of a given point 
\({\bs x}\in X_N\) is then bounded by
\(2^{cs}(\log\Delta+|B_{C\varepsilon}({\bs x})\cap X_N|)\)
for two constants \(c,C>0\), where \(\Delta\isdef \diam(X_N)/q_{X_N}\)
is the \emph{aspect ratio} of \(X_N\), see \cite{KL2004, BKL2006},
and $q_{X_N}$ is the separation distance, see \Cref{eq:SepDistance}.

\begin{remark}\label{rem:nncomplexity}
Let $X_N\subset\calX$ a quasi-uniform set of data sites,
see \Cref{eq:quasi-unif}, in an Ahlfors $n$-regular metric space $\calX$.
Especially, the doubling dimension equals $n$ in this case.
Using a volume argument it is easy to show that
\(c^{-1} N^{-1/n}\leq q_{X_N}\leq c N^{-1/n}\) for some \(c>0\). This implies 
\[
\Delta=\calO(N^{1/n}),
\qquad
\big|B_{\varepsilon}(\bfx)\cap X_N\big|
=\calO\big((\varepsilon/q_{X_N})^{n}\big)=\calO(\varepsilon^nN).
\]
Therefore, the cost of finding the \(\varepsilon\)-nearest neighbors for each
\({\bs x}\in X_N\) is $\calO\big(N(\log N+\varepsilon^nN)\big)$ with the
constant depending on the quasi-uniformity constant and the
Ahlfors regularity, measure, and constant. For the choice
$\varepsilon=\Ocal\big(N^{-1/n}\log^{1/n}(N)\big)$ the cost becomes
$\calO(N\log N)$.
\end{remark}

The data structure for the efficient \(\varepsilon\)-nearest neighbor search can
be organized as a \emph{cluster tree}, see, for example, \cite{Hac2015}.

\begin{definition}
A \emph{cluster tree} $\calT$ for \(X_N\)
is a tree whose vertices correspond to non-empty subsets of $X_N$ and are
referred to as \emph{clusters}. We require that the root of $\calT$ corresponds
to $X_N$ and that it holds $\dot{\cup}_{\tau'\in\children(\tau)}\tau' =\tau$ 
for all non-leaf clusters, that is, \(\children(\tau)\neq\emptyset\). Furthermore,
we introduce the set of leaf clusters 
\(\Lcal\isdef\{\tau\in\Tcal:\children(\tau)=\emptyset\}\). The 
\emph{level} \(\ell(\tau)\) of \(\tau\in\Tcal\) is its distance from the root.
\end{definition}

While cluster trees can be constructed on any metric space with finite doubling
dimension, algorithms and complexity estimates can become quite technical. For
simplicity we restrict ourselves in the following to quasi-uniform data sets in
Ahlfors $n$-regular spaces. In this setting, we can simply consider
\emph{uniform} \emph{$n$-d-trees}.

\begin{definition}\label{def:uniformTree}
	Let $\Tcal$ be a cluster tree for $X_N\subset\Xcal$.
We call $\Tcal$ an \emph{\(n\)-d-tree}, 
		if each non-leaf cluster has exactly \(2^{\lceil n\rceil}\) children, that is,
		\(|\children(\tau)|=2^{\lceil n\rceil}\ \text{
		for all }\tau\in\Tcal\setminus\Lcal.
		\)
It is a \emph{balanced $n$-d-tree}, if it is an $n$-d-tree and 
and there is some $\cbin>0$ such that it holds
\[
\cbin^{-1}2^{J-{\lceil n\rceil}\ell(\tau)}
\leq
|\tau|
\leq
\cbin2^{J-{\lceil n\rceil}\ell(\tau)}
\]
for all $\tau\in\Tcal$. Finally, \(\Tcal\) is a 
\emph{uniform \(n\)-d-tree} if it is a balanced \(n\)-d-tree
and
\[
\cdiam^{-1}2^{-\cunif\ell(\tau)}
\leq
\diam(\tau)
\leq
\cdiam2^{-\cunif\ell(\tau)}
\]
for all clusters \(\tau\in\Tcal\) and some constants \(\cdiam,\cunif>0\).
\end{definition}

\begin{remark}
If \(X_N\subset\Omega\) for some region \(\Omega\subset\Rbb^k\) endowed with a
norm, then a balanced \(k\)-d-tree can be obtained by successively subdividing
the bounding box \(B_\tau\), that is, the smallest 
axis-parallel cuboid that contains all points, of a given cluster \(\tau\) into
\(2^k\) congruent cuboids. If \(X_N\) is quasi-uniform, this even
yields a uniform \(k\)-d-tree, see, for example, \cite{HMSS24}.
The cost for this algorithm is then \(\Ocal({N\log N})\). 
For general metric spaces, graph-based clustering approaches may be employed
to construct a cluster tree based on nested \(2^{-\ell}\)-nets leading to
navigating nets as in \cite{KL2004}.
\end{remark}

For illustrative purposes we recall the
$\varepsilon$-nearest neighbor algorithm for the case when a cluster tree on
$X_N$ is available in \Cref{alg:epsnearestneighbor}.
\begin{algorithm}[htb]
\caption{\label{alg:epsnearestneighbor}$\varepsilon$-nearest neighbor
algorithm based on cluster trees}
\begin{algorithmic}
	\State \textbf{Input:} cluster tree $\calT$ on $X_N$, $\bfx\in X_N$,
         $\varepsilon>0$
	\State \textbf{Output:} $\calR=B_\varepsilon(\bfx)\cap X_N$
	\State Set $\calQ = \{\operatorname{root}(\calT)\}$, $\calR=\emptyset$
	\While{$\calQ\neq\emptyset$}
	\State Select the last element $\tau$ that has been added to $\calQ$ and set
  $\calQ=\calQ\setminus\{\tau\}$
	\If{$\children(\tau)=\emptyset$}
	\State Set $\calR=\calR\cup(\{\tau\}\cap B_\varepsilon(\bfx))$
	\Else
	\For{$\tau'\in\children(\tau)$}
	\If{$\dist_\calX(\tau',\bfx)-\diam_\calX(\tau')<\varepsilon$}
	\State Set $\calQ=\calQ\cup\{\tau'\}$
	\EndIf
	\EndFor
	\EndIf
	\EndWhile
	\State \Return $\calR$
\end{algorithmic}
\end{algorithm}
For practical applications we may replace the criterion
$\dist_\calX(\tau',\bfx)-\diam_\calX(\tau')<\varepsilon$ by an upper
estimate using bounding boxes or spheres if applicable, see, for example,
\cite{Hac2015}. For the more general case we refer to \cite{KL2004}.

\subsection{Algorithms for the set cover problem}
Given the sets of data sites $X_N$ and and the set of data values $Y_N$,
the $\varepsilon$-nearest neighbor search from the previous paragraph allows
for the computation of the discrete modulus of continuity with cost
$\calO\big(N(\log N+t^nN)\big)$ within the setting of
\Cref{rem:nncomplexity}. This allows for the efficient computation of the
modulus of continuity for small values of $t$.
For larger values,
the
cost of evaluating the modulus of continuity easily becomes $\Ocal(N^2)$
and, thus, computationally prohibitive for large values of $N$. Within this section,
we shall devise a coarsening strategy to mitigate this computational burden.

The key observation is that an $r/2$-covering of $X_N$ is sufficient to
approximate $\omega_N(f_N,t)$ up to an accuracy of
$\omega_N(f_N,t)-\omega_N(f_N,t-r)+2\omega_N(f_N,r)$ for all $t\geq r$, 
see \Cref{thm:coveringtobound}. Therefore, if we can coarsen the set
$X_N$ to a subset $X_{N'}\subset X_N$,
$|X_{N'}|=N'$, whose elements are the centers of an $r/2$-covering of
$X_N$, we can evaluate $\omega_{N'}f(f_{N'},t)\approx \omega_N(f_N,t)$ up to
an accuracy of $\omega_N(f_N,t)-\omega_N(f_N,t-r)+2\omega_N(f_N,r)$ with
computational cost $\calO\big((N')^2\big)$. 
This approach yields a reasonable approximation as long
as $\omega_N(f_N,t)-\omega_N(f_N,t-r)$ is small. This is the case
if $\omega_N(f_N,t)$ behaves reasonably in a vicinity of $t$. 

We note that
finding an optimal coarsening corresponds to the \emph{set cover problem},
which is well known to be NP-hard. It has been subject to intense research
efforts for polynomial approximation algorithms and we refer to
\cite{KV2018,Vaz2001} for reviews. A particular greedy strategy from
\cite{Chv1979} for its approximate solution is given in \Cref{alg:greedycovering}.
\begin{algorithm}[htb]
\caption{\label{alg:greedycovering}Greedy algorithm for set cover problem,
see \cite{Chv1979}.}
\begin{algorithmic}
\State \textbf{Input:} Set of data sites $X_N$, radius $r/2>0$
\State \textbf{Output:} $X_{N'}\subset X_N$, $|X_{N'}|=N'$,
such that $X_N\subset\bigcup_{\bfx\in X_{N'}}B_{r/2}(\bfx)$
\State Set $X_{N'}=\emptyset$, $Z=X_N$
\While{$Z\neq\emptyset$}
\State Compute $S_j=B_{r/2}(\bfx_j)\cap Z$, $j=1,\ldots,N$
\State Set $j=\argmax\{j=1,\ldots,N\colon |S_j|\}$
\State Set $X_{N'}=X_{N'}\cup \bfx_j$
\State Set $Z=Z\setminus S_j$
\EndWhile
\State \Return $X_{N'}$
\end{algorithmic}
\end{algorithm}
An efficient implementation of \Cref{alg:greedycovering} based on
boolean lists requires only a single precomputation of the sets
$S_j$, see, for example, \cite{Bau2022}. The number of steps
performed by the greedy strategy is $N'$. Moreover, it returns an
\[
H(N)
=
\sum_{k=1}^N\frac{1}{k}
\leq
\log(N)+1
\]
optimal covering, in the sense that $N'\leq H(N)\cdot\calN(X_N,r/2)$. 
Up to lower order terms this bound is optimal, see \cite{Sla1996}.

If a binary max-heap is used for the organization of the sets $S_j$, where the keys are the current cardinalities
\(|S_j|\),
the maximum can be retrieved with cost \(\Ocal(1)\). For constructing such a heap the cost is \(\Ocal(N)\) when using the
\texttt{heapify} operation, see, e.g., \cite{Cor09}, which has to be done before the while loop. Each iteration
of the while loop further entails one \texttt{pop} operation with cost at most
\(\log N\) and, assuming that $r=\calO(q_{X_N})$ and $X_N$ is quasi-uniform, \(\Ocal\big((r/2)^n\big)\) \texttt{decreaseKey} operations
also of cost at most \(\log N\). In view of $N'\leq H(N)\cdot\calN(X_N,r/2)$
and \(\calN(X_N,r/2)=\Ocal\big((r/2)^{-n}\big)\), see \Cref{ex:Assouad}.2,
this leads to an overall cost of 
\begin{equation}\label{eq:costAlgo2}
\Ocal\big(\calN(X_N,r/2)(r/2)^nN\log^2 N\big)=\Ocal\big(N\log^2 N\big)+\text{precomputation of $S_j$}
\end{equation}
for Algorithm~\ref{alg:greedycovering}.
For the
precise cost for the precomputation of the $S_j$ using the $\varepsilon$-nearest
neighbor search, we refer to \Cref{sec:epsnearest}.

Depending on the size of the covering number, the coarsening step may result
in a significant reduction of the computational cost for the evaluation of the
discrete modulus of continuity, if we are willing to invest a certain margin of
error. We recall that the value of the covering number depends on the geometry
of $(\calX,d_\calX)$, see \Cref{sec:coveringnumbers}. In the following, we
further decrease the computational cost by employing an iterative coarsening
procedure.

\subsection{Efficient approximation of the discrete modulus of continuity}
The algorithm we are going to introduce in a moment is based on the idea that
the (discrete) modulus of continuity is a monotonically increasing function with
$\omega(f,0)=0$. Since every modulus of continuity has a concave majorant, the
values of the modulus of continuity close to $t=0$ are of particular importance
for a good approximation. This is exactly the regime where the
$\varepsilon$-nearest neighbor search allows us to efficiently evaluate the
discrete modulus of continuity exactly. In contrast, with increasing distance
from $t=0$, the modulus of continuity remains either constant or is increasing.
If the modulus of continuity is constant, then its value has already been
captured for a smaller value of $t$ and a lower computational cost. If it is
increasing, then we transition to a coarser covering, resulting in less data points
to consider and thus lower computational cost with an approximation error
that is dependent on the behavior of the modulus of continuity.

In the following, in \Cref{alg:initialization,alg:evaluation}, we provide an
approach, which takes the preceding considerations into account. 
\begin{algorithm}[htb]
\caption{\label{alg:initialization}Offline phase for the efficient evaluation
of the modulus of continuity.}
\begin{algorithmic}
\State \textbf{Input:} Set of data sites $X_N$, initial radius $r>0$, growth
factor $R>1$, upper bound $T\leq T_\calX$
\State \textbf{Output:} $\omega_{N_k}(Y_N,R^k r)$, $k=0,\ldots,K-1$,
\State\phantom{\textbf{Output:}} sets
$X_{N_0}\supsetneq X_{N_1}\supsetneq\ldots\supsetneq X_{N_{K-1}}$
with $|X_{N_k}|=N_k$, $k=0,\ldots,K-1$
\State Set $X_{N_0}=X_N$, $K=\lceil\log_R(T/r)\rceil$
\State \emph{Optional: Set} $X_{\min,\max}=\{\bfx_i,\bfx_j\}$ \emph{with}
$(i,j)\in\argmax\big\{(i,j)\in\{1,\ldots,N\}^2\colon d_\calY(\bfy_i,\bfy_j)
\big\}$
\For{$k=0,\ldots,K-1$}
  \State \emph{Optional: Set} $X_{N_{k}}=X_{N_{k}}\cup X_{\min,\max}$
	\State Compute $\omega_{N_k}(Y_{N_k},R^kr)$ using an $\varepsilon$-nearest
  neighbor search with $\varepsilon=R^kr$
	\State Find an $R^{k}r$-cover $X_{N_{k+1}}$ of $X_{N_{k}}$ using
  \Cref{alg:greedycovering}
\EndFor
\State \Return $X_{N_k}$, $\omega_{N_k}(Y_{N_k},R^kr)$, $k=0,\ldots,K-1$
\end{algorithmic}%
\end{algorithm}
\begin{algorithm}[htb]
\caption{\label{alg:evaluation}Online phase for the efficient evaluation of
the modulus of continuity.}
\begin{algorithmic}
\State \textbf{Input:} Evaluation point $t\in [0,T]$, inputs and outputs from
\Cref{alg:initialization}
\State \textbf{Output:} Approximate value of $\omega_N(Y_N,t)$
\If{$t\in[0,r]$}
\State Compute $\omega_{N_0}(Y_N,t)$ on $X_{N_0}$ using an 
$\varepsilon$-nearest neighbor search with $\varepsilon=t$
\State \Return $\omega_{N_0}(Y_{N_0},t)=\omega_N(Y_N,t)$
\Else
\State Let $k\in\bbN$ such that $t\in(R^{k-1}r,R^kr]$
\State Compute $\omega_{N_k}(Y_{N_k},t)$ on $X_{N_k}$ using an $\varepsilon$-nearest
neighbor search with $\varepsilon=t$
\State \Return $\max\big\{\omega_{N_k}(Y_{N_k},t),\max_{\ell=0,\ldots,k-1}
\{\omega_{N_{\ell}}(Y_{N_{\ell}},R^{\ell}r)\}\big\}\approx\omega_N(Y_N,t)$
\EndIf
\end{algorithmic} 
\end{algorithm}%
We have divided it into an offline-online approach to allow for the repeated
evaluation for various values of $t$ with minimal computational overhead.
It is clear that the two algorithms can be merged if only a single approximate
evaluation of the discrete modulus of continuity is required. The optional step
of computing and
including $X_{\min,\max}$ can be expensive, depending on the structure of $Y_N$,
and is not strictly necessary. However, if it is efficiently computable, for
example, if $Y_N\subset\bbR$, then it can have a large positive effect on the
approximation quality. The reason for this is that \Cref{alg:evaluation} only
relies on coarsened sets for evaluating $\omega(Y_N,t)$ for larger values of
$t$, where only larger values of $d_\calY(\bfy_i,\bfy_j)$ are relevant. By
incorporating the data sites $X_{\min,\max}$ causing extremal function values,
we more actually capture the values of $d_\calY(\bfy_i,\bfy_j)$ on the coarsened
grids, thus improving the approximation.

To avoid technicalities in the following result on the cost, we assume without
loss of generality that $T_\calX\leq 1$, that is, the aspect ratio satisfies
\(\Delta=T_\calX/q_{X_N}\leq q_{X_N}^{-1}\). Otherwise, the result holds by a
suitable rescaling of the involved quantities.

\begin{theorem}[Computational cost of \Cref{alg:initialization}]
Assume that $(\calX,d_\calX)$ is an Ahlfors $n$-regular metric space 
and let the set of data sites $X_N$ be quasi-uniform.
Then, for $R\geq 1$ and \(r\approx q_{X_N}\approx N^{-1/n}\),
the cost of \Cref{alg:initialization} is
\[
\Ocal\big(R^{n(4-\lceil\log_R(T_\Xcal\cdot N^{1/n})\rceil)}N\log^3 N\big).
\]
with a constant independent of $r,R$, and $N$.
\end{theorem}
\begin{proof}
We first recall that, in an Ahlfors $n$-regular metric space, the
$\varepsilon$-nearest neighbor search for all points can be performed with cost
$\calO\big(N(\log N+\varepsilon^nN)\big)$, see \Cref{rem:nncomplexity}.
For each fixed point, the cost for computing the modulus of continuity is $\calO(\varepsilon^n N)$.
Therefore, the cost of computing $\omega_N(Y_N,\varepsilon)$ based on the
$\varepsilon$-nearest neighbor search is of order
\begin{equation*}%
  \calO\big(N(\log N+\varepsilon^nN)\big).
\end{equation*}

Furthermore, Ahlfors $n$-regularity implies Assouad dimension $n$. Thus, the
$\log(N)$-optimality of \Cref{alg:greedycovering} and \Cref{eq:assouadcover} imply
that finding an $\varepsilon$-covering entails a cost of order
\begin{equation*}\label{eq:firstest}
\calO\big(N\log^2 N + N(\log N+\varepsilon^nN)\big).
\end{equation*}
see \Cref{eq:costAlgo2}.
Combining these two costs yields an overall cost of
\begin{equation}\label{eq:firstiterationsearch}
  \calO\big(N(\log^2 N+\varepsilon^nN)\big).
\end{equation}
for the computation of $\omega_N(Y_N,\varepsilon)$ and the computation of an $\varepsilon$-covering of \(X_{N}\).
The cost of \Cref{alg:initialization} is obtained by
setting \(\varepsilon=R^kr\) for \(k=0,\ldots, K-1\) and summing up the resulting cost per
iteration. We obtain
\[
\Ocal\bigg(\sum_{k=0}^{K-1}
N_k\big(\log^2 N_k+(R^kr)^nN_k\big)\bigg).
\]
Furthermore, exploiting the 
$\log(N)$-optimality of \Cref{alg:greedycovering} and \Cref{eq:assouadcover} for \(k\geq 1\) that
\[
N_k\leq\log(N_{k-1})\calN(\calX,R^{k-1}r)=\Ocal\big((\log N_{k-1})(R^{k-1}r)^{-n}\big)
=\Ocal\big((\log N)(R^{k-1}r)^{-n}\big).
\]
We arrive at
\begin{align*}
&\Ocal\bigg(\sum_{k=0}^{K-1}
N_k\big(\log^2 N_k+(R^kr)^nN_k\big)\bigg)\\
&\qquad=
\Ocal\Big(
N\big(\log^2 N+r^nN\big)\Big)+\Ocal\bigg(\sum_{k=1}^{K-1}
(\log N)(R^{k-1}r)^{-n}\big(\log^2N+R^n\log N\big)\bigg)\\
&\qquad=
\Ocal\Big(
N\big(\log^2 N+r^nN\big)\Big)+
\Ocal\bigg((\log^3N)r^{-n}R^{2n}\sum_{k=0}^{K-2}R^{-nk}\bigg)\\
&\qquad=
\Ocal\Big(
N\big(\log^2 N+r^nN\big)\Big)
+
\Ocal\big((\log^3N)r^{-n}R^{n(4-K)}\big).
\end{align*}
Inserting $r=N^{-1/n}$ and $K=\lceil\log_R(T_\calX/r)\rceil$ yields the assertion.
\end{proof}

Based on the estimates in the proof of the previous theorem, we find the following
bound on the cost of \Cref{alg:evaluation}.
\begin{corollary}[Computational cost of \Cref{alg:evaluation}]
Assume that $(\calX,d_\calX)$ is an Ahlfors $n$-regular metric space 
and let the set of data sites $X_N$ be quasi-uniform.
Then, for $R\geq 1$ and \(r\approx q_{X_N}\approx N^{-1/n}\),
the cost of \Cref{alg:evaluation} is
\[
\begin{cases}
\calO\big(N(\log N+r^nN)\big),&t\leq r,\\
\calO\big(R^{3n}r^{-n}\log^3N\big),&t>r,
\end{cases}
\]
with constants independent of $r,R$, and $N$.
Inserting \(r=N^{-1/n}\),
the bound becomes 
\[
\Ocal\big(R^{3n}N\log^3 N\big)\quad\text{for all }0\leq t\leq T_\Xcal.
\]
\end{corollary}

\begin{remark}
The main ingredient for the fast evaluation of the modulus of continuity using
the algorithms in this section is the coarsening of the set of data sites.
The proposed coarsening strategy is independent of the values of the
site-to-value map. Hence, even though we will see favorable numerical
experiments for \Cref{alg:evaluation} in \Cref{sec:Numerics}, we emphasize that the algorithm is
is not compliant with the consistency results in
\Cref{sec:detcons,sec:Consistency}.
\end{remark}

\section{Piecewise constant approximation of discrete data}
\label{sec:Approximation}
With a solid understanding of the discrete modulus of continuity available, we
are now in the position to derive approximation results for the piecewise
constant interpolation of site-to-value maps. The cost for computing such an
interpolant scales linearly in the number of piecewise constant basis functions
and is, for a given partition of \(X_N\), independent of the number of data
sites $N$.

\subsection{Piecewise constant interpolation}
Let a data set of the form from \Cref{eq:datapoints} be given. 
To introduce piecewise
constant functions on $\Xcal$, where we may set $\calX=X_N$
in practical applications, we introduce the partition
\[
\Xcal=\bigcup_{i=1}^MX^{(i)},
\qquad
X^{(1)},\ldots, X^{(M)}\subset \Xcal~\text{mutually disjoint}.
\]
For each set \(X^{(i)}\), we select an interpolation point
$\hat\bfx_i\in X^{(i)}$, $i=1,\ldots,M$, and define
\[
\Xi_M\isdef\{(\hat\bfx_i,X_i)\}_{i=1}^M.
\]
Given a continuous function
\(f\colon\Xcal\to\Ycal\), we introduce the interpolation operator
\begin{equation}\label{eq:InterpolationOp}
f
\approx
I_{\Xi_M} f
\isdef
\sum_{i=1}^M f(\hat\bfx_i)\mathbbm{1}_{X^{(i)}},
\end{equation}
where $\mathbbm{1}_{X^{(i)}}$ denotes the characteristic function of the set
$X^{(i)}$.

The following result is a straightforward extension of
\cite[Section 1.1.2]{Dyn}.
\begin{lemma}[Piecewise constant interpolation on compact metric spaces]%
\label{lem:pwconstantapprox}
Let \(f\colon\Xcal\to\Ycal\) be of class $\rho\geq C\omega(f,\cdot)$ for some
constant \(C>0\).
Let \(h>0\) denote the \emph{mesh size} such that
\(
\diam(X^{(i)})\leq h\) for $i=1,\ldots,M$. Then, the interpolation error 
for $I_{\Xi_M}f$ from \Cref{eq:InterpolationOp} satisfies
\[
\sup_{{\bs x}\in\Xcal}d_\Ycal(f,I_{\Xi_M}f)
\leq
\rho(h)|f|_{\rho}.
\]
For the specific choice $\rho=\omega(f,\cdot)$, we obtain
\[
\sup_{{\bs x}\in\Xcal}d_\Ycal(f,I_{\Xi_M}f)
\leq
\omega(f,h).
\]
\end{lemma}
\begin{proof}
There holds
\begin{align*}
|f|_{\rho}&=\sup_{t>0}\big(\rho(t)^{-1}\omega(f,t)\big)
\geq \rho(h)^{-1}\omega(f,h)=
\rho(h)^{-1}\sup_{\gfrac{\bfx,\bfx'\in\Xcal:}{d_{\Xcal}(\bfx,\bfx')\leq h}}
d_{\Ycal}\big(f(\bfx),f(\bfx')\big)\\
&\geq \rho(h)^{-1}\sup_{{\bs x}\in\Xcal}d_\Ycal(f,I_{\Xi_M}f). 
\end{align*}
since each \({\bs x}\in X^{(i)}\) satisfies 
\(d_\Xcal({\bs x},{\bs x}_i)\leq h\). 
Multiplying the inequality by \(\rho(h)\) yields the assertion.
\end{proof}
\begin{remark}
It is easy to see that the cost of applying the interpolation operator from
\Cref{eq:InterpolationOp} is $\calO(M)$, independently of the cardinality of
$\calX$. Moreover, the approximation estimate from \Cref{lem:pwconstantapprox}
is independent of the chosen interpolation points $\{\hat{\bfx}_i\}_{i=1}^M$.
Thus, we may choose these points randomly within their containing sets
$X^{(i)}$. This implies that the interpolation operator itself can be
constructed in $\calO(M)$, if the partition $\{X^{(i)}\}_{i=1}^M$ is already
available.
\end{remark}

\begin{remark}For any mesh size \(h>0\), there holds 
\[
\Xcal\subset\bigcup_{\bfx\in\Xcal}
\operatorname{int}\big({B}_{h/2}({\bs x})\big),
\]
where \(\operatorname{int}\big({B}_{h/2}({\bs x})\big)\) is the interior of
\({B}_{h/2}({\bs x})\). Since \(\Xcal\) is compact, we may always select \(M\)
balls with corresponding centers such that
\[
\Xcal\subset\bigcup_{i=1,\ldots,M}
\operatorname{int}\big({B}_{h/2}(\hat{\bs x}_i)\big)
\subset\bigcup_{i=1,\ldots,M}{B}_{h/2}(\hat{\bs x}_i).
\]
Setting 
\[
X^{(i)}\isdef\bigg\{{\bs x}\in\Xcal : i 
= \min\Big\{k: d_\Xcal(\hat{\bfx}_k,\bfx)=\min_{\ell=1,\ldots,M}
d_\Xcal(\hat{\bfx}_\ell,\bfx)\Big\}\bigg\}
\]
then yields a suitable Voronoi-type partition with cell diameter
\(\diam(X^{(i)})\leq h\) for $i=1,\ldots,M$.
\end{remark}

Assuming that \(\hat{\bfx}_1,\ldots,\hat{\bfx}_M\subset X_N\), the previous
lemma yields an analogous interpolation result for site-to-value maps.
\begin{corollary}[Piecewise constant interpolation on data sets]%
\label{lem:pwconstantapproxstv}
Let \(f_N\colon X_N\to Y_N\) be of class $\rho\geq C\omega(f,\cdot)$
for some $C>0$.
Assume that there exists \(h>0\) with
\(
\diam(X^{(i)})\leq h\), i=1,\ldots,M. Then the interpolation error 
for $I_{\Xi_M}f_N$ from \Cref{eq:InterpolationOp} satisfies
\[
\sup_{{\bs x}\in\Xcal}d_\Ycal(f_N,I_{\Xi_M}f_N)
\leq
\rho(h)|f_N|_{\rho,N}.
\]
For the specific choice $\rho=\omega_N(f,\cdot)$ we particularly obtain
\begin{equation}\label{eq:interapproxdiscmodcont}
\sup_{{\bs x}\in\Xcal}d_\Ycal(f_N,I_{\Xi_M}f_N)
\leq
\omega_N(f_N,h).
\end{equation}
\end{corollary}

The upper bound based on the discrete modulus of continuity is fully discrete,
compare also the definition of the discrete modulus of continuity from
\Cref{eq:discmodcont,eq:discretemodulusdata}. It is therefore in principle
computable. For large data sets it can be estimated using the algorithms from
\Cref{sec:Computation}.

\begin{remark}
It is tempting to conjecture that a function of class $\rho$ with $\rho\in o(t)$
as $t\to 0$ leads to a higher-order approximation. While true from an
approximation theoretic perspective, these kind of function classes are rather
restrictive in general. For example, it is well known that H\"older continuous
functions with $\alpha>1$, that is, $\rho(t)=t^\alpha$, on connected domains in
$\bbR^d$ are necessarily constant.
\end{remark}

Under the additional assumption that $(\calY,\|\cdot\|_{\Ycal})$ is a normed space,
the results in \Cref{lem:pwconstantapprox} and
\Cref{lem:pwconstantapproxstv}, respectively, can be rephrased 
with respect to the usual \(\sup\)-norm
\[
\|f\|_{C(\Xcal;\Ycal)}\isdef\sup_{{\bs x}\in\Xcal}\|f({\bs x})\|_\Ycal.
\]
We further observe that the interpolation operator \eqref{eq:InterpolationOp} is
stable in this case, that is,
\begin{equation}\label{eq:StabInterp}
\|I_{\Xi_M}f\|_{C(\Xcal;\Ycal)}
\leq
\|f\|_{C(\Xcal;\Ycal)}.
\end{equation}
\subsection{Interpolation in H\"older spaces}\label{sec:hoelderinterpolation}
If $(\calY,\|\cdot\|_{\Ycal})$ is a normed space, the
previous approximation results relate to more conventional approximation
results for H\"older continuous functions. 

\begin{corollary}[Piecewise constant interpolation in H\"older spaces]%
\label{cor:HoelderApprox}
Let \(\Ycal\) be a normed vector space. Let \(f\colon\Xcal\to\Ycal\) satisfy 
$|f|_{\operatorname{Lip}(\alpha)}\isdef|f|_{t^\alpha}<\infty$ for some
$\alpha\in(0,1)$. Assume that there exists a mesh size \(h>0\) with
\(
\diam(X^{(i)})\leq h\) for $i=1,\ldots,M$. Then, there holds
\begin{equation}
\label{eq::ApproxEst}
\|f-I_{\Xi_M}f\|_{C(\Xcal;\Ycal)}
\leq
h^\alpha|f|_{\operatorname{Lip}(\alpha)}.
\end{equation}
\end{corollary}

As before, we have an analogous result for site-to-value maps, which we
state for completeness.
\begin{corollary}[Piecewise constant interpolation in H\"older %
spaces on data sets]
Let \(\Ycal\) be a normed vector space. Under the assumptions of
\Cref{lem:pwconstantapproxstv}, there holds with
$|f_N|_{\operatorname{Lip}(\alpha),N}\isdef|f_N|_{t^\alpha,N}$ that
\[
\|f_N-I_{\Xi_M}f_N\|_{C(X_N;\Ycal)}
\leq
h^\alpha|f_N|_{\operatorname{Lip}(\alpha),N}.
\]
\end{corollary}

\begin{remark}
Obviously, for \(N<\infty\), the semi-norm 
\(|f_N|_{\operatorname{Lip}(\alpha),N}\) is always bounded. This implies that
any site-to-value map is H\"older continuous for any H\"older exponent
\(\alpha\leq 1\). Hence, to distinguish between different  H\"older
regularity, it is suggested in \cite{GNC10} to introduce an explicit constant
\(C>0\) and to define the discrete H\"older class of functions with
\(|f_N|_{\operatorname{Lip}(\alpha),N}\leq C\). This additional assumption
can be dropped, if the limit \(f_N\to f\) with
\(|f|_{\operatorname{Lip}(\alpha)}<\infty\)
exists, see \Cref{thm:seminormconv,thm:empiricalseminormconsistency}.
\end{remark}

\section{Multilevel Monte Carlo for discrete data}\label{sec:MLMC}
As a relevant application for the piecewise constant approximation of discrete
data, we consider the computation of first and second order statistics of
independent samples of identically distributed random vectors.
To this end, let $Y\colon\calX\times\Omega\to\calY$ be a random field
taking values in a real, separable Hilbert space $\big(\calY,(\cdot,\cdot)_\Ycal\big)$, where  
$(\Omega,\Fcal,\bbQ)$ is a complete probability space.
In addition, we assume that \(\Xcal\) is equipped with a
probability measure.

Sampling \(Y\) at certain locations \(X_N=\{\bfx_1,\ldots,\bfx_N\}\)
amounts to a sample of random vectors
\begin{equation}\label{eq:randVec}
\bs y(\omega)=[Y({\bs x},\omega)]_{{\bs x}\in X_N}=Y_N(\omega)\in\calY^N.
\end{equation}
Given samples \({\bs y}_1,\ldots,{\bs y}_Q\) of \(\bs y\) we aim at 
efficiently computing the 
\emph{sample mean}
\begin{equation}\label{eq:sampleMean}
\overline{\bs y}\isdef\frac 1 Q\sum_{k=1}^Q\bs y_k,
\end{equation}
the \emph{sample covariance}
\begin{equation}\label{eq:sampleCov}
  \overline{\bfC}^{\textnormal{(Cov)}}
  \isdef\frac{1}{Q-1}\sum_{k=1}^Q(\bfy_k-\overline\bfy)
(\bfy_k-\overline\bfy)^\intercal
=\frac{1}{Q-1}\bigg(\bfY\bfY^\intercal
-Q\overline\bfy\,\overline\bfy^\intercal\bigg),
\end{equation}
and the \emph{sample correlation}
\begin{equation}\label{eq:sampleCor}
  \overline{\bfC}^{\textnormal{(Cor)}}
  \isdef\frac{1}{Q}\sum_{k=1}^Q\bfy_k
\bfy_k^\intercal
=\frac{1}{Q}\bfY\bfY^\intercal,
\end{equation}
with the data matrix
\[
\bfY\isdef[\bfy_1,\ldots,\bfy_Q]\in\calY^{N\times Q}.
\]
Clearly, the sample covariance can be deduced from the sample mean and the
sample correlation, such that we restrict our discussion to the latter two
quantities.

Computing the sample mean \eqref{eq:sampleMean} naively is of cost
\(\Ocal(QN)\), while
forming the product \(\bfY\bfY^\intercal\) in \eqref{eq:sampleCor}
using the singular value decomposition of \(\bfY\) is of cost 
\(\min\{N^2Q,NQ^2\}\) and, hence, only efficient if either \(Q\ll N\)
or \(N\ll Q\). 
For \(Q\approx N\), \eqref{eq:sampleMean} scales quadratically
in the number of data sites, while \eqref{eq:sampleCor} scales
cubically in the number of data sites. To mitigate this cost, we assume
that $Y$ is contained pathwise in the Banach space
\[
C_\rho(\calX,\calY)=\{f\in C(\calX,\calY)\colon |f|_\rho<\infty\}
\]
which is equipped with the norm
\[
\|f\|_{C_\rho}(\calX,\calY)
=
\|f\|_{C(\calX,\calY)}+|f|_\rho.
\]
Similar estimates in terms of approximation and stability as in
\Cref{sec:hoelderinterpolation} can be stated. A suitable choice for $\rho$ is
the pointwise essential supremum over all pathwise moduli of continuity, that is,
\[
\rho(t)
=
\big\|\omega_N\big(Y_N(\cdot),t\big)\big\|_{L^\infty(\Omega)}.
\]

\subsection{Piecewise constant approximation of empirical statistics}
We assume that \(Y\) is contained in the 
\emph{Lebesgue-Bochner space} 
\(L^p\big(\Omega;C_\rho(\Xcal;\Ycal)\big)\), \(p\geq 1\). 
We recall that, given a Banach space \((\Ecal,\|\cdot\|_{\Ecal})\), 
the Lebesgue-Bochner space \(L^p(\Omega;\Ecal)\)
consists of all equivalence classes
of strongly \(\Qbb\)-measurable maps 
\(f\colon\Omega\to\Ecal\) with
finite norm
\[
\|f\|_{L^p(\Omega;\Ecal)}
\isdef\begin{cases}
\displaystyle{\bigg(\int_\Omega\|f\|_{\calE}^p\d\Qbb\bigg)^{\frac 1 p}},& 
p < \infty,\\
\displaystyle{\esssup_{\omega\in\Omega}\|f\|_{\calE}},& p =\infty.
\end{cases}
\]

In the following, we specifically assume that $Y\in L^2\big(\Omega;C_\rho(\calX;\calY)\big)$.
This particularly implies that $Y\in L^2\big(\Omega;L^2(\calX;\calY)\big)$.
Then, given an independent sample \(Y_1,\ldots, Y_Q\) of \(Y\), 
the sample mean
\[
E_Q[Y]\isdef\frac 1 Q\sum_{k=1}^Q Y_k
\]
satisfies
\begin{equation}\label{eq:MCestimate}
\big\|\Ebb[Y]-E_Q[Y]\big\|_{L^2(\Omega;L^2(\calX;\Ycal))}^2
\leq\frac{1}{Q}\|Y\|_{L^2(\Omega;L^2(\calX;\Ycal))}^2,
\end{equation}
see \cite{BSZ11} and also \cite[Chapter 9]{TL91}.
Furthermore, it is easy to see that
\begin{align*}
\|Y\|_{L^2(\Omega;L^2(\calX;\Ycal))}^2
&=\int_\Omega\|Y\|_{L^2(\calX;\Ycal)}^2\d\Qbb\\
&\leq\int_\Omega\|Y\|_{C(\calX;\Ycal)}^2\d\Qbb
=\|Y\|_{L^2(\Omega;C(\calX;\Ycal))}^2
\leq\|Y\|_{L^2(\Omega;C_\rho(\calX;\Ycal))}^2.
\end{align*}
Combining the previous estimates with \eqref{eq:MCestimate} and taking square 
roots yields
\begin{equation}
\big\|\Ebb[Y]-E_Q[Y]\big\|_{L^2(\Omega;L^2(\calX;\calY))}\leq\frac{1}{\sqrt{Q}}
\|Y\|_{L^2(\Omega;C_\rho(\calX;\calY))}.
\end{equation}

In view of Lemma~\ref{lem:pwconstantapprox}, given a partition
\(X^{(1)},\ldots,X^{(M)}\) with \(\diam(X^{(i)})\leq h\), we 
can further estimate
\begin{align*}
&\big\|\Ebb[Y]-E_Q[I_{\Xi_M} Y]\big\|_{L^2(\Omega;L^2(\calX;\Ycal))}\\
&\qquad\leq\big\|\Ebb[Y]-\Ebb[I_{\Xi_M}Y]\big\|_{L^2(\Omega;L^2(\calX;\Ycal))}
+\big\|\Ebb[I_{\Xi_M}Y]-E_Q[I_{\Xi_M} Y]\big\|_{L^2(\Omega;L^2(\calX;\Ycal))}\\
&\qquad\leq\|Y-I_{\Xi_M} Y\|_{L^2(\Omega;C(\calX;\Ycal))}
+\frac{1}{\sqrt{Q}}\|I_{\Xi_M}Y\|_{L^2(\Omega;L^2(\calX;\Ycal))}\\
&\qquad\leq\bigg(\rho(h)
+\frac{1}{\sqrt{Q}}\bigg)\|Y\|_{L^2(\Omega;C_\rho(\calX;\Ycal))}
\end{align*}
by invoking Bochner's inequality
and the stability estimate \eqref{eq:StabInterp}.

Moreover, since \(Y\in L^2\big(\Omega;C_\rho(\calX;\calY)\big)\), the
correlation 
\[
\Cor[Y]({\bs x},{\bs x}')=\int_\Omega Y({\bs x},\omega)Y({\bs x}',\omega)\d\Qbb
\]
exists and satisfies
\(\Cor[Y]\in C_\rho(\Xcal;\Ycal)\otimes C_\rho(\Xcal;\Ycal)\) for 
both, the injective and the projective tensor product, see
\cite[Chapter 3]{JK15}. Similarly to the case of the expectation, we find
\begin{align*}
&\big\|\Ebb[Y\otimes Y]-E_Q[I_{\Xi_M}Y\otimes I_{\Xi_M}Y]
\big\|_{L^2(\Omega;L^2(\Xcal;\Ycal)\otimes L^2(\Xcal;\Ycal))}\\
&\qquad\leq\|Y\otimes Y-I_{\Xi_M}Y
\otimes I_{\Xi_M}Y\|_{L^2(\Omega;C(\calX;\Ycal))}
+\frac{1}{\sqrt{Q}}\|I_{\Xi_M}Y
\otimes I_{\Xi_M}Y\|_{L^2(\Omega;L^2(\calX;\Ycal))}\\
&\qquad\leq\bigg(2\rho(h)
+\frac{1}{\sqrt{Q}}\bigg)
\|Y\|_{L^2(\Omega;C_\rho(\Xcal;\Ycal)\otimes C_\rho(\Xcal;\Ycal))}
\end{align*}
by standard tensor product arguments.

Fixing a set of data sites 
\(X_N=\{\bfx_1,\ldots,\bfx_N\}\subset\calX\), the 
corresponding random vector \({\bs y}(\omega)\), see 
\eqref{eq:randVec}, satisfies
\[
\bfy
\in
L^2\big(\Omega;C_\rho(X_N;\calY)\big)
\subset L^2\big(\Omega;L^2(X_N;\calY)\big).
\]
The previous derivation directly carries over to this case
and results in the error estimates
\[
\big\|\Ebb[{\bs y}]-E_Q[I_{\Xi_M}{\bs y}]\big\|_{L^2(\Omega;L^2(X_N;\Ycal))}
\leq\bigg(\rho(h)
+\frac{1}{\sqrt{Q}}\bigg)\|{\bs y}\|_{L^2(\Omega;C_\rho(X_N;\Ycal))}
\]
and
\begin{align*}
&\big\|\Ebb[{\bs y}{\bs y}^\intercal]
-E_Q[(I_{\Xi_M}{\bs y})(I_{\Xi_M}{\bs y})^\intercal]
\big\|_{L^2(\Omega;L^2(X_N;\Ycal)\otimes L^2(X_N;\Ycal))}\\
&\qquad\leq\bigg(2\rho(h)
+\frac{1}{\sqrt{Q}}\bigg)
\|{\bs y}\|_{L^2(\Omega;C_\rho(X_N;\Ycal)
\otimes C_\rho(X_N;\Ycal))},
\end{align*}
respectively. 

It is easy to see that the error is equilibrated if $\rho(h)\sim Q^{-1/2}$. Even
for the regular case of Lipschitz continuous functions, that is, $\rho(t)=t$,
and quasi-uniform spatial points in $\calX=\bbR^d$, that is, $h\sim N^{-1/d}$,
this implies $N\sim Q^{d/2}$ and results in a cost of $\calO(Q^{1+d/2})$ for the
computation of \(E_Q[I_{\Xi_M}{\bs y}]\) and in a cost of
\(\Ocal\big(Q^{1+d}\big)\) for the computation of
\(E_Q[(I_{\Xi_M}{\bs y})(I_{\Xi_M}{\bs y})^\intercal]\), which is intractable.
To significantly reduce this cost, we shall employ sparse tensor product
techniques, that is, the multilevel Monte Carlo method for the sample mean and
the multiindex Monte Carlo method for the correlation.

\subsection{Multilevel Monte Carlo method}
In this section, we propose a multilevel Monte Carlo approach,
see \cite{Gil15} and the references therein, for discrete data. 
Given a uniform binary tree \(\Tcal\) for the set \(X_N\subset\Xcal\), we
consider the partitions
\[
\Xi_j\isdef\big\{({\bs x}_\tau,\tau)\colon\tau\in\calT,\ell(\tau)=j\big\},
\quad\text{for }j=0,\ldots,J,
\]
where the interpolation points \({\bs x}_\tau\in\tau\) may be chosen
arbitrarily. A nested sequence of interpolation points can be obtained by the
choice \({\bs x}_\tau\in\{{\bs x}_{\tau'}:\tau'\in\children(\tau)\}\)
for \(\tau\in\Tcal\setminus\Lcal\).

We have the following lemma, which bounds the pointwise difference of two
consecutive interpolants.
The result directly carries over to site-to-value maps in 
\(C_\rho(X_N;\Ycal)\).

\begin{lemma}\label{lem:scaledec}
There holds for any \(f\in C_\rho(\Xcal;\Ycal)\) that
\begin{equation}\label{eq:varianceReduction}
\|I_{\Xi_j}f-I_{\Xi_{j-1}}f\|_{C(\Xcal;\Ycal)}
\leq
2\rho(\cdiam2^{-\cunif (j-1)})\|f\|_{C_\rho(\Xcal;\Ycal)},
\end{equation}
where the constants \(\cdiam,\cunif>0\) depend on the 
uniformity from Definition~\ref{def:uniformTree}.
\end{lemma}
\begin{proof}
By the triangle inequality, we have
\[
\begin{aligned}
\|I_{\Xi_j}f-I_{\Xi_{j-1}}f\|_{C(\Xcal;\Ycal)}
&\leq
\|I_{\Xi_j}f-f\|_{C(\Xcal;\Ycal)}
+\|I_{\Xi_{j-1}}f-f\|_{C(\Xcal;\Ycal)}\\
&\leq
\Big(\rho(\cdiam2^{-\cunif j})
+\rho(\cdiam2^{-\cunif (j-1)})\Big)\|f\|_{C\rho(\Xcal;\Ycal)}\\
&\leq
2\rho(\cdiam2^{-\cunif (j-1)})
\|f\|_{C_\rho(\Xcal;\Ycal)}
\end{aligned}
\]
due to \(\diam(\tau)\leq\cdiam 2^{-\cunif j}\) and using
\Cref{lem:pwconstantapprox}. 
\end{proof}

The decay across scales, shown in the previous lemma, corresponds to a 
level-wise variance reduction in case of random fields. This
stipulates the idea of the multilevel Monte Carlo estimator.
Given a sequence \(Q_0\geq\ldots\geq Q_J\) of natural numbers, 
the multilevel Monte Carlo estimator reads
\[
E^{\textrm{ML}}[Y]\isdef\sum_{j=0}^J E_{Q_{J-j}}[(I_{\Xi_j}-I_{\Xi_{j-1}})Y],
\quad\text{where }I_{\Xi_{-1}}Y\equiv 0.
\]

The following theorem is a well established result on the multilevel Monte Carlo
estimator. We provide its proof here adapted to the present context
for the readers convenience.

\begin{theorem}\label{thm:MLMC}
Let \(Y\in L^2\big(\Omega;C_\rho(\Xcal;\Ycal)\big)\). Then there holds
\[
\big\|\Ebb[Y]-E^{\textnormal{ML}}[Y]\big\|_{L^2(\Omega;L^2(\Xcal;\Ycal))}
\leq
\sigma_{\textnormal{ML}}(J,\rho,\{Q_j\}_j)\|Y\|_{L^2(\Omega;C_\rho(\Xcal;\Ycal))}
\]
with the convergence factor
\begin{equation}\label{eq:ConvFactorML}
\sigma_{\textnormal{ML}}(J,\rho,\{Q_j\}_j)
\isdef
\rho(\cdiam2^{-\cunif J})+2\sum_{j=0}^J
\frac{\rho(\cdiam2^{-\cunif (j-1)})}{\sqrt{Q_{J-j}}}.
\end{equation}
\end{theorem}
\begin{proof}
By the triangle inequality, we obtain
\begin{align*}
&\big\|\Ebb[Y]-E^{\textrm{ML}}[Y]\big\|_{L^2(\Omega;L^2(\Xcal;\Ycal))}\\
&\qquad\leq
\big\|\Ebb[Y]-\Ebb[I_{\Xi_J}Y]\big\|_{L^2(\Omega;L^2(\Xcal;\Ycal))}+
\big\|\Ebb[I_{\Xi_J}Y]-E^{\textrm{ML}}[Y]\big\|_{L^2(\Omega;L^2(\Xcal;\Ycal))}.
\end{align*}
The first term on the left-hand side can be estimated by
Lemma~\ref{lem:pwconstantapprox} and Bochner's inequality. For the second term,
we obtain by the linearity of expectation, the fact that
\(I_{\Xi_J}Y=\sum_{j=0}^J(I_{\Xi_j}-I_{\Xi_{j-1}})Y\)
and the triangle inequality that
\begin{align*}
&\big\|\Ebb[I_{\Xi_J}Y]
-E^{\textrm{ML}}[Y]]\big\|_{L^2(\Omega;L^2(\Xcal;\Ycal))}\\
&\qquad=
\bigg\|\sum_{j=0}^J\Ebb[(I_{\Xi_j}-I_{\Xi_{j-1}})Y]
-E_{Q_{J-j}}[(I_{\Xi_j}-I_{\Xi_{j-1}})Y]\bigg\|_{L^2(\Omega;L^2(\Xcal;\Ycal))}\\
&\qquad\leq\sum_{j=0}^J
\big\|\Ebb[(I_{\Xi_j}-I_{\Xi_{j-1}})Y]
-E_{Q_{J-j}}[(I_{\Xi_j}-I_{\Xi_{j-1}})Y]\big\|_{L^2(\Omega;L^2(\Xcal;\Ycal))}.
\end{align*}
Employing \Cref{eq:MCestimate,eq:varianceReduction}, we obtain
for the terms within the sum that
\begin{align*}
&\big\|\Ebb[(I_{\Xi_j}-I_{\Xi_{j-1}})Y]
-E_{Q_{J-j}}[(I_{\Xi_j}-I_{\Xi_{j-1}})Y]\big\|_{L^2(\Omega;L^2(\Xcal;\Ycal))}\\
&\qquad\leq\frac{1}{\sqrt{Q_{J-j}}}
\|I_{\Xi_j}Y-I_{\Xi_{j-1}}Y\|_{L^2(\Omega;L^2(\Xcal;\Ycal))}
\leq 2\frac{\rho(\cdiam2^{-\cunif (j-1)})}{\sqrt{Q_{J-j}}}
\|Y\|_{L^2(\Omega;C_\rho(\Xcal;\Ycal))}.
\end{align*}
Combining the preceding estimates, we finally arrive at the error bound
\[
\big\|\Ebb[I_{\Xi_J}Y]-E^{\textrm{ML}}[Y]]\big\|_{L^2(\Omega;L^2(D))}
\leq
\bigg(
\rho(\cdiam2^{-\cunif J})
+2\sum_{j=0}^J\frac{\rho(\cdiam2^{-\cunif (j-1)})}{\sqrt{Q_{J-j}}}\bigg)
\|Y\|_{L^2(\Omega;C_\rho(\Xcal;\Ycal))}.
\]
\end{proof}

In principle, a similar estimator with similar error bound could
also be stated for the correlation. However, the computational
cost of the arising estimator will be dominated by the quadratic
cost of assembling the samples $I_{\Xi_k}Y\otimes I_{\Xi_k}Y$.
Instead, we observe that applying \eqref{eq:varianceReduction}
to two different levels of a dyad
\(f\otimes f\in C_\rho(\Xcal;\Ycal)\otimes C_\rho(\Xcal;\Ycal)\)
results in the bound
\begin{equation}\label{eq:TPestimate}
\begin{aligned}
&\big\|(I_{\Xi_j}-I_{\Xi_{j-1}})f\otimes (I_{\Xi_j'}-I_{\Xi_{j'-1}})f
\big\|_{C(\Xcal;\Ycal)\otimes C(\Xcal;\Ycal)}\\
&\qquad\qquad\qquad\leq
4\rho(\cdiam2^{-\cunif (j-1)})\rho(\cdiam2^{-\cunif (j'-1)})
\|f\otimes f
\|_{C_\rho(\Xcal;\Ycal)\otimes C_\rho(\Xcal;\Ycal)}.
\end{aligned}
\end{equation}

This suggests the multiindex or sparse tensor product estimator
\[
E^{\textrm{MI}}[Y\otimes Y]\isdef\sum_{j=0}^J\sum_{k+k'=j}E_{Q_{J-j}}[
(I_{\Xi_k}-I_{\Xi_{k-1}})Y\otimes (I_{\Xi_{k'}}-I_{\Xi_{k'-1}})Y]
\]
for the sample correlation. The important difference to the
multilevel estimator is that the tensor products are formed
between interpolants at different levels, which reduces the computational
footprint. In analogy to \Cref{thm:MLMC}, we obtain the following result.

\begin{theorem}\label{thm:MLMCCor}
Let \(Y\in L^2\big(\Omega;C_\rho(\Xcal;\Ycal)\big)\). Then, there holds
\begin{equation*}
\begin{aligned}
\big\|\Ebb[Y\otimes Y]-E^{\textnormal{MI}}[Y\otimes Y]
\big\|_{L^2(\Omega;L^2(\Xcal;\Ycal)\otimes L^2(\Xcal;\Ycal))}
\leq
\sigma_{\textnormal{MI}}(J,\rho,\{Q_j\}_j)\|Y\otimes Y
\|_{L^2(\Omega;C_\rho(\Xcal;\Ycal)\otimes C_\rho(\Xcal;\Ycal))}
\end{aligned}
\end{equation*}
with the convergence factor
\begin{equation}\label{eq:ConvFactorMI}
\sigma_{\textnormal{MI}}(J,\rho,\{Q_j\}_j)\isdef
2\rho(\cdiam2^{-\cunif J})+
4\sum_{j=0}^J\sum_{k+k'=j}\frac{\rho(\cdiam2^{-\cunif (k-1)})
\rho(\cdiam2^{-\cunif (k'-1)})}{\sqrt{Q_{J-j}}}.
\end{equation}
\end{theorem}
Replacing \(Y\) in Theorems \ref{thm:MLMC} and \ref{thm:MLMCCor} by its discrete
version \({\bs y}\) immediately yields the following corollary.
\begin{corollary}\label{cor:datamlmc}
	Let \({\bs y}\in L^2\big(\Omega;C_\rho(X_N;\calY)\big)\).
	Then, there holds
	\[
	\big\|\Ebb[{\bs y}]-E^{\textnormal{ML}}[{\bs y}]]
	\big\|_{L^2(\Omega;L^2(X_N;\Ycal))}
	\leq
	\sigma_{\textnormal{ML}}(J,\rho,\{Q_j\}_j)
	\|{\bs y}\|_{L^2(\Omega;C_\rho(X_N;\Ycal))}
	\]
	and
	\begin{equation*}
		\begin{aligned}
			\big\|\Ebb[{\bs y}{\bs y}^\intercal]
      -E^{\textnormal{MI}}[{\bs y}{\bs y}^\intercal]
			\big\|_{L^2(\Omega;L^2(X_N;\Ycal)\otimes L^2(X_N;\Ycal))}
			\leq \sigma_{\textnormal{MI}}(J,\rho,\{Q_j\}_j)
			\|{\bs y}{\bs y}^\intercal
      \|_{L^2(\Omega;C_\rho(X_N;\Ycal)\otimes C_\rho(X_N;\Ycal))}.
		\end{aligned}
	\end{equation*}
	respectively.
\end{corollary}
Assuming additional structure on $\rho$ allows for more concrete error estimates
and for bounds of the computational cost of the estimators.

\subsection{The multilevel Monte Carlo method for H\"older continuous functions}
Our initial observation is that for H\"older continuous functions with
H\"older exponent $\alpha\in (0,1]$, the error estimate \Cref{thm:MLMC}
coincides to what is known from literature. In particular, the convergence
factor \eqref{eq:ConvFactorML} for the multilevel Monte Carlo simplifies towards
\[
\sigma_{\textnormal{ML}}(J,t^\alpha,\{Q_j\}_j)
=\cdiam^\alpha\bigg(2^{-\alpha\cunif J}+
2^{\alpha\cunif+1}\sum_{j=0}^J\frac{2^{-\alpha\cunif j}}{\sqrt{Q_{J-j}}}\bigg).
\]

Based on the cost analysis from \cite{Gil15}, in order to achieve
an overall root mean square error of \(1/\sqrt{Q}\), there holds for the cost
\(C_{\textnormal{MLMC}}\)
of the multilevel Monte Carlo estimator that
\begin{equation*}
	C_{\textnormal{MLMC}}=\begin{cases} \Ocal(Q),&\text{if } \cunif\alpha> 1,\\
		\Ocal\big(Q\log^2 Q\big),&\text{if }\cunif\alpha = 1,\\
		\Ocal\big(Q^{\frac 1 2+\frac{1}{2\cunif\alpha}}\big),
   &\text{if }\cunif\alpha < 1.
	\end{cases}
\end{equation*}

For the convergence factor \eqref{eq:ConvFactorMI} of the multiindex estimator,
we first note that \(|\{k+k'=j:k,k'\in\Nbb\}|=j+1\) is the number of weak
integer compositions of \(j\). Therefore, the convergence factor of the
multiindex estimator simplifies towards
\[
\sigma_{\textnormal{MI}}(J,t^\alpha,\{Q_j\}_j)=
2^{-\alpha\cunif J+1}\cdiam^\alpha
+4^{\alpha\cunif+1}\cdiam^{2\alpha}\sum_{j=0}^J
\frac{(j+1)2^{-\alpha\cunif j}}{\sqrt{Q_{J-j}}}.
\]

\begin{remark}
Instead of using the cost bound for the multiindex estimator from \cite{Gil15}, 
we may in our particular setting just use a similar cost bound as for the
multilevel estimator by noticing that there
exists for every \(\varepsilon>0\) a \(j_0\) such that 
\((j+1)\leq 2^{\varepsilon j}\) for all \(j\geq j_0\). In order to achieve
an overall root mean error of \(1/\sqrt{Q}\), we obtain the cost bound 
\begin{equation*}
C_{\textnormal{MIMC}}=\begin{cases} \Ocal(Q),&\text{if }
  \cunif\alpha-\varepsilon> 1,\\
\Ocal\big(Q\log^2 Q\big),&\text{if }\cunif\alpha-\varepsilon = 1,\\
\Ocal\big(Q^{\frac 1 2+\frac{1}{2(\cunif\alpha-\varepsilon)}}\big),
                         &\text{if }\cunif\alpha-\varepsilon < 1.
\end{cases}
\end{equation*}
\end{remark}

Replacing \(Y\) by its discrete version \({\bs y}\) and employing the 
representation of the convergence factors for the H\"older case
immediately yields the following variant of \Cref{cor:datamlmc}.
\begin{corollary}\label{cor:covarianceapprox}
Let \({\bs y}\in L^2\big(\Omega;C^{0,\alpha}(X_N;\calY)\big)\).
Then, there exists a constant \(c_{\textnormal{ML}}>0\) such that
\[
  \big\|\Ebb[{\bs y}]-E^{\textnormal{ML}}[{\bs y}]]
  \big\|_{L^2(\Omega;L^2(X_N;\Ycal))}
\leq c_{\textnormal{ML}}
\bigg(\sum_{j=0}^J\frac{2^{-\cunif j\alpha}}{\sqrt{Q_{J-j}}}\bigg)
\|{\bs y}\|_{L^2(\Omega;C^{0,\alpha}(X_N;\Ycal))}.
\]
Similarly, there exists a constant \(c_{\textnormal{MI}}>0\) such that
\begin{equation*}
\begin{aligned}
&\big\|\Ebb[{\bs y}{\bs y}^\intercal]
-E^{\textnormal{MI}}[{\bs y}{\bs y}^\intercal]
\big\|_{L^2(\Omega;L^2(X_N;\Ycal)\otimes L^2(X_N;\Ycal))}\\
&\qquad\leq c_{\textnormal{MI}}\sum_{j=0}^J
\frac{(j+1)2^{-\cunif j\alpha}}{\sqrt{Q_{J-j}}}
\|{\bs y}{\bs y}^\intercal\|_{L^2(\Omega;C^{0,\alpha}(X_N;\Ycal)
\otimes C^{0,\alpha}(X_N;\Ycal))}.
\end{aligned}
\end{equation*}
respectively.
\end{corollary}

We note that similar multilevel estimates as in
\Cref{thm:MLMCCor} and \Cref{cor:covarianceapprox} have been shown for an
$\calH^2$-matrix based approach in \cite{Doe2023} and for a wavelet based
approach in \cite{HHKS2024}. Both approaches are able to approximate the second
moment with optimal cost, but have limited applicability. The $\calH^2$-matrix
approach is based on additional assumptions on the regularity of the second
moment, which are hard to verify for general site-to-value maps. The wavelet
based approach leverages the multilevel structure of wavelet representations
of site-to-value maps. However, as wavelet constructions are usually confined
to intervals or tensor products thereof, it remains to be investigated how the
approach can be applied to general site-to-value maps.

\section{Numerical examples}\label{sec:Numerics}
In this section, we verify the theoretical findings and give illustrative
examples, ranging from the computation of the discrete modulus of continuity for
different types of functions over the convergence of the discrete modulus of
continuity and the piecewise constant approximation of discrete data to the
multilevel Monte Carlo method for discrete data.

\subsection{Computation of the discrete modulus of continuity}%
\label{sec:discmodcomp}

\paragraph{Wiener process on \([0,1]\).}\label{sec:wiener1}
In this example, we consider the pathwise smoothness of the Wiener process,
that is, a Gaussian process with mean zero and covariance function
\(k(s,t)=\min\{s,t\}\) on \([0,1]\). It is well known that the paths of the
Wiener process are H\"older continuous for all exponents
\(\alpha<1/2\). Especially, their modulus of continuity is well approximated for
small \(t\) by
\begin{equation}\label{eq:wienermodcont}
\omega(f,t)
\approx \sqrt{2t\log(1/t)}.
\end{equation}
We compute the discrete modulus of continuity using
\Cref{alg:initialization,alg:evaluation}, respectively for 10 different
realizations of the Wiener process, where we set \(r=10^{-5},\ R=2,\ T=1\). For
the generation of the evaluation points of the Wiener process, we consider
independently drawn uniformly distributed random numbers in $[0,1]$. We keep
these time points fixed over all runs. We employ
\(N=10^6\) time points resulting in a fill distance of
\(h_{X_N}=6.00\cdot 10^{-5}\) and a separation distance of
\(q_{X_N}=5.00\cdot 10^{-12}\). To evaluate the discrete modulus of continuity
by means of \Cref{alg:evaluation}, we choose a quadratically graded mesh with
1000 points with the smallest value being \(10^{-5}\) and the largest
value being $1$. 

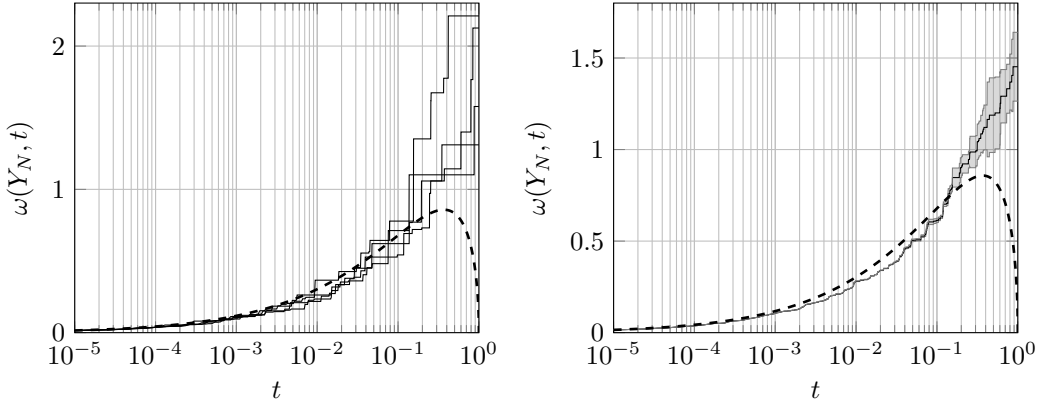
\begin{figure}[htb]
\centering
\begin{tikzpicture}
  \begin{axis}[
  width = 0.46\textwidth,
  xmin = 1e-5, xmax =1,
  ymin = 0, ymax = 2.3,
  xmode=log,
    cycle list={{black, mark=none}},
      xlabel={$t$},
      ylabel={$\omega(Y_N,t)$} ,
      grid = both 
    ]
\addplot table [x index=0, y index=1, mark=none,black, solid] {%
  ./Images/BM/eval_1D.txt};
  \addplot table [x index=0, y index=4, mark=none,black, solid] {%
  ./Images/BM/eval_1D.txt};
    \addplot table [x index=0, y index=7, mark=none,black, solid] {%
  ./Images/BM/eval_1D.txt};
    \addplot table [x index=0, y index=10, mark=none,black, solid] {%
  ./Images/BM/eval_1D.txt};
        \addplot+[
      no marks,
            line width = 1pt,
             dashed,
      black
    ] table[
      x index=0,
      y expr={(2*\thisrowno{0}*ln(1/\thisrowno{0}))^(0.5)}
    ] {./Images/BM/eval_1D.txt};
  \end{axis}
\end{tikzpicture}
\begin{tikzpicture}
  \begin{axis}[
  width = 0.46\textwidth,
  xmin = 1e-5, xmax =1,
  ymin = 0, ymax = 1.8,
  xmode=log,
      xlabel={$t$},
      ylabel={$\omega(Y_N,t)$} ,
      grid = both 
    ]
    \addplot+ [gray, mark=none, solid, name path=lower]
      table [x index=0, y index=1] {./Images/BM/mean_1D.txt};
    \addplot+ [black, mark=none, solid] table [x index=0, y index=2]
      {./Images/BM/mean_1D.txt};
    \addplot+ [gray, mark=none, solid, name path=upper]
      table [x index=0, y index=3] {./Images/BM/mean_1D.txt};
    \addplot+ [
      fill=gray, fill opacity=0.3,
    ] fill between[
      of=lower and upper,
    ];
        \addplot+[
      no marks,
            line width = 1pt,
             dashed,
      black
    ] table[
      x index=0,
      y expr={(2*\thisrowno{0}*ln(1/\thisrowno{0}))^(0.5)}
    ] {./Images/BM/eval_1D.txt};
  \end{axis}
\end{tikzpicture}

\caption{\label{fig:DMOC1D}Samples of the discrete modulus of continuity 
of the Wiener process (left).
Average modulus of continuity and one standard deviation envelope (right).
Dashed lines indicate the approximation \eqref{eq:wienermodcont}
for small $t$.}
\end{figure}

The left panel of \Cref{fig:DMOC1D} shows four realizations of the resulting
discrete moduli of continuity together with the approximate analytical rate from
\eqref{eq:wienermodcont}, while the right panel shows the expected modulus
of continuity averaged over the ten runs, where the gray shaded area
corresponds to one standard deviation. As can be seen and as expected, the
obtained trajectories as well as the mean closely follow the graph of
\eqref{eq:wienermodcont} for small values of \(t\).

\paragraph{Continuous but not H\"older continuous function on $\bbS^2$.}
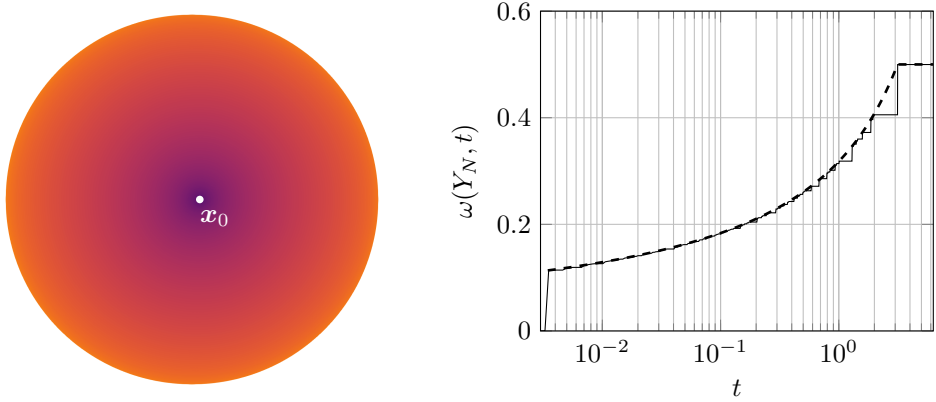
\begin{figure}[htb]
\centering
\begin{minipage}{0.4\textwidth}
	\begin{tikzpicture}
		\draw (0,0) node {%
        \includegraphics[width=0.9\textwidth,clip,trim=1600 500 1600 500]{%
				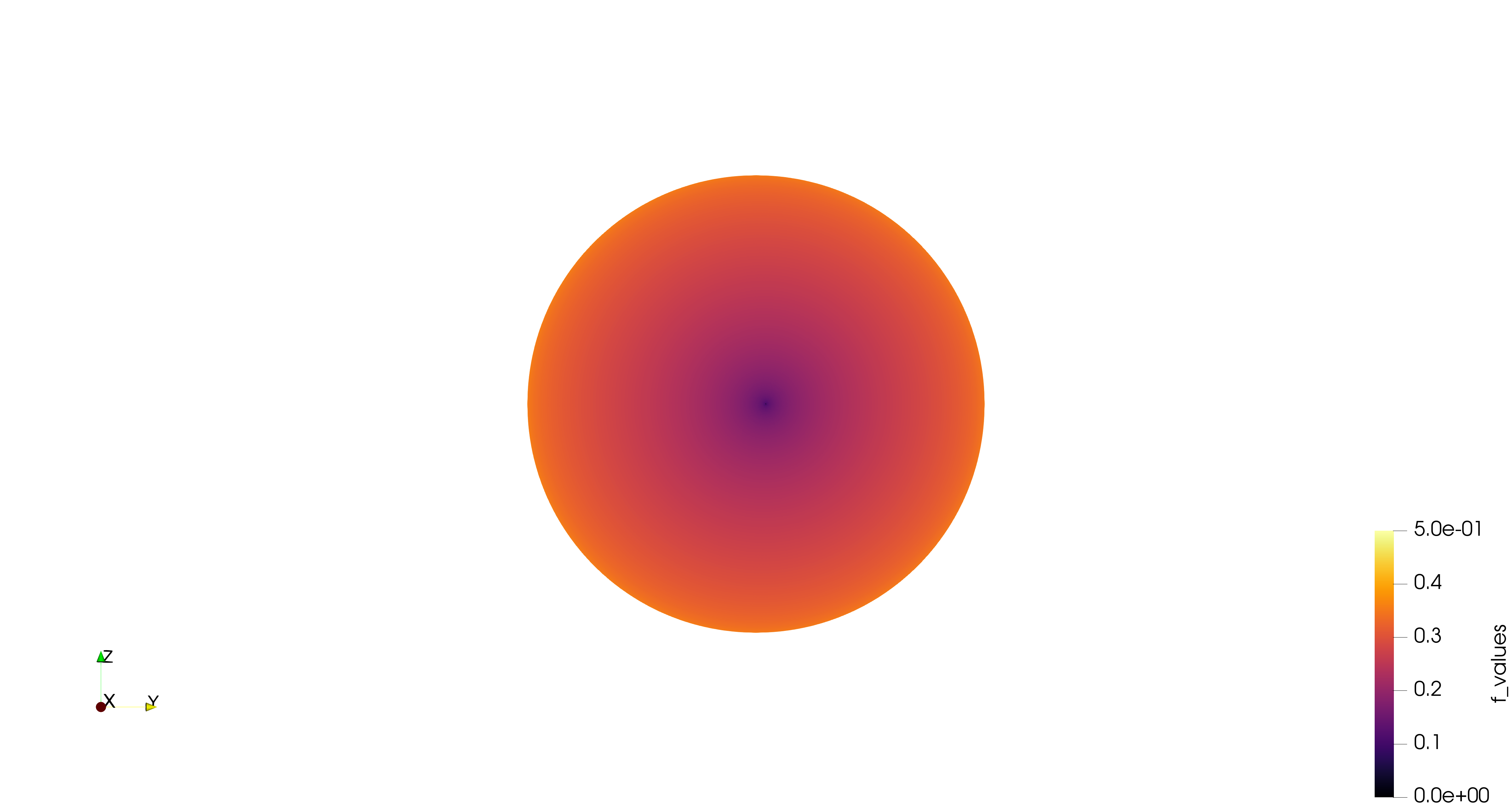}};
		\fill[white] (0.1,0) circle (0.05);
		\draw[white] (0.3,-0.25) node {$\bfx_0$};
	\end{tikzpicture}
\end{minipage}
\begin{minipage}{0.5\textwidth}
\begin{tikzpicture}
\begin{axis}[
	width = 0.9\textwidth,
	xmin =3e-3, xmax = 6.3,
	ymin = 0, ymax = 0.6,
xmode=log,
	xlabel={$t$},
	ylabel={$\omega(Y_N,t)$} ,
	grid = both 
	]
	\addplot+ [black, mark=none, solid] table [x index=0, y index=1] {%
    ./Images/Sphere_log/eval_sphere_log.txt};
	\addplot+[
	no marks,
	line width = 1pt,
	dashed,
	black
	] table[
	x index=0,
	y expr={1 * (abs(1/(ln(\thisrowno{0}/3.141)-2)) * (\thisrowno{0} < 3.141)
  + 0.5 *  (\thisrowno{0} >= 3.141))}
	] {./Images/Sphere_log/eval_sphere_log.txt};
\end{axis}
\end{tikzpicture}
\end{minipage}
\caption{\label{fig:logSphere}Continuous but not H\"older continuous function
from \Cref{eq:nothoelder} on $\bbS^2$ (left) and corresponding discrete modulus
of continuity (right). The dashed line indicates the exact modulus of
continuity.}
\end{figure}
The next example is dedicated to showing that the modulus of continuity can
quantify smoothness which is weaker than H\"older continuity. As an example
of a continuous but not H\"older continuous function, we consider the function
$f\colon\bbS^2\to\bbR$ defined by
\begin{equation}\label{eq:nothoelder}
f({\bs x}) =
\begin{cases}
\displaystyle\bigg|\log\frac{d_{\bbS^2}({\bs x},{\bs x}_0)}{\pi}-2\bigg|^{-1},
& {\bs x} \neq{\bs x}_0, \\
	0, & {\bs x} = {\bs x}_0,
\end{cases}
\end{equation}
where ${\bs x}_0\in\bbS^2$ and
\(d_{\Sbb^2}({\bs x},{\bs y})\) denotes the geodesic distance.
A visualization of this function is shown on the left hand side of
\Cref{fig:logSphere} for the choice
\({\bs x}_0=[1.00,3.53\cdot 10^{-2},1.00\cdot 10^{-6}]^\intercal\). 
The function is evaluated at a Fibonacci lattice, see \cite{MBRS+13}, with
\(N=10^6\) points. The fill distance is \(h_{X_N}\approx 3.55\cdot 10^{-3}\) 
and the separation distance \(q_{X_N}\approx3.01\cdot 10^{-3}\). Both
quantities have been approximated using Euclidean distances.
It is a short computation to check that the modulus of continuity of
\Cref{eq:nothoelder} is given by
\begin{equation}\label{eq:nothoeldermodcont}
\omega(f,t)
=
\begin{cases}
	0,& t=0,\\
  \displaystyle \bigg|\log\frac{t}{\pi}-2\bigg|^{-1}, & t\in (0,\pi),\\
	\displaystyle \frac{1}{2}, & t\geq \pi.
\end{cases}
\end{equation}

We approximate the modulus of continuity choosing \(r=10^{-3},\ R=2,\ T=2\pi\)
in \Cref{alg:initialization,alg:evaluation}. Within the algorithms, the
distances are chosen as the geodesic distances on the sphere. The resulting
approximation to the discrete modulus of continuity is found on the right
hand side of \Cref{fig:logSphere}. As can be seen, the approximation closely
follows the true behavior, where we notice that the discrete modulus of
continuity becomes zero for \(t<q_X\), as it is to be expected.

\paragraph{Real-world weather data}
Our final example concerning the estimation of the modulus of continuity
considers a time series of real-world data. We consider the temperature 
two meters above ground level at the airport Cologne/Bonn (CGN) in ${}^\circ$C.
The data is provided in 10 minute intervals by the DWD Climate Data Center (CDC)
for the period from 2 December 2008 to 11 September 2025, see \cite{DWD},
and consists of 881\,528 data points covering a time span of almost
$9\cdot 10^6$ minutes. We compute the modulus of continuity on the full data
set using \(r=10,\ R=2,\ T=10^6\) in 
\Cref{alg:initialization,alg:evaluation}. 
\begin{figure}[htb]
\centering
\begin{tikzpicture}
  \begin{axis}[
  width = 0.46\textwidth,
  xmin = 0, xmax = 9e6,
  ymin = -20, ymax = 40,
      xlabel={$t$ (min)},
      ylabel={$T\ ({}^\circ\text{C})$},
      grid = both 
    ]
    \addplot+ [black, mark=none, solid, name path=lower]
      table [x index=0, y index=1] {./Images/Tmp/temp_sub.txt};
  \end{axis}
\end{tikzpicture}
\begin{tikzpicture}
  \begin{axis}[
  width = 0.46\textwidth,
  xmin = 10, xmax =1e6,
  ymin = 0, ymax = 40,
  xmode=log,
      xlabel={$t$ (min)},
      ylabel={$\omega(Y_N,t)$} ,
      grid = both 
    ]
    \addplot+ [black, mark=none, solid, name path=lower]
      table [x index=0, y index=1] {./Images/Tmp/eval_1D_tmp.txt};
  \end{axis}
\end{tikzpicture}
\caption{\label{fig:weathermod}Time series of the real-world weather data (left)
and its discrete modulus of continuity (right).}
\end{figure}
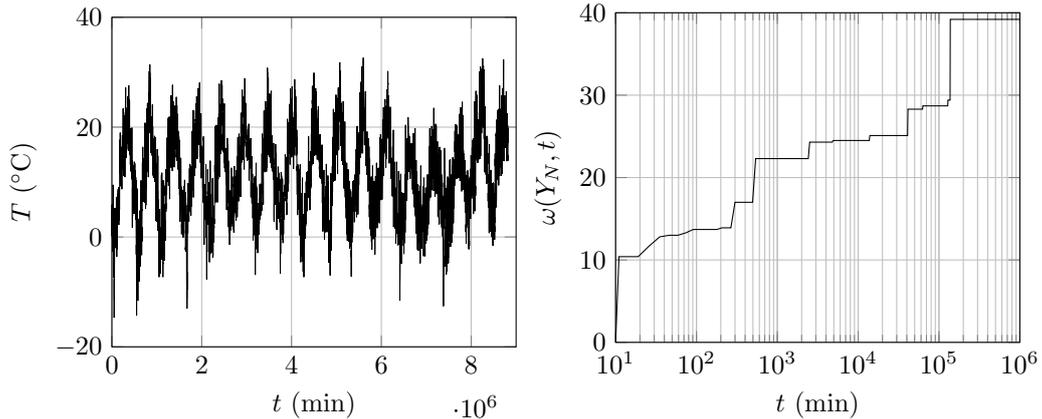
A daily subsample of the data set can be found on the
left-hand side of \Cref{fig:weathermod}. The separation
distance is \(q_{X_N}=10\), implying that the discrete modulus
of continuity is zero for $t<q_X$. The fill distance is \(h_{X_N} = 60\),
indicating that the dataset contains gaps with only one measurement recorded
per hour. The approximate discrete
modulus of continuity is shown on the right hand side of \Cref{fig:weathermod}.
The modulus of continuity indicates a rather irregular behavior already at the
smallest time interval of 10 minutes. Indeed, closer investigation of the data
reveals several sudden temperature drops of more than 10\textdegree C within a
10 minute span, which are responsible for this irregularity,
see \Cref{tab:weathertempdrop}. Thus, the modulus of continuity 
correctly captures the smoothness of the data.

\begin{table}
	\centering
	\begin{tabular}{|c|c|}
		\hline
		date and time (YYYY-MM-DD hh:mm:ss) & \textdegree C \\\hline
		$\vdots$ & $\vdots$ \\\hline
		2017-07-19 13:30:00	& 30.5 \\\hline
		\textbf{2017-07-19 13:40:00}	& \textbf{29.6} \\\hline
		\textbf{2017-07-19 13:50:00}	& \textbf{19.2} \\\hline
		2017-07-19 14:00:00 & 18.9 \\\hline
		$\vdots$ & $\vdots$ \\\hline
		2018-08-09 12:30:00	& 30 \\\hline
		\textbf{2018-08-09 12:40:00}	& \textbf{29.6} \\\hline
		\textbf{2018-08-09 12:50:00}	& \textbf{19.3} \\\hline
		2018-08-09 13:00:00	& 21.2 \\\hline
		$\vdots$ & $\vdots$ \\\hline
	\end{tabular}
	\caption{\label{tab:weathertempdrop}Excerpt from the real world weather data
  indicating the temperature drops responsible for the irregularity of the
temperature measurements as determined by the modulus of continuity 
in \Cref{fig:weathermod}.}
\end{table}

\subsection{Consistency for deterministic data sites}
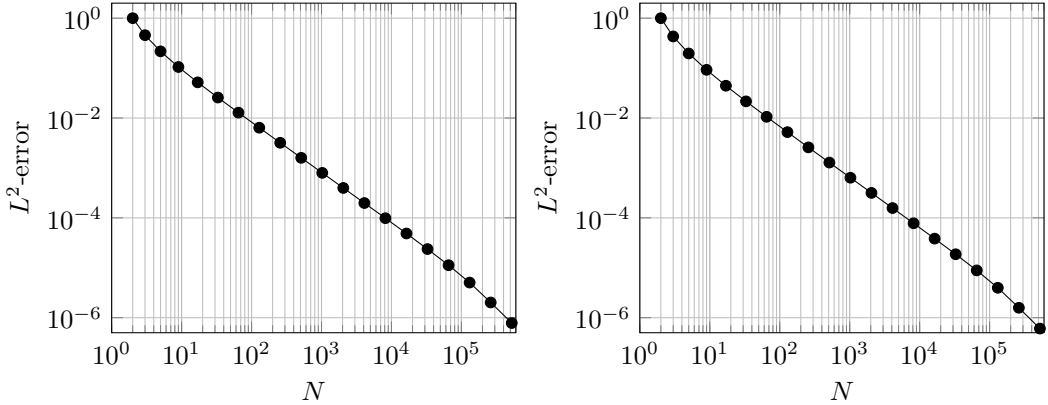
\begin{figure}[htb]
\centering
\begin{tikzpicture}
  \begin{axis}[
  width = 0.46\textwidth,
  xmin = 1, xmax =6e5,
  ymin = 5e-7, ymax = 2,
  ymode=log,
  xmode=log,
      xlabel={$N$},
      ylabel={$L^2$-error} ,
      grid = both 
    ]
    \addplot+ [black, solid, name path=lower,      
      mark options={draw=black,fill=black},] 
      table [x index=0, y index=1] {./Images/DMOC_Conv/conv_moc_311_sqrt.txt};
  \end{axis}
\end{tikzpicture}
\begin{tikzpicture}
  \begin{axis}[
  width = 0.46\textwidth,
  xmin = 1, xmax =6e5,
  ymin = 5e-7, ymax = 2,
    ymode=log,
      xmode=log,
      xlabel={$N$},
      ylabel={$L^2$-error} ,
      grid = both 
    ]
    \addplot+ [black, solid, name path=lower,      
      mark options={draw=black,fill=black},] 
      table [x index=0, y index=1] {./Images/DMOC_Conv/conv_moc_311_log.txt};
  \end{axis}
\end{tikzpicture}
\caption{\label{fig:Conv311}Convergence of the discrete modulus of continuity
  in \(L^2\) for \(N\to\infty\) using deterministic data sites.
The left panel shows the convergence for the square root function, while the
right panel shows the convergence for
the log function on the interval.}
\end{figure}

\paragraph{Consistency in $L^2$}
We verify the consistency estimates from \Cref{cor:detconsab} in $L^p$ for
$p=2$. To this end, we consider two test cases, namely the function
\(f(x)=\sqrt{x}\) on $[0,1]$ and the univariate version of the non-H\"older
continuous function from \eqref{eq:nothoelder}, that is,
\[
f(x) =
\begin{cases}
	\displaystyle |\log x-2|^{-1}, & x\neq 0, \\
	0, & x = 0,
\end{cases}
\]
also evaluated on \([0,1]\).
For the square root, the modulus of continuity is \(\omega(f,t)=t^{-1/2}\),
while it is 
\[
\omega(f,t)
=
\begin{cases}
	0,& t=0,\\
	\displaystyle |\log t-2|^{-1}, & t\in (0,1),\\
	\displaystyle \frac{1}{2}, & t\geq 1.
\end{cases}
\]
for the non-H\"older continuous function from \Cref{eq:nothoelder},
see also \Cref{eq:nothoeldermodcont}. To demonstrate the consistency, we
compute the modulus of continuity using a brute-force approach with runtime
\(\Ocal(N^2)\) and compare it to the true one. For the spatial discretization,
we employ uniform subdivisions of the unit interval with \(N=2^j+1\) points
for \(j=0,\ldots,19\). As the discrete moduli of continuity are nonsmooth, we
use a composite trapezoidal rule on \([0,1]\) with $10^{6}$ quadratically
graded points. The resulting quadrature error is around $2.5\cdot 10^{-13}$ for
both moduli of continuity.
We remark that such a high accuracy of the quadrature has been necessary, as
otherwise the convergence stalled.
The results are depicted in \Cref{fig:Conv311}. The left-hand side of the
figure shows the convergence for the modulus of continuity of the
square root, while the right-hand
side shows the convergence for the modulus of continuity of the
non-H\"older continuous function. 
As can be inferred from the figure,
in both cases, the rate of convergence is approximately linear.

\subsection{Consistency for empirical data sites}
To show consistency also for the case of empirical data sites,
we perform a similar experiment as in the previous paragraph. This time, we
randomize the locations of the data sites using independently drawn uniform
random points in $[0,1]$. All other parameters are kept as before.
We show the average error over 10 runs in \Cref{fig:Conv45}.
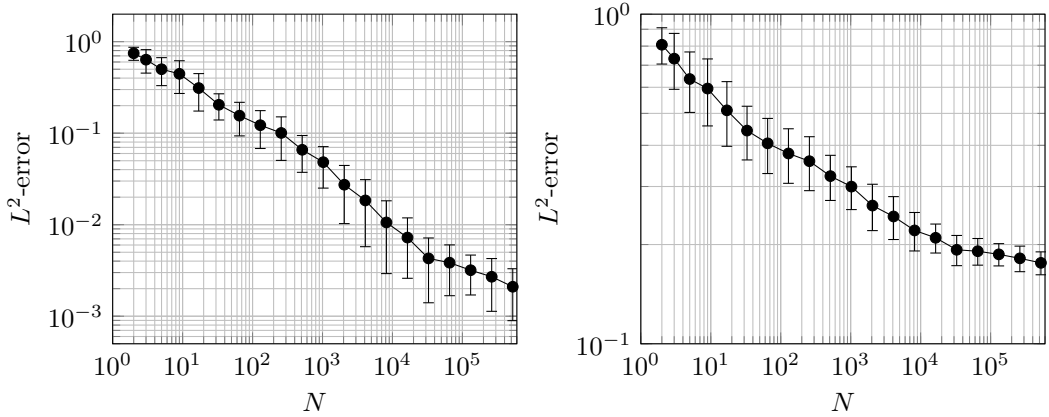
\begin{figure}[htb]
\centering
\begin{tikzpicture}
  \begin{axis}[
  width = 0.46\textwidth,
  xmin = 1, xmax =6e5,
  ymin = 5e-4, ymax = 2,
  ymode=log,
  xmode=log,
      xlabel={$N$},
      ylabel={$L^2$-error} ,
      grid = both,
          error bars/y dir=both,
    error bars/y explicit, 
    ]
    \addplot+ [
      black,
      mark options={draw=black,fill=black},
      error bars/.cd,
      y explicit,
    ] table [
      x index=0, %
      y index=11, %
      y error index=12 %
    ] {./Images/DMOC_Conv/conv_moc_45_sqrt.txt};
  \end{axis}
\end{tikzpicture}
\begin{tikzpicture}
  \begin{axis}[
  width = 0.46\textwidth,
  xmin = 1, xmax =6e5,
  ymin = 1e-1, ymax = 1,
    ymode=log,
      xmode=log,
      xlabel={$N$},
      ylabel={$L^2$-error} ,
      grid = both,
          error bars/y dir=both,
    error bars/y explicit, 
    ]
    \addplot+ [
      black,
      mark options={draw=black,fill=black},
      error bars/.cd,
      y explicit,
    ] table [
      x index=0, %
      y index=11, %
      y error index=12 %
    ] {./Images/DMOC_Conv/conv_moc_45_log.txt};
  \end{axis}
\end{tikzpicture}
\caption{\label{fig:Conv45}Convergence of the discrete modulus of continuity in
  \(L^2\) for \(N\to\infty\) using uniformly random data sites.
The left panel shows the convergence for the square root function, while the
right panel shows the convergence for the log function on the interval. The
error bars correspond to one standard deviation.}
\end{figure}
As we can see from the left panel, the rate of convergence is approximately
\(\Ocal(N^{-1/2})\) for the modulus of continuity of the square root,
while it is significantly reduced for the 
 modulus of continuity of the non-H\"older continuous function.

\subsection{Piecewise constant approximation of discrete data}
In this example, we verify the approximation rates for piecewise constant
approximation spaces defined directly at the data sites, compare
\Cref{sec:Approximation}. We focus on the approximation estimate
\Cref{eq:interapproxdiscmodcont} from \Cref{lem:pwconstantapproxstv} which
bounds the interpolation error on data sets in terms of the discrete modulus of
continuity. To this end, we use the data sets from \Cref{sec:discmodcomp} as 
ground truth on which we build coarser piecewise constant approximation spaces
by an $s$-$d$-tree based coarsening procedure. The interpolation points for the
constant approximation within the resulting clusters are chosen randomly.
If a cluster at a certain level is empty, we replace its value by that of the
parent. For the Wiener process, we choose \(N=10^6\) for the ground truth
and set \(r=10^{-5},\ R=2,\ T=1\) 
for the computation of the discrete modulus of continuity by
\Cref{alg:initialization,alg:evaluation}. As before, the time points
are sampled uniformly at random from the interval $[0,1]$ and remain fixed
across ten runs. Each run selects different interpolation points and samples a
new path of the Wiener process. 
\Cref{fig:ConvWproc} shows the average maximum
pointwise approximation error versus the mesh size \(h\), that is, the
maximum cluster diameter at each level.
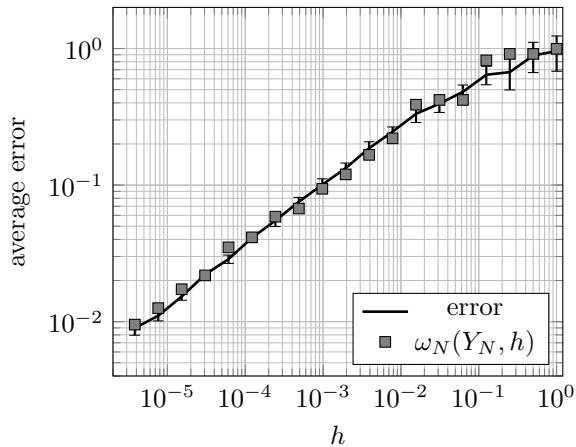
\begin{figure}[hbt]
\begin{center}
\begin{tikzpicture}
  \begin{axis}[
    width = 0.5\textwidth,
    xmin = 2e-6,
    xmax = 1.2,
    ymin = 4e-3,
    ymax = 2,
    xmode=log,
    ymode=log,
    xlabel={$h$},
    ylabel={average error},
    error bars/y dir=both,
    error bars/y explicit,
    grid=both,
    legend pos=south east
  ]
    \addplot+ [
      line width = 1pt,
      black,
      no marks,
      mark options={draw=black,fill=black},
      error bars/.cd,
      y explicit,
  error bar style={thick},
    ] table [
      x index=0, %
      y index=12, %
      y error index=13 %
    ] {./Images/Disc_Conv/discrete_err_wproc.txt};
    \addlegendentry{error}
        \addplot+[
      only marks,
      mark options={draw=black,fill=gray},
    ] table[
      x index=0,
      y index=1,
    ] {./Images/Disc_Conv/discrete_err_wproc.txt};
    \addlegendentry{$\omega_N(Y_N,h)$}
  \end{axis}
\end{tikzpicture}
\caption{\label{fig:ConvWproc}Average approximation error for the Wiener 
  process on \([0,1]\)
versus the mesh size \(h\).
The error bars show one standard deviation. The gray boxes indicate the 
corresponding discrete modulus of continuity, evaluated approximately using
\Cref{alg:initialization,alg:evaluation}.}
\end{center}
\end{figure} 
The error bars correspond to one standard deviation. The
boxes mark the approximate discrete modulus of continuity computed for the
ground truth. Especially, for
smaller mesh sizes, the interpolation error is bounded by the 
approximate discrete modulus of continuity, as expected. For larger
mesh sizes, the discrete modulus of continuity is evaluated at larger values
of $t$, where we need to expect a larger approximation error from
\Cref{alg:initialization,alg:evaluation}. Still, the values of the discrete
modulus of continuity are close to the true errors.

For the non-H\"older continuous function \eqref{eq:nothoelder}, we choose a
Fibonacci lattice with \(N=10^6\) points as ground truth
and set \(r=10^{-2},\ R=2,\ T=2\pi\)
for the computation of the discrete modulus of continuity by
\Cref{alg:initialization,alg:evaluation}. As before, 
the interpolation points are randomly sampled within each cluster.
\Cref{fig:ConvSphereLog} shows the average maximum
pointwise approximation error versus the mesh size \(h\), computed using
ten runs. Here, the mesh size is computed using the geodesic distance. 
\begin{figure}[htb]
\begin{center}
\begin{tikzpicture}
  \begin{axis}[
    width = 0.5\textwidth,
    xmin = 0,
    xmax = 7,
    ymin = 1e-1,
    ymax = 1,
    xmode=log,
    ymode=log,
    xlabel={$h$},
    ylabel={average error},
    error bars/y dir=both,
    error bars/y explicit,
    grid=both,
    legend pos=south east
  ]
    \addplot+ [
      black,
      no marks,
      mark options={draw=black,fill=black},
      error bars/.cd,
      y explicit,
  error bar style={thick},
    ] table [
      x index=0, %
      y index=12, %
      y error index=13 %
    ] {./Images/Disc_Conv/discrete_err_log.txt};
    \addlegendentry{error}
        \addplot+[
      only marks,
      mark options={draw=black,fill=gray},
    ] table[
      x index=0,
      y index=1
    ] {./Images/Disc_Conv/discrete_err_log.txt};
    \addlegendentry{$\omega_N(Y_N,h)$}
  \end{axis}
\end{tikzpicture}
\caption{\label{fig:ConvSphereLog}Average approximation error for the
  non-H\"older 
  continuous function 
\eqref{eq:nothoelder} on \(\Sbb^2\)
versus the mesh size \(h\).
The error bars show one standard deviation. The gray boxes indicate the 
corresponding discrete modulus of continuity, evaluated approximately using
\Cref{alg:initialization,alg:evaluation}.}
\end{center}
\end{figure}
In this example, the maximum interpolation error stays clearly below the
approximate discrete modulus of continuity computed for the ground truth.

Finally, we consider the entire time series data set with \(N=881\,528\) points
and set \(r=10,\ R=2,\ T=10^6\) for the computation of the discrete modulus of
continuity by \Cref{alg:initialization,alg:evaluation}. As in the
other two examples, the interpolation points are randomly sampled within each
cluster.
\Cref{fig:ConvTmp} shows the average maximum
pointwise approximation error versus the mesh size \(h\), computed using
ten runs.
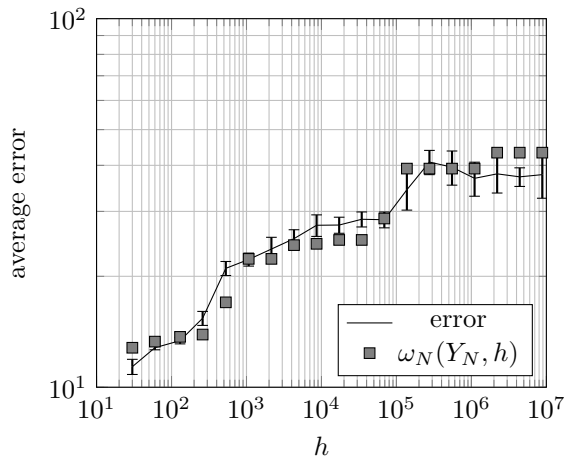
\begin{figure}[htb]
\begin{center}
\begin{tikzpicture}
  \begin{axis}[
    width = 0.5\textwidth,
    xmin = 10,
    xmax = 1e7,
    ymin = 10,
    ymax = 100,
    xmode=log,
    ymode=log,
    ytick={10,100},
    xlabel={$h$},
    ylabel={average error},
    error bars/y dir=both,
    error bars/y explicit,
    grid=both,
    legend pos=south east
  ]
    \addplot+ [
      black,
      no marks,
      mark options={draw=black,fill=black},
      error bars/.cd,
      y explicit,
  error bar style={thick},
    ] table [
      x index=0, %
      y index=12, %
      y error index=13 %
    ] {./Images/Disc_Conv/discrete_err_tmp.txt};
    \addlegendentry{error}
        \addplot+[
      only marks,
            mark options={draw=black,fill=gray},
    ] table[
      x index=0,
      y index=1
    ] {./Images/Disc_Conv/discrete_err_tmp.txt};
    \addlegendentry{$\omega_N(Y_N,h)$}
  \end{axis}
\end{tikzpicture}
\caption{\label{fig:ConvTmp}Average approximation error for the time series data
versus the mesh size \(h\).
The error bars show one standard deviation. The gray boxes indicate the 
corresponding approximate discrete modulus of continuity, evaluated approximately
using \Cref{alg:initialization,alg:evaluation}.}
\end{center}
\end{figure}
Here, for the large and the small values of the mesh size,
the discrete modulus of continuity bounds the approximation error, while
it is slightly smaller than the approximation error for intermediate mesh sizes.
We attribute this again to the approximation error of
\Cref{alg:initialization,alg:evaluation}.

\subsection{Multilevel Monte Carlo}
\paragraph{Wiener process}
In this example, we consider the multilevel Monte Carlo approximation of the
mean of the Wiener process on \([0,1]\), sampled at random time points as
in \Cref{sec:wiener1}. Four different realizations of the
Wiener process are depicted in \Cref{fig:BM}.

\begin{figure}[htb]
\centering
\begin{tikzpicture}
  \begin{axis}[
  width = 0.5\textwidth,
  xmin = 0, xmax =1,
  ymin = -1, ymax = 1.35,
    cycle list={{black, mark=none}},
      xlabel={$t$},
      ylabel={$W_t$}  
    ]
\addplot table [x index=0, y index=7, mark=none,black, solid] {%
  ./Images/BM/combined_sss.txt};
\addplot table [x index=0, y index=3, mark=none,black, solid] {%
./Images/BM/combined_sss.txt};
\addplot table [x index=0, y index=8, mark=none,black, solid] {%
./Images/BM/combined_sss.txt};
\addplot table [x index=0, y index=10, mark=none,black, solid] {%
./Images/BM/combined_sss.txt};
  \end{axis}
\end{tikzpicture}
\caption{\label{fig:BM}Different realizations of the Wiener process on $[0,1]$.}
\end{figure}
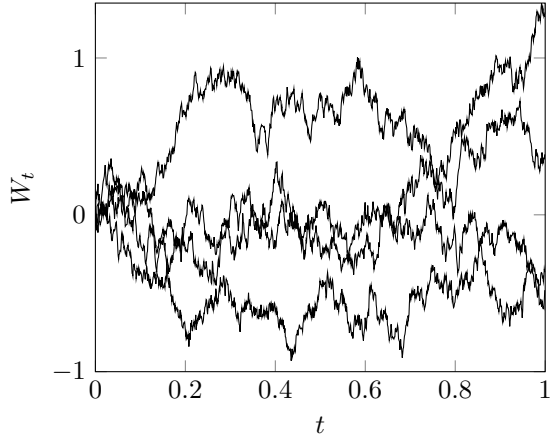

As noted before, the paths of the Wiener process are H\"older continuous
for all exponents \(\alpha<1/2\).
Hence, in view of \Cref{cor:HoelderApprox}, based on a balanced binary tree, the
variance at level \(\ell\) scales as \(V_\ell=\Ocal(2^{-\alpha\ell})\), while the
cost is linear with the level, that is, \(C_\ell=\Ocal(2^\ell)\). Therefore, the 
number of samples to achieve an overall root mean square error of \(1/\sqrt{N}\)
has to scale like \(\sqrt{V_\ell/C_\ell}\) on level \(\ell\), see \cite{Gil15}. 
Consequently, choosing \(N_0=N\) samples at the coarsest level, the number of
samples has to decrease by a factor of 
\(N_\ell/N_{\ell-1}=\sqrt{V_\ell/V_{\ell-1}}\cdot\sqrt{C_{\ell-1}/C_\ell}\).
In the current example this leads to 
\(N_\ell/N_{\ell-1}=2^{-(\alpha+1)/2}\approx 2^{-3/2}\).

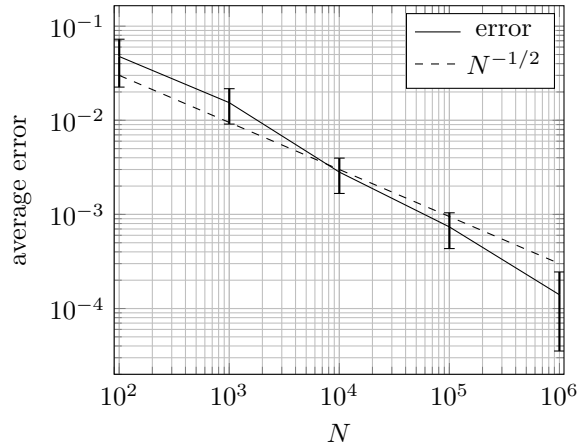
\begin{figure}[htb]
\begin{center}
\begin{tikzpicture}
  \begin{axis}[
    width = 0.5\textwidth,
    xmin = 90,
    xmax = 1100000,
    ymin = 2e-5,
    xmode=log,
    ymode=log,
    xlabel={$N$},
    ylabel={average error},
    error bars/y dir=both,
    error bars/y explicit,
    grid=both,
  ]
    \addplot+ [
      black,
      no marks,
      mark options={draw=black,fill=black},
      error bars/.cd,
      y explicit,
  error bar style={thick},
    ] table [
      x index=0, %
      y index=1, %
      y error index=2 %
    ] {./Images/BM/convergenceBM.txt};
    \addlegendentry{error}
        \addplot+[
      no marks,
             dashed,
      black
    ] table[
      x index=0,
      y expr={0.3*(\thisrowno{0})^(-0.5)}
    ] {./Images/BM/convergenceBM.txt};
    \addlegendentry{$N^{-1/2}$}
  \end{axis}
\end{tikzpicture}
\caption{\label{fig:RMSEBM}Average error of the multilevel Monte Carlo
  approximation of the mean of the Wiener process for different $N$ over ten
  different runs. The error bars show one standard deviation.}
\end{center}
\end{figure}

The convergence of the multilevel Monte Carlo approximation to the mean of the
Wiener process on \([0,1]\) is shown in \Cref{fig:RMSEBM}. The figure shows the
maximum error averaged over 10 runs. The error bars indicate one standard
deviation. Over each of these runs, the uniform random timepoints points for the
Wiener process are kept fixed, while the interpolation points used for the
multilevel Monte Carlo method are randomized in each leaf of the cluster tree.
Nested interpolation points are obtained by selecting the first non-empty
child and assigning its sample point to the parent cluster. This interpolation
procedure is compliant with the theoretical result in
\Cref{cor:HoelderApprox}. The theoretical rate of \(N^{-1/2}\) is clearly
achieved in this example.

\paragraph{Isotropic Gaussian random field on $\Sbb^2$.}
We consider an isotropic Gaussian field on the
unit sphere $\Sbb^2$ with mean zero and covariance function
\(k({\bs x},{\bs y})=e^{-4d_{\Sbb^s}({\bs x},{\bs y})}\), where
\(d_{\Sbb^2}({\bs x},{\bs y})\) denotes the geodesic distance. The field is
evaluated at a Fibonacci lattice. \Cref{fig:FibonacciSphere} shows the resulting
points for \(N=10^2,10^3,10^4,10^5\) together with the bounding boxes of
the leaves of the corresponding $s$-d-tree.
\begin{figure}[htb]
	\begin{center}
		\begin{tikzpicture}
			\draw(0,0) node{\includegraphics[scale=0.05,clip,trim=1350 550 1300 550]{%
        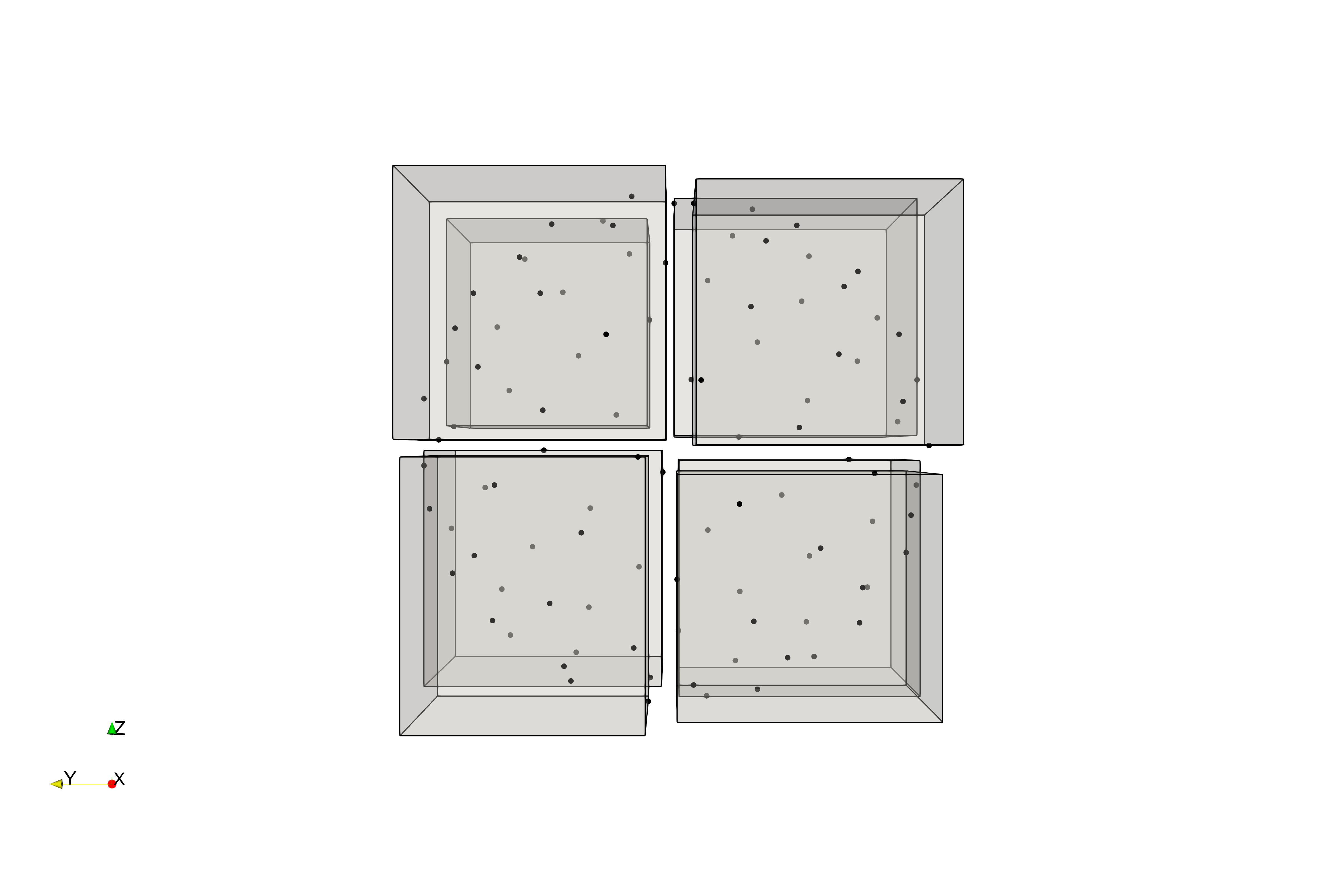}};
		\draw(3.6,0) node{\includegraphics[scale=0.05,clip,trim=1400 550 1400 550]{%
      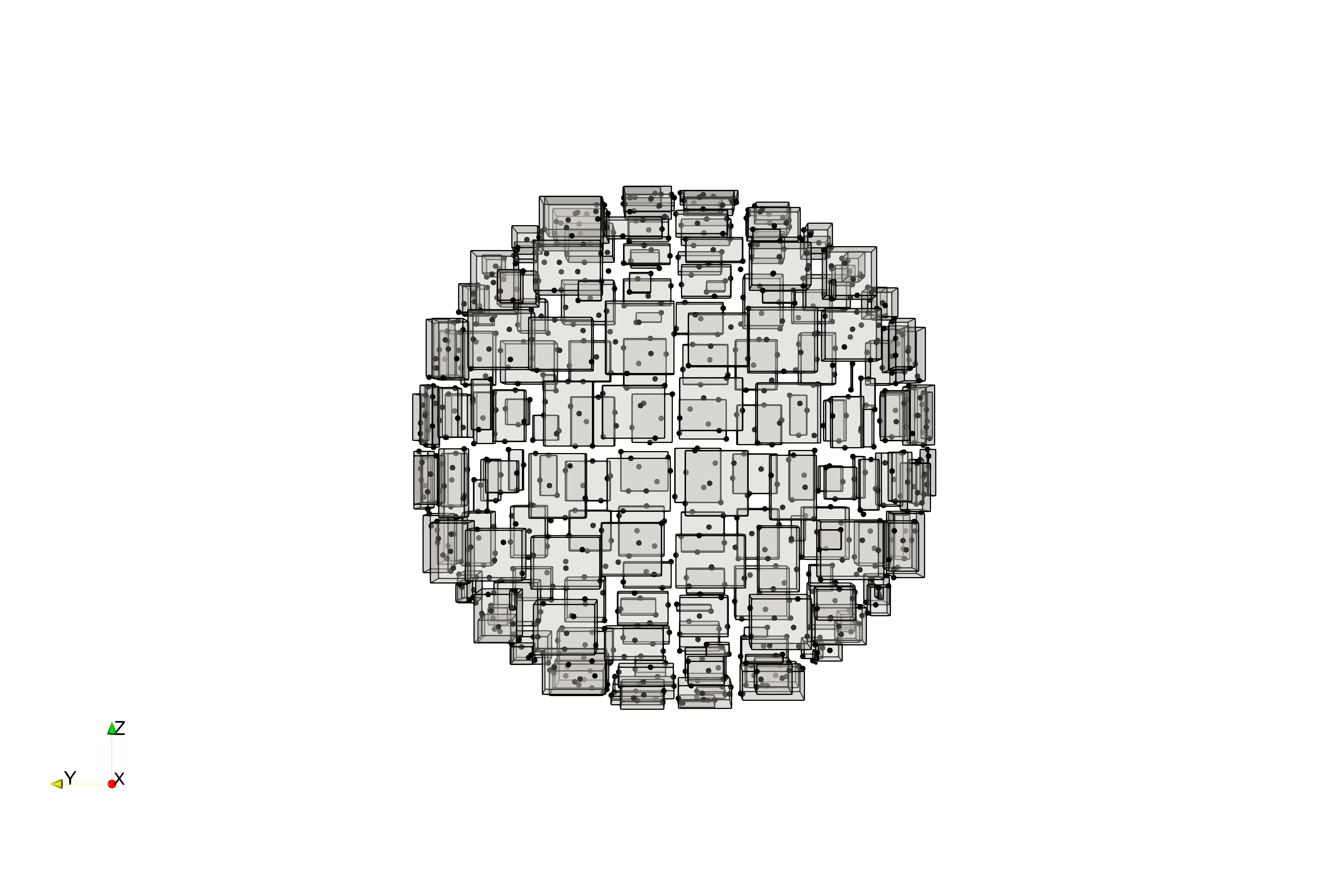}};
		\draw(7.2,0) node{\includegraphics[scale=0.05,clip,trim=1400 550 1400 550]{%
      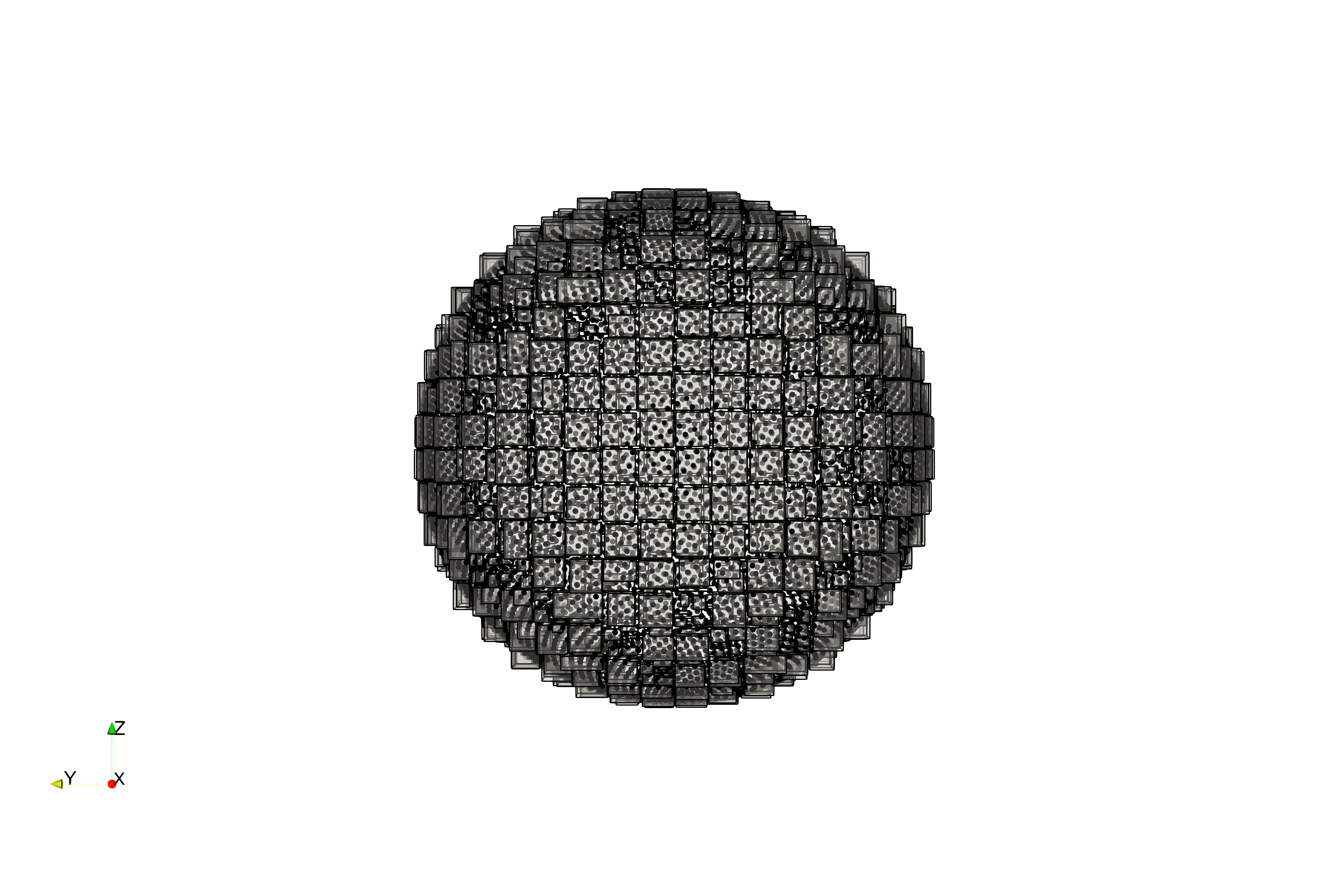}};
	\draw(10.8,0) node{\includegraphics[scale=0.05,clip,trim=1400 550 1400 550]{%
    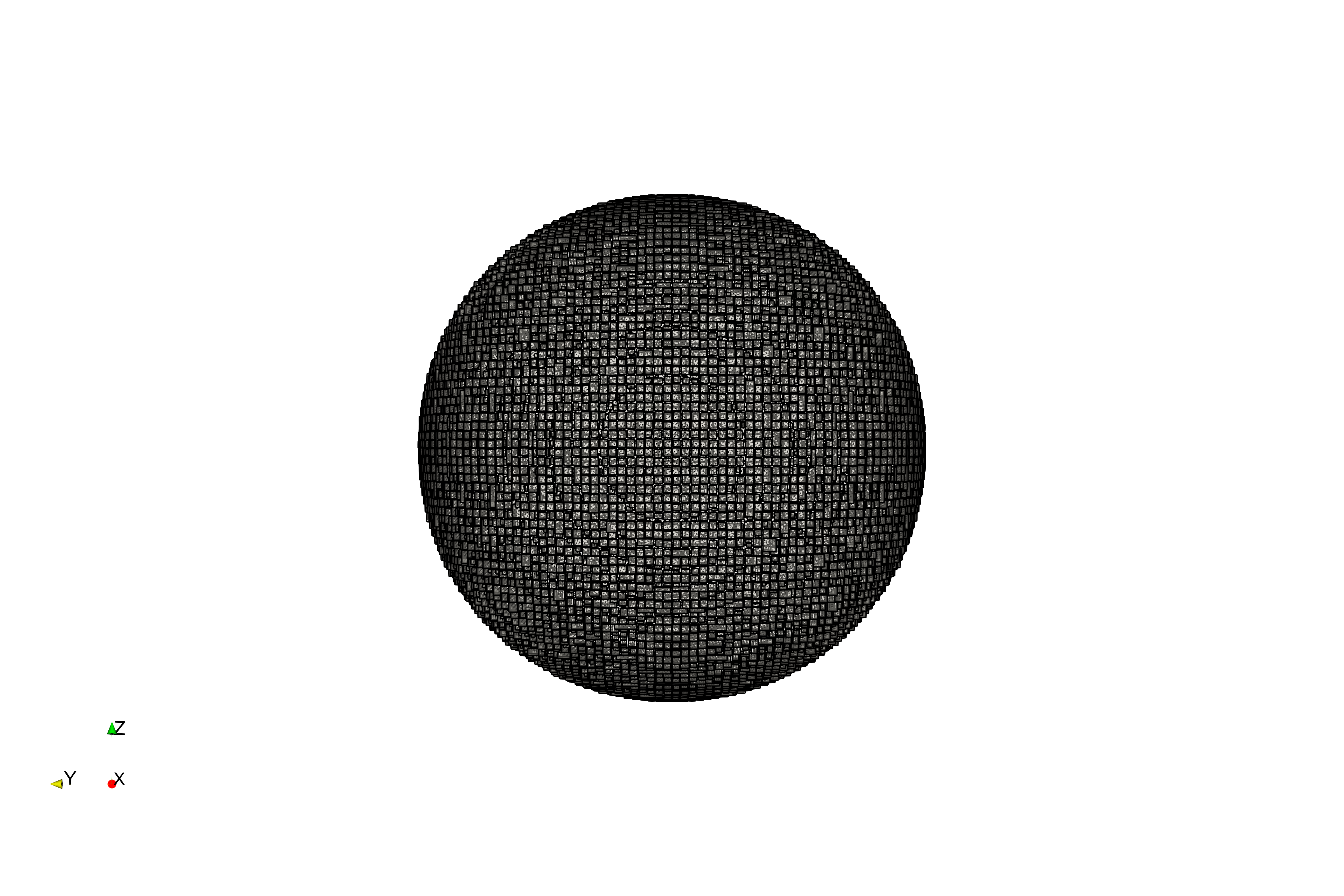}};
		\end{tikzpicture}
		\caption{\label{fig:FibonacciSphere}Fibonacci lattices for the unit sphere
      with $10^2,\ 10^3,\ 10^4,\ 10^5$
			points and bounding boxes of the leaves of the corresponding $s$-d-trees.}
	\end{center}
\end{figure}

Since the paths of the field are only H\"older continuous with coefficient
\(\alpha<1/2\), they cannot efficiently be approximated by a spectral approach.
Therefore, we perform the sampling using a Krylov subspace approximation of the
matrix square root. For \(N\leq 10^3\),
we directly use the covariance matrix, while for larger \(N\), we first perform
a samplet matrix compression thereof, see \cite{HM25} and the references
therein. We use samplets with three vanishing moments and the cutoff criterion
is set to \(\eta=10\). An a-posteriori compression is performed with a relative
threshold of \(10^{-4}\) in the Frobenius norm. For \(N=10^6\), the resulting
compression error is \(1.53\cdot 10^{-4}\) using a randomized estimator. The
compressed matrix contains 141 entries per row on average. For all \(N\), we
employ a Krylov subspace of dimension 10. Four different realizations
of the computed Gaussian field can be found in \Cref{fig:GPsphere}.
Particularly, we observe that the high frequency oscillations in the
realizations, known from the Wiener process, are captured by the taken approach.
\begin{figure}[htb]
	\begin{center}
		\begin{tikzpicture}
			\draw(0,0) node{\includegraphics[scale=0.05,clip,trim=1450 650 1450 650]{%
        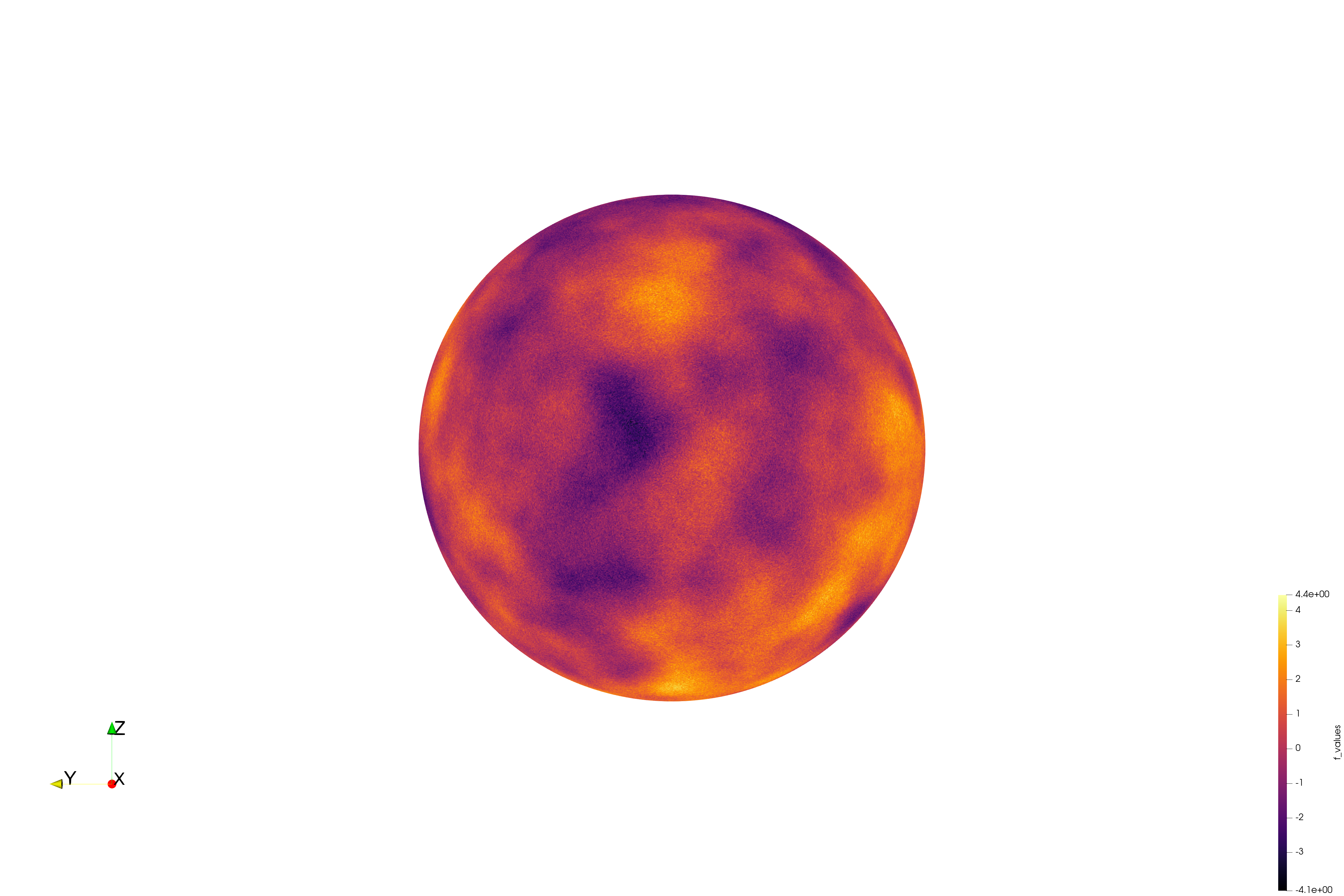}};
		\draw(3.6,0) node{\includegraphics[scale=0.05,clip,trim=1450 650 1450 650]{%
      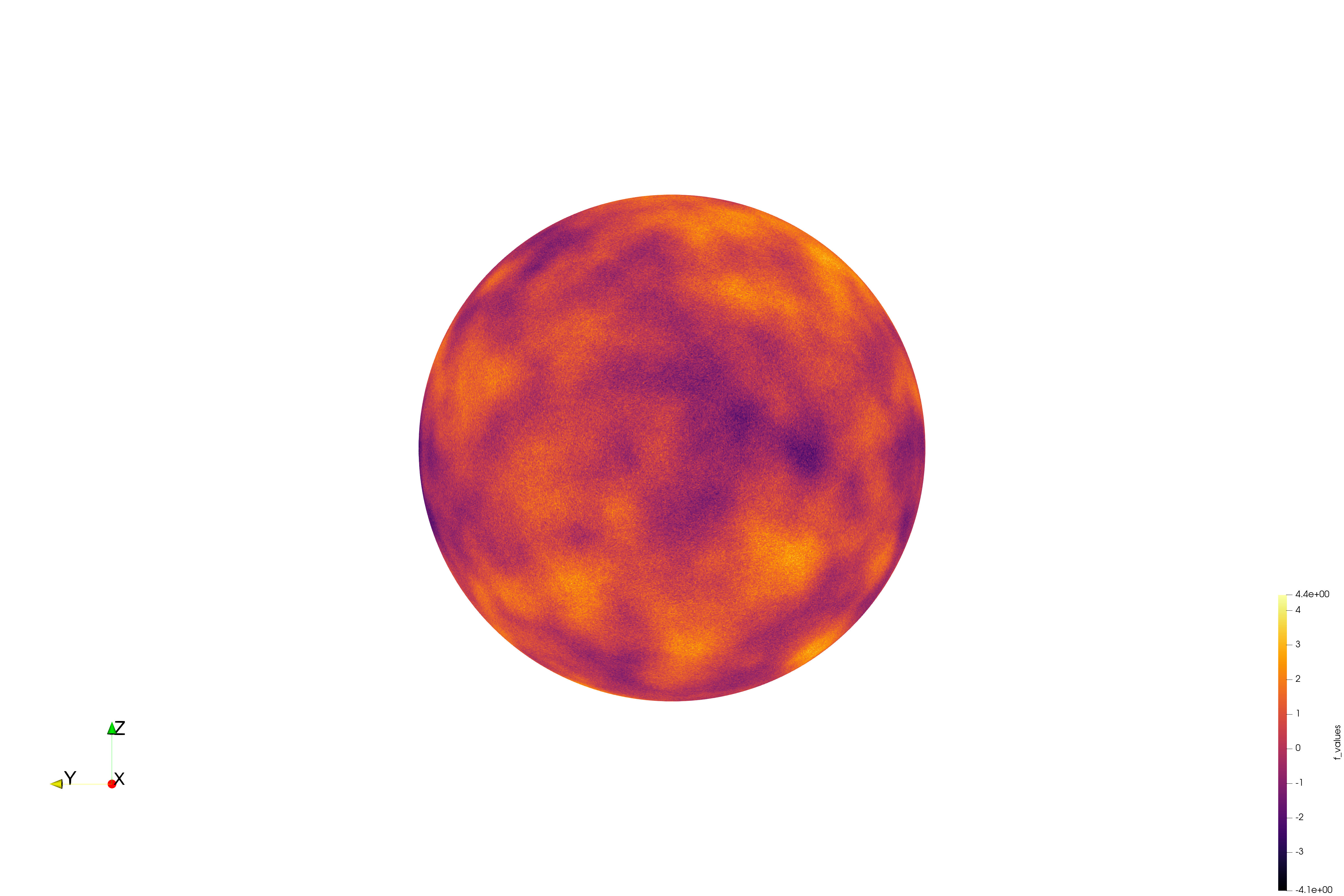}};
	\draw(7.2,0) node{\includegraphics[scale=0.05,clip,trim=1450 650 1450 650]{%
    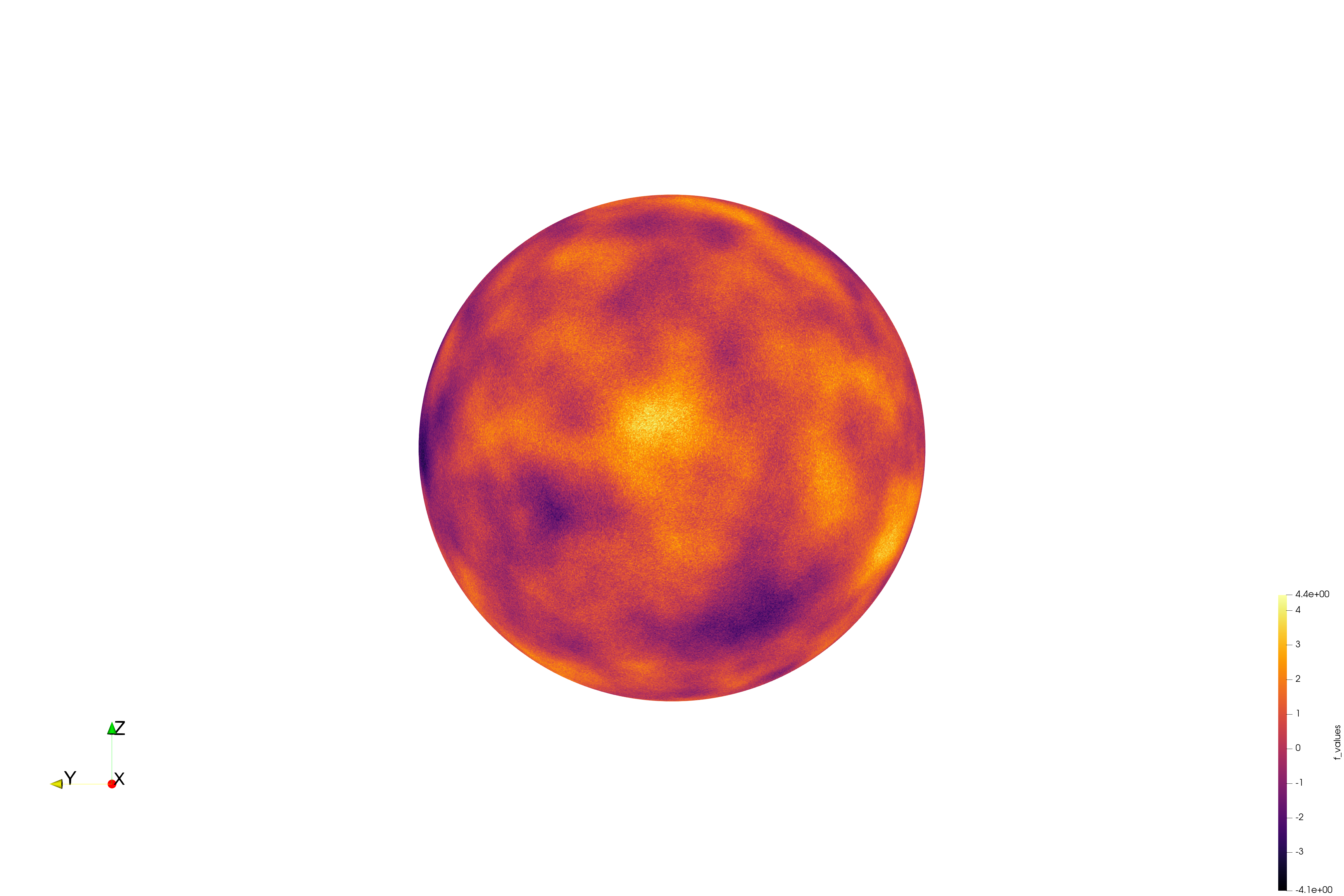}};
	\draw(10.8,0) node{\includegraphics[scale=0.05,clip,trim=1450 650 1450 650]{%
    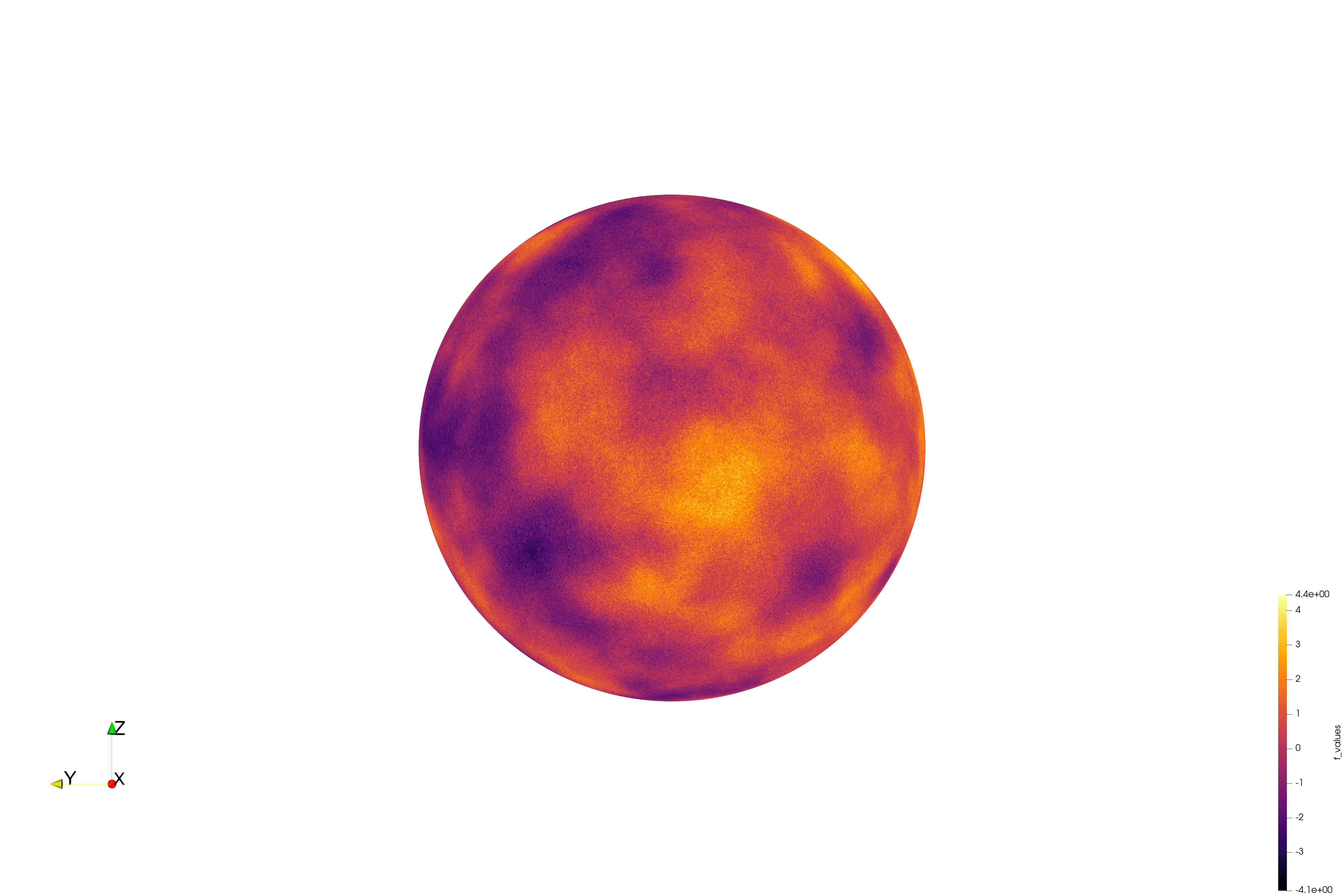}};
		\end{tikzpicture}
		\caption{\label{fig:GPsphere}Different realizations of an isotropic Gaussian
			random field on $\Sbb^2$ evaluated at a Fibonacci lattice with \(N=10^6\)
      points. All shown realizations assume values in the interval $[-4.1,4.4]$.
      Darker colors correspond to smaller values.}
	\end{center}
\end{figure}

In this example, we have \(V_\ell=\Ocal(2^{-\alpha\ell})\) as before, while the
cost is now \(C_\ell=\Ocal(2^{2\ell})\), since the sphere is a two dimensional
manifold and the \(n\)-d-tree becomes a quad-tree asymptotically. Therefore, we
decrease the number of samples in this example by the factor
\(N_\ell/N_{\ell-1}=2^{-(\alpha+2)/2}\approx 2^{-5/4}\). The convergence of the
multilevel Monte Carlo approximation to the mean of the Gaussian process
on \(\Sbb^2\) is shown in \Cref{fig:RMSES2}. 
\begin{figure}[htb]
\begin{center}
\begin{tikzpicture}
  \begin{axis}[
    width = 0.5\textwidth,
    xmin = 90,
    xmax = 1100000,
    ymin = 1e-3,
    xmode=log,
    ymode=log,
    xlabel={$N$},
    ylabel={average error},
    error bars/y dir=both,
    error bars/y explicit,
    grid=both,
  ]
    \addplot+ [
      black,
      no marks,
      mark options={draw=black,fill=black},
      error bars/.cd,
      y explicit,
  error bar style={thick},
    ] table [
      x index=0, %
      y index=1, %
      y error index=2 %
    ] {./Images/Sphere/convergenceS2_eiger.txt};
    \addlegendentry{error}
        \addplot+[
      no marks,
             dashed,
      black
    ] table[
      x index=0,
      y expr={2.3*(\thisrowno{0})^(-0.5)}
    ] {./Images/Sphere/convergenceS2_eiger.txt};
    \addlegendentry{$N^{-1/2}$}
  \end{axis}
\end{tikzpicture}
\caption{\label{fig:RMSES2}Average error of the multilevel Monte Carlo
  approximation of the mean of a Gaussian process on \(\Sbb^2\) for different
  $N$ over 10 different runs. The error bars show one standard deviation.}
\end{center}
\end{figure}
The figure shows the maximum error
averaged over 10 runs. The error bars indicate one standard deviation. Over each
of these runs, the Fibonacci lattice is kept fixed, while the interpolation
points used for the multilevel Monte Carlo method are randomized in each leaf of
the cluster tree. Nested interpolation points are again obtained by selecting
the first non-empty child and assigning its sample point to the parent cluster.
The theoretical rate of \(N^{-1/2}\) is roughly achieved in this example.

\section{Conclusion}\label{sec:Conclusion}
We have developed a data-centric framework for the approximation of
site-to-value maps. This framework covers both deterministic and
empirical data sets in metric spaces.
We have identified the discrete modulus of continuity as a natural and
computable intrinsic regularity measure for discrete data.
This makes the framework applicable to all site-to-value maps mapping
from a compact metric space to a metric space.
Our analysis shows the consistency of this regularity measure in the infinite
data limit for deterministic and empirical data sites. Motivated by these
findings, a data-intrinsic
approximation theory for site-to-value maps has been introduced. For random
site-to-value maps we have proposed multilevel approximation spaces and a
variant of the multilevel Monte Carlo method to compute statistical quantities
of interest. The obtained results can serve as a foundation for more intricate
data-driven applications. In the numerical studies, synthetic and
real-world data sets have been considered. The numerical results for the
consistency of the discrete modulus of continuity and the validation of the
discrete approximation error bounds corroborate the theoretical findings.
The multilevel Monte-Carlo approximation of the expectation of a spatial
low-regularity Gaussian random field on the sphere evaluated at a Fibonacci
lattice clearly demonstrates the practical relevance of data-centric approaches.

\section*{Acknowledgements}
JD received support from the DFG under Germany’s Excellence Strategy,
project 390685813. MM was funded by the SNSF starting grant 
``Multiresolution methods for unstructured data'' (TMSGI2\_211684).
\bibliographystyle{plain}
\bibliography{references}
\end{document}